\documentclass[11pt]{article}
\usepackage[title]{appendix}
\usepackage{subfiles}
\usepackage{graphicx} 
\usepackage{caption}
\usepackage{subcaption}
\graphicspath{ {./images/} }
\usepackage[utf8]{inputenc}
\usepackage{amsmath}
\usepackage{amssymb, amsthm, physics,leftidx, bbm}

\usepackage{yhmath}
\usepackage{mathrsfs}  

\usepackage[maxnames=99, doi=false,isbn=false,url=false,eprint=false]{biblatex}
\usepackage{tikz-cd}

\usepackage{xcolor} 

\usepackage{enumerate}
 \usepackage{soul}
\setstcolor{red}
\usepackage[normalem]{ulem}
\usepackage{dashrule}

\allowdisplaybreaks 
\usepackage[all]{xy}
\usepackage{mathtools}
\usepackage{stmaryrd}
\usepackage[shortlabels]{enumitem}

\usepackage{tikz-cd}
\usepackage[
 letterpaper, 
 margin = 1.5 in
 ]{geometry}

\usepackage[T1]{fontenc}

\usepackage{newtxtext,newtxmath}
  
 \theoremstyle{definition}

\newtheorem{thm}{Theorem} 

\newtheorem*{thmnonumber}{Theorem}
\newtheorem*{propnonumber}{Proposition}
\newtheorem{defprop}{Definition/Proposition}
\newtheorem{defn}{Definition}[section]	

\newtheorem{prop}{Proposition}[section]
\newtheorem{lem}{Lemma}[section]
\newtheorem{cor}{Corollary}[section]

\newtheoremstyle{dotlessS}{}{}{}{}{\color{blue}\bfseries}{}{ }{}
\theoremstyle{dotlessS}

\DeclareMathOperator{\fuk}{Fuk}

\DeclareMathOperator{\N}{\mathbb{N}}

\DeclareMathOperator{\C}{\mathbb{C}}

\DeclareMathOperator{\Span}{span}
 
\DeclareMathOperator{\Cl}{Cl}

\DeclareMathOperator{\Z}{\mathbb{Z}}

\DeclareMathOperator{\R}{\mathbb{R}}

\DeclareMathOperator{\orbit}{\mathcal{O}}

\DeclareMathOperator{\g}{\mathfrak{g}}

\DeclareMathOperator{\lam}{\lambda}

\newcommand{\upgeq}{\rotatebox{330}{$\geq$}} 
\newcommand{\downgeq}{\rotatebox{30}{$\geq$}}
\newcommand{\updots}{\rotatebox{330}{$\cdots$}}
\newcommand{\downdots}{\rotatebox{30}{$\cdots$}}

\DeclareMathOperator{\un}{\mathfrak{u}(n)}

\DeclareMathOperator{\m}{\mathfrak{m}}
\DeclareMathOperator{\Hom}{Hom}
\DeclareMathOperator{\End}{End}

\DeclareMathOperator{\Mq}{\mathcal{M}}

\DeclareMathOperator{\id}{Id}

\DeclareMathOperator{\qar}{\quad \Rightarrow \quad}

\DeclareMathOperator{\val}{val}

\DeclareMathOperator{\Ad}{Ad}

\DeclareMathOperator{\opposite}{op}
\DeclareMathOperator{\interior}{int}
\DeclareMathOperator{\Hess}{Hess}

\newcommand\greyt{\includegraphics[width=.07in]{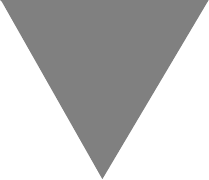}}

\newcommand\blackt{{\includegraphics[width=.07in]{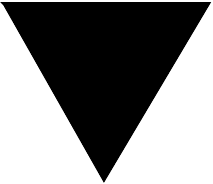}}}

\newcommand\whitet{{\includegraphics[width=.07in]{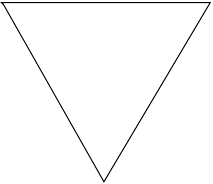}}}

\newcommand{\proj}{\mathbb{CP}}

\newcommand{\pair}[1]{\left\langle #1 \right\rangle }

\newcommand{\set}[2]{\left\{ #1 \, \middle|\, #2\right\}}
\newcommand{\lrp}[1]{\left(#1\right)}

\newcommand{\littletaller}{\mathchoice{\vphantom{\big|}}{}{}{}}
\newcommand\restr[2]{{% we make the whole thing an ordinary symbol
  \left.\kern-\nulldelimiterspace % automatically resize the bar with \right
  #1 % the function
  \littletaller % pretend it's a little taller at normal size
  \right|_{#2} % this is the delimiter
  }}

\usepackage{imakeidx}
\makeindex[intoc]
\usepackage[hidelinks]{hyperref}
\def\equationautorefname~#1\null{(#1)\null}

\usepackage{xurl}
\hypersetup{breaklinks=true}

\usepackage{lipsum}

\title{Lagrangian Floer theory on Woodward's multiplicity-free $U(2)$-manifolds}

\author{Yao Xiao}
\date{\today}

\begin{document}
\maketitle

\newcommand{\Addresses}{{ 
  \bigskip
  \footnotesize
\noindent\textsc{110 Frelinghuysen Road,
Piscataway, NJ, USA 08854}\par\nopagebreak
 \noindent \textit{E-mail address}: \texttt{yx489@math.rutgers.edu}
}}
\begin{abstract}
In this paper, we study a family of symplectic manifolds introduced by Woodward. These manifolds belong to the broader class of \emph{multiplicity-free} Hamiltonian $G$-manifolds, a generalization of toric manifolds for non-abelian Hamiltonian group actions. Prominent examples of multiplicity-free spaces include coadjoint orbits of $U(n)$ and $SO(n)$ equipped with multiplicity-free $U(n-1)$- and $SO(n-1)$-actions, respectively.  

Although these multiplicity-free $U(2)$-manifolds are not toric, we may study a family of Lagrangian tori by performing a symplectic cut that allows us to apply the toric Lagrangian Floer theory. In particular, we employ Venugopalan--Woodward's study of pseudoholomorphic curves under symplectic cuts to obtain the leading order potential. This allows us to identify a number of critical points of the potential function which correspond to a non-displaceable Lagrangian submanifold. 
Moreover, we adapt McDuff's probe method to show that the majority of the other Lagrangian submanifolds in these spaces are displaceable. 
Finally, we prove that the open-closed map for the Fukaya subcategory generated by these branes is an isomorphism. It follows that they satisfy Abouzaid--Fukaya--Oh--Ohta--Ono's generation criterion.
\end{abstract}

\tableofcontents
\section{Introduction} 

The work of Fukaya--Oh--Ohta--Ono and Cho--Oh on Lagrangian Floer theory has provided foundational insights into homological mirror symmetry, particularly in the context of toric manifolds, which have been studied extensively  \cite{CO, FOOOI, FOOOII, FOOOIII, fooobook1, FOOObook2, MR4800139}.

However, many Hamiltonian $G$-manifolds are not toric. A particularly interesting class is given by
multiplicity-free Hamiltonian manifolds, which generalize toric manifolds while exhibiting
richer geometric structures.

The examples constructed by Woodward \cite{WoodwardExample} provide important test cases in this direction. These manifolds are obtained from $U(3)$-coadjoint orbits via $U(2)$-equivariant symplectic cuts and admit Hamiltonian $U(2)$-actions. 
They  closely resemble the example
constructed by Tolman in \cite{Tolman}, in the sense that they have the same X-ray, an invariant associated with complexity-one Hamiltonian torus actions that is used to show that these manifolds do not admit compatible invariant K\"{a}hler structures.  The existence of such manifolds demonstrates that equivariant symplectic geometry is richer than equivariant K\"{a}hler geometry.
   
The goal of this paper is to study Woodward's multiplicity-free $U(2)$-manifolds via Lagrangian Floer theory. 
A key observation is that auxiliary symplectic cuts (different from the cuts that produced Woodward's examples) break these manifolds into tractable pieces, making them fit into the framework of split $A_{\infty}$ algebras by Venugopalan--Woodward \cite{VWSplit, VW}, which is built on \cite{CharestWoodward}. Accordingly, we adopt their model of $A_{\infty}$ algebras in place of the model used in \cite{xiao2023equivariantlagrangianfloertheory, xiao2024moment} or the aforementioned works of Fukaya--Oh--Ohta--Ono. This allows us to identify non-displaceable Lagrangian torus fibers by studying critical points of the potential function. 
Furthermore, we show that the resulting collection of Lagrangian branes split-generates the Fukaya category.

The study of Lagrangian Floer theory on certain multiplicity-free manifolds in the form of flag varieties has also been pursued in the works of Nishinou--Nohara--Ueda \cite{NNU}, Nohara--Ueda \cite{NU, NoharaUedaNonTorusFibers}, and Cho--Kim--Oh \cite{ChoKimOhCrit, ChoKimOh} among others. 
Most of these works study the potential functions of Gelfand--Cetlin fibers via toric degeneration, under the assumption that the symplectic manifold is Fano. 
In contrast, we apply the split $A_{\infty}$ algebras, constructed using symplectic field-theoretic techniques, developed by Charest--Woodward \cite{CharestWoodward} and Venugopalan--Woodward \cite{VW, VWSplit}, to obtain a leading order potential function without Fano assumptions.  

\medskip 

In the following, we denote a multiplicity-free $U(2)$-manifold constructed by Woodward in \cite{WoodwardExample} by $M$. (See Figure \ref{fig:polytopes of M}.)
\begin{figure}
    \centering
\resizebox{0.5\textwidth}{!}{ 

\tikzset{every picture/.style={line width=0.75pt}} %set default line width to 0.75pt        

\begin{tikzpicture}[x=0.75pt,y=0.75pt,yscale=-1,xscale=1]
%uncomment if require: \path (0,190); %set diagram left start at 0, and has height of 190

%Straight Lines [id:da6913648583729183] 
\draw    (141,86) -- (141.62,176) ;
%Straight Lines [id:da31135744064187376] 
\draw    (141.62,176) -- (230.83,86) ;
%Straight Lines [id:da4275024319638867] 
\draw    (185.38,13.2) -- (141,86) ;
%Straight Lines [id:da6612297875239461] 
\draw    (230.71,86) -- (252.3,77) ;
%Straight Lines [id:da7403957044951885] 
\draw    (141.11,86) -- (230.83,86) ;
\draw [shift={(230.83,86)}, rotate = 0] [color={rgb, 255:red, 0; green, 0; blue, 0 }  ][fill={rgb, 255:red, 0; green, 0; blue, 0 }  ][line width=0.75]      (0, 0) circle [x radius= 2.01, y radius= 2.01]   ;
%Straight Lines [id:da5691546872850083] 
\draw [color={rgb, 255:red, 208; green, 2; blue, 27 }  ,draw opacity=1 ][line width=1.5]    (220.36,231) -- cycle ;
%Straight Lines [id:da0007977295090592085] 
\draw    (185.38,13.2) -- (252.13,52.27) ;
%Straight Lines [id:da11943867514982742] 
\draw    (252.3,52.27) -- (252.3,77) ;
%Straight Lines [id:da09386662623023612] 
\draw    (252.3,77) -- (186.3,154) ;
%Straight Lines [id:da3132716998255566] 
\draw  [dash pattern={on 4.5pt off 4.5pt}]  (185.38,13.2) -- (186.3,154) ;
%Straight Lines [id:da543636490320943] 
\draw    (362.3,89) -- (494.3,133) ;
%Straight Lines [id:da7131617051878614] 
\draw    (287.42,129.88) -- (404.18,13.12) ;
\draw [shift={(406.3,11)}, rotate = 135] [fill={rgb, 255:red, 0; green, 0; blue, 0 }  ][line width=0.08]  [draw opacity=0] (5.36,-2.57) -- (0,0) -- (5.36,2.57) -- cycle    ;
\draw [shift={(285.3,132)}, rotate = 315] [fill={rgb, 255:red, 0; green, 0; blue, 0 }  ][line width=0.08]  [draw opacity=0] (5.36,-2.57) -- (0,0) -- (5.36,2.57) -- cycle    ;
%Straight Lines [id:da6941107063940581] 
\draw    (329.3,89) -- (329.3,134) ;
\draw [shift={(329.3,89)}, rotate = 90] [color={rgb, 255:red, 0; green, 0; blue, 0 }  ][fill={rgb, 255:red, 0; green, 0; blue, 0 }  ][line width=0.75]      (0, 0) circle [x radius= 2.01, y radius= 2.01]   ;
%Straight Lines [id:da5182885858757001] 
\draw    (329.3,134) -- (494.3,133) ;
%Straight Lines [id:da35764144938551445] 
\draw    (329.3,89) -- (362.3,89) ;
%Straight Lines [id:da060034258506083904] 
\draw    (252.15,55) -- (230.71,86) ;
%Straight Lines [id:da8212887817110629] 
\draw    (141.62,176) -- (186.3,154) ;

\end{tikzpicture}

}
    \caption{Left: GC polytope of $M$, Right: $U(2)$-Kirwan polytope of $M$}
    \label{fig:polytopes of M}
\end{figure}
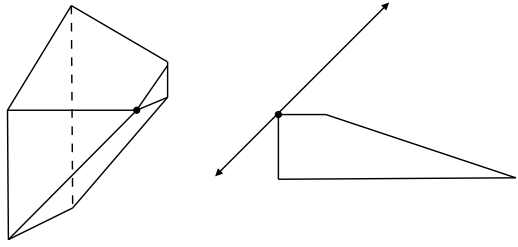 
There is a Gelfand--Cetlin system $\Phi:M\to \R^3$ on $M$ with image $\Delta = \Phi(M)$, whose interior fibers are Lagrangian tori. 

\medskip

First, we show that non-displaceable Lagrangian tori exist in Woodward's examples. 
\sloppy
\begin{thmnonumber}[Theorem \ref{thm Existence of a non-displaceable Lagrangian}]  
 There exists a non-displaceable Lagrangian torus in every multiplicity-free $U(2)$\-manifold $M$ constructed by Woodward in \cite{WoodwardExample}. 
\end{thmnonumber}  
To prove the theorem, we first obtain the leading order potential function of $M$ using the methods developed in Charest--Woodward \cite{CharestWoodward} and Venugopalan--Woodward \cite{VW}, and results in \cite{FOOOI, CO}. 
This allows us to locate a Gelfand--Cetlin fiber corresponding to certain critical points of the potential function. 
In the end, we show that the Lagrangian torus has non-trivial Floer cohomology, which implies that the Lagrangian is non-displaceable.

Moreover, we also show that the majority of the other fibers are displaceable by McDuff's probes \cite{Probes}. 
\begin{propnonumber}
[Proposition \ref{Proposition Probe-nondisplaceable set}] 
For $u=(u_1,u_2,u_3)\in \interior \Delta$, the Lagrangian torus $\Phi^{-1}(u)$ is displaceable if  
\[ u \not \in \set{\left(\frac{3\lambda_2-\lambda_3}{2}-t,\frac{\lambda_2+\lambda_3}{2}+t,\lambda_2 \right) }{0\leq t\leq \frac{\lambda_2-\lambda_3}{2}} \cup\{ u_0\}, \]
where $u_0$ is as in \eqref{valuation of crit points}. 
(See Figure \ref{fig: possibly non-displaceable set}.)
\end{propnonumber}

Finally, we prove that the collection of Lagrangian branes identified in Section \ref{section Critical points of the potential}
satisfies the generation criterion (\cite{AbouzaidGeneration, AFOOO}) in the sense that the open-closed map is an isomorphism .

\begin{thmnonumber}
(Theorem \ref{OC is an isomorphism})
    The map $OC: \bigoplus\limits_{i=1}^6 HH_*(CF(L_i,L_i))\to  QH^*(M, \Lambda) $ in \autoref{length 0 OC} is an isomorphism.  
\end{thmnonumber} 

\medskip
The structure of the paper is as follows.
In Section \ref{section The multiplicity-free $U(2)$-manifolds}, we begin by reviewing the Gelfand--Cetlin systems on $U(n)$-coadjoint orbits and recalling the construction in \cite{WoodwardExample}.
We then compute the leading order potential function in Section \ref{section potential}.
Analyzing the critical points of the potential function, in Section \ref{section Critical points of the potential}, we identify a Lagrangian torus inside the Gelfand--Cetlin polytope along with a collection of weak bounding cochains.
In Section \ref{section nondisplaceability}, we prove that this torus is non-displaceable and also show that the majority of the other Gelfand--Cetlin fibers are displaceable by McDuff’s probes \cite{Probes}. Finally, in Section \ref{section Generation of the Fukaya category}, we prove Theorem \ref{OC is an isomorphism}, which establishes that the collection of Lagrangian branes we obtain in Section \ref{section Critical points of the potential} satisfies the generation criterion. 

\subsection*{Acknowledgement} 
I express my gratitude to Chris Woodward for his guidance and support throughout this project and to Kenji Fukaya for suggesting this problem to me. I am thankful to Chi Hong Chow, Guangbo Xu, and Yiyue Zhu for helpful discussions. Finally, I am deeply grateful to Nate Bottman for his support and the Max Planck Institute for Mathematics in
Bonn for its hospitality and financial support.

\section{\texorpdfstring{The multiplicity-free $U(2)$-manifolds}{The multiplicity-free U(2)-manifolds}}
\label{section The multiplicity-free $U(2)$-manifolds}

A symplectic manifold that is equipped with a Hamiltonian action by a compact connected Lie group is called \textit{multiplicity-free} if the reduction of every moment fiber is zero-dimensional. Other equivalent definitions are included in Appendix \ref{section Multiplicity-free manifolds}.  
Compact, connected manifolds with effective, multiplicity-free torus actions are precisely the compact toric manifolds classified by Delzant \cite{DelzantClassification}.
They form one of the main families of examples in which homological mirror symmetry has been extensively studied.
Other classical examples of multiplicity-free manifolds include the coadjoint orbits of $U(n)$  and $SO(n)$.
On these spaces, Guillemin and Sternberg~\cite{GSGC,GSThimm} constructed completely integrable systems, called Gelfand--Cetlin systems, via Thimm’s trick.

Woodward's construction of the multiplicity-free $U(2)$-manifolds in \cite{WoodwardExample} relies on the existence of a Gelfand--Cetlin system on a generic $U(3)$-coadjoint orbit. 
We review some properties of the coadjoint orbits, the construction of the Gelfand--Cetlin system,  and the symplectic cut that produces Woodward's examples. 
We refer the reader to \cite{WoodwardExample}, \cite{WoodwardClassification}, \cite{GSGC}, \cite{GSThimm},  \cite{Pabiniak}, and \cite{lanethesis} for detailed discussions.

\subsection{Coadjoint orbits of unitary groups}
\label{section Coadjoint orbits of unitary groups}
We first review some properties of the coadjoint orbits of $G=U(n)$. 
We will identify the coadjoint action of $U(n)$ on $\mathfrak{u}(n)^*$  with the conjugation action of $U(n)$ on the space $\operatorname{Herm}_n$ of Hermitian $n\times n$ matrices through the isomorphisms 
\[ \mathfrak{u}(n)^* \xrightarrow{\varphi} \mathfrak{u}(n) \xrightarrow{\psi}\operatorname{Herm}_n.   \]
Here $\varphi(B):=A$ is characterized by $\trace(AX)=\pair{B,X}$ for all $X\in \mathfrak{u}(n)$, and $\psi(A)=iA$. 

Let $\lam \in \mathfrak{u}(n)^*$ be the diagonal matrix with nonincreasing diagonal entries 
\begin{equation}
\label{decreasing evalues}
\underbrace{\lam_{1}=\ldots =\lam_{n_1}}_{k_1\text{ copies}}\geq \underbrace{\lam_{n_1+1}= \ldots= \lam_{n_2}}_{k_2\text{ copies}}\geq \cdots \geq \underbrace{\lam_{n_{r}+1}= \ldots= \lam_{n}}_{k_{r+1}\text{ copies}},    
\end{equation} 
The coadjoint action of $U(n)$ on $\lam$ has stabilizer 
$U(k_1)\times \cdots U(k_{r+1})$. 
Hence, the coadjoint orbit $\orbit_{\lam}$ of $\lam$ is a homogeneous manifold 
\begin{equation}\label{coadjoint orbit as a homogeneous manifold}
 \orbit_{\lam} \cong U(n)/ \lrp{U(k_1)\times \cdots U(k_{r+1})}    
\end{equation} 
of dimension $n^2 - \sum_{i=1}^{r+1}k_i^2$. And the tangent space at a point $x\in \orbit_{\lam}$ is isomorphic to $\un/ \bigoplus_{i=1}^{r+1}\mathfrak{u}(k_i)$. 

Every coadjoint orbit comes equipped with the \textit{Kirillov–Kostant–Souriau symplectic form}, given by 
\[ \omega_{x}(A,B) =  \pair{x, [A,B]} \quad \forall x\in \orbit_{\lam}, \quad \forall A,B \in T_x\orbit_{\lam}. \]
The coadjoint action on $\mathcal{O}_{\lambda}$ is Hamiltonian, and the inclusion map $\mathcal{O}_{\lambda}\hookrightarrow \g^*$ is a moment map.

Moreover, $\orbit_{\lam}$ can be identified with a partial flag manifold. 
Let 
\[ 0 = n_0 <n_1< \cdots < n_{r+1} = n\]
be an strictly increasing sequence of integers. 
The \textit{partial flag manifold} $F(n_1,\ldots, n_r; n)$ is defined by
\[ \left \{ 0 =V_0 \subset V_1 \subset \cdots \subset V_{r+1}=\C^n \;\middle| \; \begin{aligned}
& V_i\subset \C^n \text{ is a complex subspace}, \\
&  \dim V_i = n_i \quad \forall 1\leq i \leq r  
\end{aligned}
 \right\}.  \]
 We call $F(1,2,\ldots, n-1;n)$ the \textit{full flag manifold} and denote it by $F(n)$.  

Let $k_i=n_i-n_{i-1}$.  
Since $U(n)$ acts on $F(n)$ transitively, and the stabilizer of the standard flag
\[ 0 \subset \Span_{\C}\{e_1,\ldots, e_{n_1}\} \subset \cdots \subset \Span_{\C}\{e_1,\ldots, e_n\} = \C^n \]
is $U(k_1)\times \cdots U(k_{r+1})$, 
we can identify the partial flag manifold $ F(n_1,\ldots, n_r; n)$ with the coadjoint orbit $\orbit_{\lam}$: 
\[ F(n_1,\ldots, n_r; n)\cong U(n)/ \lrp{U(k_1)\times \cdots U(k_{r+1})}\cong \orbit_{\lam}. \]

We demonstrate the construction of a completely integrable system on $\orbit_{\lam}$ as follows. 
Identifying $A\in U(j)$ with block matrices of the form $\begin{pmatrix}
    A & 0 \\
    0 & I
\end{pmatrix} \in U(n)$, we have the following chain of Lie algebras
\[ \mathfrak{u}(1) \subset \mathfrak{u}(2)\subset \cdots \subset  \mathfrak{u}(n)  
\]
and projection maps $p_j:\mathfrak{u}(n)^* \to \mathfrak{u}(j)^*$. 

Let us choose, for each $j$, the positive Weyl chamber $t_{j,+}^*$ of $\mathfrak{u}(j)^*$ to be the one corresponding to the subset of Hermitian matrices consisting of diagonal matrices with non-increasing diagonal entries.  
Let $\iota: \mathcal{O}_{\lambda} \to \mathfrak{u}(n)^*$ be the inclusion map, and let $q_j: \mathfrak{u}(j)^*
\to \mathfrak{t}_{j,+}^*$ be the quotient map by the coadjoint action of $U(j)$. Note that $q_j$ is not smooth everywhere. 
Guillemin-Sternberg \cite{GSGC} showed that 
\[ \left(q_{n-1}\circ p_{n-1}\circ\iota, \ldots,  q_1\circ p_{1}\circ\iota\right): \mathcal{O}_{\lambda} \to \mathfrak{t}_{n-1}^*\oplus\cdots \oplus  \mathfrak{t}_{1}^* \] 
defines a completely integrable system 
known as the \textit{Gelfand–Cetlin system}, which is continuous and smooth on an open dense subset of $\orbit_{\lam}$. 

More specifically, $q_k\circ p_{k}\circ\iota = \lrp{\lam_{1}^{(k)}, \ldots \lam_{k}^{(k)}}$, 
where $\lam_{j}^{(k)}: \mathcal{O}_{\lambda}  \to \R$ sends $x$ to the $j$-th largest eigenvalue of the principal $k\times k$ minor of $x$. 
By the min-max theorem, 
% \footnote{Theorem 2 (min-max theorem) Let $A$ be an $n \times n$ Hermitian matrix. Then, for all $1\leq i \leq n$, the $i$-th eigenvalue $\lam_i(A)$ of $A$ satisfies
% \[ \lam_i(A) = \sup_{
% \substack{V\subset \C^n \text{subspace}\\ \dim V=i }
% } \inf_{ \substack{v \in V\\ |v|=1} } v^* A v \]
% and
% \[ \lam_i(A) = \inf_{\substack{V\subset \C^n\\ \dim V=n-i+1}} \sup_{ \substack{v \in V\\ |v|=1} } v^* A v 
% \] 
% for all $1 \leq i \leq n$.}
we have the following inequalities 
\[
\xymatrix{
\lam_1 
\ar @{} [dr] |{\upgeq}
& 
& 
\lam_2 
\ar @{} [dr] |{\upgeq} 
& \cdots 
&  
\lam_{n-1} 
\ar @{} [dr] |{\upgeq}
& &
\lam_n 
\\
& 
\lam^{(n-1)}_{1} 
\ar @{} [dr] |{\upgeq} 
\ar @{} [ur] |{\downgeq} 
&  
& 
\cdots 
\ar @{} [ur] |{\downgeq} 
\ar @{} [dr] |{\upgeq} 
& 
& \lam_{n-1}^{(n-1)}
\ar @{} [ur] |{\downgeq} 
& 
\\
& 
& 
\lam_1^{(n-2)}
\ar @{} [ur] |{\downgeq} 
\ar @{} [dr] |{\updots} 
& 
&
\lam_{n-2}^{(n-2)}
\ar @{} [ur] |{\downgeq} 
\\
& 
&
& 
\lam^{(1)}_{1}
\ar @{} [ur] |{\downdots}
& 
&
& 
}
\]
By \eqref{decreasing evalues}, there are 
\[ \frac{n(n-1)}{2} - \sum_{i=1}^{r+1}\frac{(k_i-1)k_i}{2} = \frac{n^2-\sum\limits_{i=1}^{r+1}k_i^2}{2} = \frac{1}{2}\dim \orbit_{\lam}  \]
non-constant elements in the set 
\begin{equation}
\label{GC functions}
\set{ \lam_{j}^{(k)}}{1\leq k \leq n-1, 1\leq j\leq k }.     
\end{equation}  

\subsection{Construction of Woodward's examples}
\label{section Construction of Woodward's examples} 

In this section, we review the construction \cite{WoodwardExample} of the multiplicity-free Hamiltonian $U(2)$-manifolds through non-abelian symplectic cuts \cite{lmtw} using an action induced by the Gelfand--Cetlin system on a generic coadjoint orbit of $U(3)$.  

\begin{defn}[Multiplicity-free]
   A compact connected manifold Hamiltonian $G$-manifold $(M,\omega,G,\Phi)$ is \textit{multiplicity-free} if all of the reduced spaces $\Phi^{-1}(a)/G_a$ are points, where $\Phi: M\to \mathfrak{g}^*$ is the moment map for the Hamiltonian action of a compact connected Lie group $G$ on $(M,\omega)$. 
\end{defn}
Notable examples include compact symplectic toric manifolds, which are exactly the manifolds equipped with effective, multiplicity-free Hamiltonian actions by tori. 
Another example of a multiplicity-free manifold is the $U(3)$-coadjoint orbit
$\orbit_{\lam}$ of $\lam$, where $\lam$ is any diagonal matrix with diagonal entries
\[ \lam_1 > \lam_2 > \lam_3, \qquad \lam_1 , \lam_2 ,\lam_3  \in \Z.  \] 
To simplify the notation, we denote the Gelfand--Cetlin coordinates by
\[ 
u_1= \lam^{(2)}_1, \quad u_2= \lam^{(2)}_2, \quad u_3= \lam^{(1)}_1. 
\]
They satisfy the following inequalities, which define the Gelfand--Cetlin polytope $\Tilde{\Delta}$ of $\orbit_{\lam}$. (See Figure \ref{fig: polytopes of the coadjoint orbit}.)
\begin{equation}
\label{interlacing inequalities}
\xymatrix{
\lam_1 
\ar @{} [dr] |{\upgeq}
% _{\text{Back}}
& & 
\lam_2 
\ar @{} [dr] |{\upgeq}
% ^{\text{Right}}
& & 
\lam_3 
\\
& 
u_1 
\ar @{} [dr] |{\upgeq}
% _{\text{Top}}
\ar @{} [ur] |{\downgeq}
% ^{\text{Front}}
& & 
u_2
\ar @{} [ur] |{\downgeq}
% _{\text{Left}}
& 
\\
& & 
u_3 \ar @{} [ur] |{\downgeq}
% _{\text{Bottom}}
& & 
}    
\end{equation} 
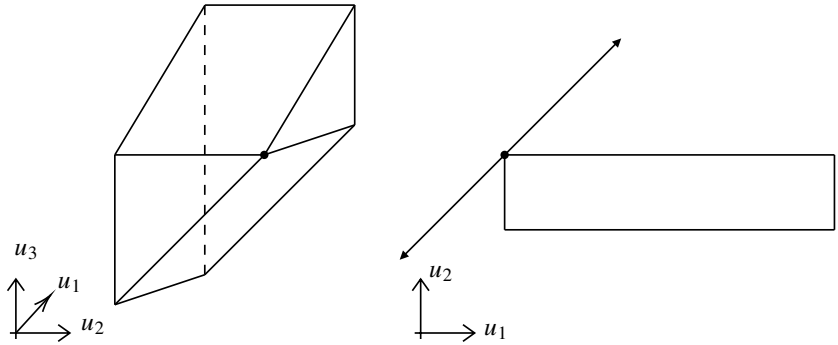
\begin{figure}[h!]
\centering
%%% tikz picture flag manifold 

\tikzset{every picture/.style={line width=0.75pt}} %set default line width to 0.75pt        
\resizebox{0.8\textwidth}{!}{
\begin{tikzpicture}[x=0.75pt,y=0.75pt,yscale=-1,xscale=1]
%uncomment if require: \path (0,263); %set diagram left start at 0, and has height of 263

%Straight Lines [id:da5353435675181301] 
\draw    (99,117) -- (99,207) ;
%Straight Lines [id:da21932078003118483] 
\draw    (99,207) -- (188.83,117) ;
%Straight Lines [id:da06621808443074606] 
\draw    (243,27) -- (243,99) ;
%Straight Lines [id:da7509819717961005] 
\draw    (188.72,117) -- (243,99) ;
%Straight Lines [id:da8418816539897478] 
\draw    (99,207) -- (153,189) ;
%Straight Lines [id:da9820152390451401] 
\draw  [dash pattern={on 4.5pt off 4.5pt}]  (153,27) -- (153,189) ;
%Straight Lines [id:da43540384432073753] 
\draw    (153,189) -- (243,99) ;
%Straight Lines [id:da19943670549152237] 
\draw    (272.12,177.88) -- (400.88,49.12) ;
\draw [shift={(403,47)}, rotate = 135] [fill={rgb, 255:red, 0; green, 0; blue, 0 }  ][line width=0.08]  [draw opacity=0] (5.36,-2.57) -- (0,0) -- (5.36,2.57) -- cycle    ;
\draw [shift={(270,180)}, rotate = 315] [fill={rgb, 255:red, 0; green, 0; blue, 0 }  ][line width=0.08]  [draw opacity=0] (5.36,-2.57) -- (0,0) -- (5.36,2.57) -- cycle    ;
%Shape: Axis 2D [id:dp8055327977818851] 
\draw  (36,224) -- (72,224)(39.6,191.6) -- (39.6,227.6) (65,219) -- (72,224) -- (65,229) (34.6,198.6) -- (39.6,191.6) -- (44.6,198.6)  ;
%Straight Lines [id:da8232646163291494] 
\draw    (39.6,224) -- (59.25,202.48) ;
\draw [shift={(60.6,201)}, rotate = 132.4] [color={rgb, 255:red, 0; green, 0; blue, 0 }  ][line width=0.75]    (10.93,-3.29) .. controls (6.95,-1.4) and (3.31,-0.3) .. (0,0) .. controls (3.31,0.3) and (6.95,1.4) .. (10.93,3.29)   ;
%Shape: Axis 2D [id:dp04947058155185846] 
\draw  (279,224) -- (315,224)(282.6,191.6) -- (282.6,227.6) (308,219) -- (315,224) -- (308,229) (277.6,198.6) -- (282.6,191.6) -- (287.6,198.6)  ;
%Straight Lines [id:da44002198075609944] 
\draw    (99,117) -- (188.72,117) ;
\draw [shift={(188.72,117)}, rotate = 0] [color={rgb, 255:red, 0; green, 0; blue, 0 }  ][fill={rgb, 255:red, 0; green, 0; blue, 0 }  ][line width=0.75]      (0, 0) circle [x radius= 2.01, y radius= 2.01]   ;
%Straight Lines [id:da3734754046516503] 
\draw    (188.83,117) -- (243,27) ;
%Straight Lines [id:da33116734977217954] 
\draw    (99,117) -- (153.17,27) ;
%Straight Lines [id:da09134199290596545] 
\draw    (153.17,27) -- (243,27) ;
%Straight Lines [id:da3101402793344197] 
\draw    (333,117) -- (333,162) ;
\draw [shift={(333,117)}, rotate = 90] [color={rgb, 255:red, 0; green, 0; blue, 0 }  ][fill={rgb, 255:red, 0; green, 0; blue, 0 }  ][line width=0.75]      (0, 0) circle [x radius= 2.01, y radius= 2.01]   ;
%Straight Lines [id:da4976503856459954] 
\draw    (333,162) -- (531,162) ;
%Straight Lines [id:da9597002803347622] 
\draw    (333,117) -- (531,117) ;
%Straight Lines [id:da19801054242253247] 
\draw    (531,117) -- (531,162) ;

% Text Node
\draw (63.6,189.4) node [anchor=north west][inner sep=0.75pt]  [xscale=0.8,yscale=0.8]  {\Large $u_{1}$};
% Text Node
\draw (77.6,214.4) node [anchor=north west][inner sep=0.75pt]  [xscale=0.8,yscale=0.8]  {\Large $u_{2}$};
% Text Node
\draw (37.6,169.4) node [anchor=north west][inner sep=0.75pt]  [xscale=0.8,yscale=0.8]  {\Large $u_{3}$};
% Text Node
\draw (319.6,217.4) node [anchor=north west][inner sep=0.75pt]  [xscale=0.8,yscale=0.8]  {\Large $u_{1}$};
% Text Node
\draw (286.6,182.4) node [anchor=north west][inner sep=0.75pt]  [xscale=0.8,yscale=0.8]  {\Large $u_{2}$};

\end{tikzpicture}
}
\caption{Left: GC polytope of $\orbit_{\lam}$, Right: $U(2)$-moment polytope of $\orbit_{\lam}$}
    \label{fig: polytopes of the coadjoint orbit}
\end{figure}
Moreover, we can define a Hamiltonian $T^3$-action on a dense open subset $U$ of $\orbit_{\lam}$ by applying ``Thimm's trick". 
Let $1\leq k\leq 2$. 
Let $t\in T^k\subset U(k)$ be the set of all diagonal matrices in $U(k)$, and let $x\in \orbit_{\lam}$. 
Since $\interior \mathfrak{t}_{k,+}^*$ is the principal stratum of the coadjoint $U(k)$-action on $\orbit_{\lam}$, 
the principal cross-section $U_k:=\left(\lambda_1^{(k)},\ldots, \lambda_k^{(k)}\right)^{-1}(\interior \mathfrak{t}_{k,+}^*)$ is an open dense subset of $\orbit_{\lam}$ by \cite{lmtw} Theorem 3.1. 
Since every coadjoint orbit of $U(k)$ intersects $\mathfrak{t}_{k,+}^*$ in exactly one point, there exists $B \in U(k)$ such that $\Ad_B^*(x^{(k)}) = q_2(x)$, where $x^{(k)}$ is the principal $k\times k$ minor of $x$.  
By \cite{WoodwardExample} Proposition 3.4, 
\begin{equation}
\label{Thimm trick} t\ast x 
=  
\Ad^*_{ \small \begin{pmatrix}
 B^{-1}tB & 0\\
 0 & 1   
\end{pmatrix}}(x)  
\end{equation}
defines a Hamiltonian $T^k$-action on $U_k$ with moment map $\left(\lambda^{(k)}_1, \ldots,  \lambda^{(k)}_k\right)$.

Taking the product action by $T^2\times T^1$, we get an effective Hamiltonian $T^3$-action with moment map $(u_1,u_2,u_3):U\to \R^3$ on the open dense subset 
\[ U=U_1\cap U_2 =\orbit_{\lam}\setminus (u_1,u_2,u_3)^{-1}\{(\lambda_2,\lambda_2,\lambda_2)\}
\] 
of $\orbit_{\lam}$.  
Under this action, $U$ is a toric manifold with toric moment map image $\tilde{\Delta} \setminus \{(\lambda_2, \lambda_2,\lambda_2)\}$.

We now review the construction of the non-abelian symplectic cut. 
Let $n> 2$ and $ S^1\hookrightarrow  T^2, \theta \mapsto \operatorname{diag}(e^{i\theta}, e^{in\theta})$ be a circle subgroup.  
Consider the Gelfand--Cetlin action restricted to this circle subgroup with moment map $\mu=u_1+n u_2$. 
Let $r$ be an integer such that $\lambda_2 < r < \lambda_1$ and $\frac{\lambda_2-\lambda_3}{r-\lambda_2}<\frac{1}{n}$. 
We claim that we can use this circle action to perform a symplectic cut such that $S^1$ acts freely on $\mu^{-1}(r+n\lambda_3)$ as shown in Figure \ref{fig:the symplectic cut in Woodward's construction}. This level set is mapped to the dotted line in the $U(2)$-moment polytope. 
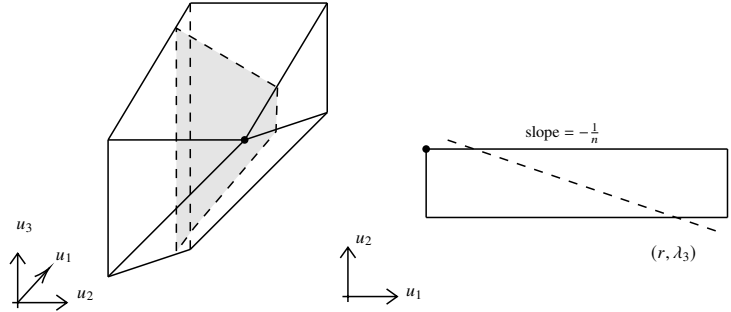
\begin{figure}[h!]
    \centering
\resizebox{0.7\textwidth}{!}{
\tikzset{every picture/.style={line width=0.75pt}} %set default line width to 0.75pt        

\begin{tikzpicture}[x=0.75pt,y=0.75pt,yscale=-1,xscale=1]
%uncomment if require: \path (0,236); %set diagram left start at 0, and has height of 236

%Shape: Polygon [id:ds1945935296189384] 
\draw  [color={rgb, 255:red, 0; green, 0; blue, 0 }  ,draw opacity=1 ][fill={rgb, 255:red, 128; green, 128; blue, 128 }  ,fill opacity=0.2 ][dash pattern={on 4.5pt off 4.5pt}][line width=0.75]  (230.33,69.57) -- (229.33,98.57) -- (175.42,161.89) -- (163.86,177.5) -- (163.86,30.5) -- cycle ;
%Shape: Axis 2D [id:dp154765808627652] 
\draw  (273,207) -- (309,207)(276.6,174.6) -- (276.6,210.6) (302,202) -- (309,207) -- (302,212) (271.6,181.6) -- (276.6,174.6) -- (281.6,181.6)  ;
%Straight Lines [id:da08766519906119286] 
\draw  [dash pattern={on 4.5pt off 4.5pt}]  (342.12,104.09) -- (519.12,164.09) ;
%Straight Lines [id:da8816931258679127] 
\draw    (328,110) -- (328,155) ;
\draw [shift={(328,110)}, rotate = 90] [color={rgb, 255:red, 0; green, 0; blue, 0 }  ][fill={rgb, 255:red, 0; green, 0; blue, 0 }  ][line width=0.75]      (0, 0) circle [x radius= 2.01, y radius= 2.01]   ;
%Straight Lines [id:da6782042959885013] 
\draw    (328,155) -- (526,155) ;
%Straight Lines [id:da8590539406637137] 
\draw    (328,110) -- (526,110) ;
%Straight Lines [id:da18942635354923365] 
\draw    (526,110) -- (526,155) ;
%Straight Lines [id:da01561317723361455] 
\draw    (119,104) -- (119,194) ;
%Straight Lines [id:da8047231461250318] 
\draw    (119,194) -- (208.83,104) ;
%Straight Lines [id:da11847973681560042] 
\draw    (208.72,104) -- (263,86) ;
%Straight Lines [id:da37525593136080904] 
\draw    (119,194) -- (173,176) ;
%Straight Lines [id:da07620184921349638] 
\draw  [dash pattern={on 4.5pt off 4.5pt}]  (173,14) -- (173,176) ;
%Straight Lines [id:da3221781422876696] 
\draw    (173,176) -- (263,86) ;
%Shape: Axis 2D [id:dp8058169920157092] 
\draw  (56,211) -- (92,211)(59.6,178.6) -- (59.6,214.6) (85,206) -- (92,211) -- (85,216) (54.6,185.6) -- (59.6,178.6) -- (64.6,185.6)  ;
%Straight Lines [id:da2699402745244187] 
\draw    (59.6,211) -- (79.25,189.48) ;
\draw [shift={(80.6,188)}, rotate = 132.4] [color={rgb, 255:red, 0; green, 0; blue, 0 }  ][line width=0.75]    (10.93,-3.29) .. controls (6.95,-1.4) and (3.31,-0.3) .. (0,0) .. controls (3.31,0.3) and (6.95,1.4) .. (10.93,3.29)   ;
%Straight Lines [id:da10232040760017258] 
\draw    (119,104) -- (208.72,104) ;
\draw [shift={(208.72,104)}, rotate = 0] [color={rgb, 255:red, 0; green, 0; blue, 0 }  ][fill={rgb, 255:red, 0; green, 0; blue, 0 }  ][line width=0.75]      (0, 0) circle [x radius= 2.01, y radius= 2.01]   ;
%Straight Lines [id:da0036339257811669468] 
\draw    (208.83,104) -- (263,14) ;
%Straight Lines [id:da10609919445906557] 
\draw    (119,104) -- (173.17,14) ;
%Straight Lines [id:da5722933685600614] 
\draw    (173.17,14) -- (263,14) ;
%Straight Lines [id:da10527896798246783] 
\draw    (263,14) -- (263,86) ;

% Text Node
\draw (313.6,200.4) node [anchor=north west][inner sep=0.75pt]  [xscale=0.8,yscale=0.8]  {$u_{1}$};
% Text Node
\draw (280.6,165.4) node [anchor=north west][inner sep=0.75pt]  [xscale=0.8,yscale=0.8]  {$u_{2} \ $};
% Text Node
\draw (83.6,176.4) node [anchor=north west][inner sep=0.75pt]  [xscale=0.8,yscale=0.8]  {$u_{1}$};
% Text Node
\draw (97.6,201.4) node [anchor=north west][inner sep=0.75pt]  [xscale=0.8,yscale=0.8]  {$u_{2}$};
% Text Node
\draw (57.6,156.4) node [anchor=north west][inner sep=0.75pt]  [xscale=0.8,yscale=0.8]  {$u_{3} \ $};
% Text Node
\draw (392,91.4) node [anchor=north west][inner sep=0.75pt]  [font=\small,xscale=0.8,yscale=0.8]  {$\text{slope} =-\frac{1}{n}$};
% Text Node
\draw (474,172.4) node [anchor=north west][inner sep=0.75pt]  [xscale=0.8,yscale=0.8]  {$( r,\lambda _{3})$};

\end{tikzpicture}

}
    \caption{The symplectic cut. 
    Right: Gelfand--Cetlin $T^2$-moment polytope. }
    \label{fig:the symplectic cut in Woodward's construction}
\end{figure}
For a symplectic cut as in Figure \ref{fig:the symplectic cut in Woodward's construction} to make sense, we need the action to be free on $\mu^{-1}(r+n\lambda_3)$. 
Recall the following lemma by Delzant. 
\begin{lem}[Delzant, \cite{WoodwardExample} Lemma 3.6] 
\label{Delzant lemma}
Let $M$ be a multiplicity-free compact Hamiltonian $G$-space with moment polytope $\Delta$, let $F$ be a face of $\Delta$ contained in $\interior \mathfrak{t}_+^*$, and let $m\in \Phi^{-1}(F)$. Then the isotropy subgroup $G_m$ of $m$ is connected and its Lie algebra is $\operatorname{ann}(F) \subset \mathrm{t}$.
\end{lem}

We apply Lemma \ref{Delzant lemma} to the multiplicity-free $U(2)$-manifold $\mathcal{O}_{\lambda}$, the isotropy subgroup for the Gelfand--Cetlin circle action at an element $C\in q\circ \Phi^{(2)}(x)$ for an $x$ in the interior of $[\lambda_2,\lambda_1]\times [\lambda_3,\lambda_2]$ is trivial.  
Suppose $A$ is in the preimage of one of the endpoints of the intersection of the dotted line with the rectangle, and let $B= \begin{pmatrix}
    \alpha & \beta \\
    \gamma & \delta
\end{pmatrix}$ be as in \autoref{Thimm trick}. 
If $t=\begin{pmatrix}
    t_1 & 0 \\
    0 & t_2
\end{pmatrix}\in $ satisfies $t \ast A=A$, then by Lemma \ref{Delzant lemma},
\[
B^{-1}AB \in \begin{pmatrix}
    1 & 0 \\
    0 & U(1)
\end{pmatrix} \quad
\Longrightarrow  \quad t_1=t_2=1 \in U(1). 
\]
 
This implies that the isotropy subgroup for the Gelfand--Cetlin action by the chosen circle at $A$ is trivial.  
Therefore, the action is free on $\mu^{-1}(r+n\lambda_3)$. 

We take $M$ to be the symplectic cut at a regular level $r+n\lambda_3$ of $\mu$: 
\begin{align}
\label{defn of M}
    M \cong (\orbit_{\lam})_{\mu \leq r+n\lambda_3} \cong \mu^{-1}(-\infty, r+n\lambda_3) \cup \mu^{-1}(r+n\lambda_3)/S^1. 
\end{align}
We require the parameters $n\in \Z, r,\lambda_1,\lambda_2,\lambda_3\in \R$ to satisfy the following conditions.
\begin{align}
\label{n>2}
& n > 2 
\\
&\lambda_2<r<\lambda_1 \label{relations between lambda and r}\\
& \frac{\lambda_2-\lambda_3}{r-\lambda_2}<\frac{1}{n}
\label{cannot cut the S3 vertex}
\end{align}
Woodward (\cite{WoodwardExample} and \cite{WoodwardClassification}) showed that $M$, denoted by $H_n$ in \cite{WoodwardExample}, is a multiplicity-free Hamiltonian $U(2)$-manifold. Using Tolman's X-ray criterion in \cite{Tolman}, Woodward also showed that $M$ does not admit a $T^2$-invariant K\"{a}hler structure. 

The Gelfand--Cetlin polytope $\Delta$ of the manifold $M$ defined in Section \ref{section Construction of Woodward's examples} resulting from the symplectic cut is defined by the inequalities in \autoref{interlacing inequalities}, except that we replace the inequality $\lambda_1-u_1\geq 0$ with $-u_1-nu_2+r+n\lambda_3\geq 0$, and is shown in Figure \ref{fig:polytopes of M}. 
$M$ inherits the coordinates of the Gelfand--Cetlin system, which we denote by
\begin{equation}
\label{Gelfand--Cetlin system on M map}
\Phi = (u_1,u_2,u_3): M \to \Delta \subset \R^3.     
\end{equation}

\section{\texorpdfstring{The split $A_{\infty}$ algebra}{The split A-infinity algebra}}

\label{section split algebra}
Consider $M$ as in \autoref{defn of M} and $L= \Phi^{-1}(u)$ for any $u\in \interior \Delta$. 
Since $\Phi: M\setminus \Phi^{-1}\{(\lambda_2,\lambda_2,\lambda_2)\} \to \Delta \setminus \{(\lambda_2,\lambda_2,\lambda_2)\}$ is a toric moment map,  $L$ is a Lagrangian torus. 
Following \cite{VWSplit}, we consider the split $A_{\infty}$ algebra associated to a relateive Lagrangian $L$, and any local system in $ \Hom(\pi_1(L),  U_{\Lambda} )$.  

Let 
	\begin{equation}
	\label{Novikov ring}
	\Lambda_{0}
	=
	\set{
		x=\sum_{i\in \N} a_iT^{\lambda_i} 
	}
	{  
     a_i\in \C, \quad \lambda_i\in \R  \quad \forall i\in \N, \quad \lim_{i\to \infty}\lambda_i = \infty 
	}	  
	\end{equation}  
be the Novikov ring, and $\Lambda$ be the fractional field of $\Lambda_0$. 
Let 
\begin{equation}
\label{unitary novikov}
U_{\Lambda} 
=\set{
	x=a_0+\sum_{i\in \Z_+}a_i T^{\lambda_i}
}
{  
\begin{aligned}
& x\in \Lambda_0, \quad 
    a_0\in \C\setminus \{0\},  \\
& \lambda_i\in \R_{>0}   \quad \forall i\in \Z_+, \quad \lim_{i\to \infty}\lambda_i = \infty
\end{aligned}
}.     
\end{equation}

The coordinates on $\R^3\supset \Delta $ induce a basis $\{e_1^*,e_2^*, e_3^*\}$ for $\pi_1(L)$. Then every $\rho \in \Hom(\pi_1(L),  U_{\Lambda} )$ is determined by $\rho(e_1^*), \rho(e_2^*), \rho(e_3^*)$. 
 
Let $\epsilon>0$ be a rational number such that $Q_> = \{ u_2-u_1+ \epsilon <0\}$ intersects $\Delta$ nontrivially and $\Phi(L)\in Q_>$. 
We also define
\[ Q= \{u_2-u_1+ \epsilon = 0\},\quad  Q_< =  \{ u_2-u_1+ \epsilon >0 \}.\] 
 Then $(M, \mathcal{Q}=\{Q,Q_>,Q_<\}, \Phi)$ is a \textit{tropical Hamiltonian action} as in \cite{VW} Definition 3.3.

Consider the hypersurface $Z\subset M$ defined by $\Phi^{-1}\lrp{Q}$.

Let 
\[ M_> =\Phi^{-1}(\{ u_2-u_1+ \epsilon <0\}), \qquad M_< =\Phi^{-1}(\{ u_2-u_1+ \epsilon >0\}).  \]
Let $\mathfrak{t}_Q = \Span\{(-1,1,0)\}\subset \mathfrak{t}\cong \R^3$ and $T_Q = \exp(\mathfrak{t}_Q )\cong S^1$. By Lemma \ref{Delzant lemma}, the $T_Q$-action on $Z$ is free, and we can perform a symplectic cut along $Z$. 
Let $Y=Z/T_Q$ be the quotient. 
By properties of symplectic cuts \cite{Cuts}, \[ M_+ = M_>\cup Y, \quad  M_- = M_<\cup Y  \]
carry natural symplectic structures, and $L$ is a Lagrangian torus in $M_+$. 

The disjoint union of $M_>, M_<$, and 
$M_Q\cong \Phi^{-1}(Q)\times \mathfrak{t}_Q$, is the \textit{broken manifold} $\mathfrak{X}$, as in Definition 3.31 in \cite{VW}, corresponding to this tropical Hamiltonian action.  

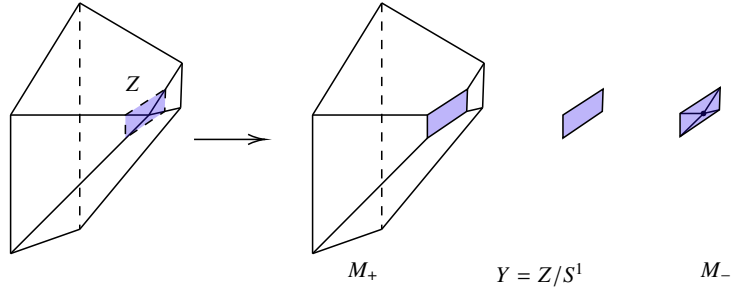
\begin{figure}[h!]
    \centering
 
\tikzset{every picture/.style={line width=0.75pt}} %set default line width to 0.75pt        
\resizebox{0.7\textwidth}{!}{
\begin{tikzpicture}[x=0.75pt,y=0.75pt,yscale=-1,xscale=1]
%uncomment if require: \path (0,263); %set diagram left start at 0, and has height of 263

%Straight Lines [id:da7767697310387749] 
\draw    (45,116) -- (45,206) ;
%Straight Lines [id:da020239944311097546] 
\draw    (45,206) -- (134.83,116) ;
%Straight Lines [id:da5965769731413388] 
\draw    (89.89,43.2) -- (45,116) ;
%Straight Lines [id:da3601985669043801] 
\draw    (134.72,116) -- (155.36,111.27) ;
%Straight Lines [id:da439840051517534] 
\draw    (45,206) -- (89.89,190.2) ;
%Straight Lines [id:da6195671424377922] 
\draw    (45,116) -- (134.72,116) ;
%Straight Lines [id:da40242990600552975] 
\draw    (156.36,82.27) -- (134.72,116) ;
%Straight Lines [id:da580485178596951] 
\draw    (89.89,43.2) -- (156.36,82.27) ;
%Straight Lines [id:da4091566912168394] 
\draw    (156.36,82.27) -- (155.36,111.27) ;
%Straight Lines [id:da31307308272605294] 
\draw    (155.36,111.27) -- (89.89,190.2) ;
%Straight Lines [id:da8524017191268627] 
\draw  [dash pattern={on 4.5pt off 4.5pt}]  (89.89,43.2) -- (89.89,190.2) ;
%Shape: Polygon [id:ds907437540743236] 
\draw  [fill={rgb, 255:red, 38; green, 7; blue, 241 }  ,fill opacity=0.3 ][dash pattern={on 4.5pt off 4.5pt}] (145.54,99.13) -- (145.04,113.63) -- (145.04,113.63) -- (119.5,131) -- (119.5,116) -- cycle ;
%Straight Lines [id:da22714144298587968] 
\draw    (241,116) -- (241,206) ;
%Straight Lines [id:da2652807515835355] 
\draw    (241,206) -- (315.5,131) ;
%Straight Lines [id:da13306667181949583] 
\draw    (285.89,43.2) -- (241,116) ;
%Straight Lines [id:da6300284992572426] 
\draw    (341.04,113.63) -- (351.36,111.27) ;
%Straight Lines [id:da3392847642540684] 
\draw    (241,206) -- (285.89,190.2) ;
%Straight Lines [id:da533363562198773] 
\draw    (241,116) -- (315.5,116) ;
%Straight Lines [id:da2931883424385523] 
\draw    (352.36,82.27) -- (341.54,99.13) ;
%Straight Lines [id:da16098168886509545] 
\draw    (285.89,43.2) -- (352.36,82.27) ;
%Straight Lines [id:da056288319845834045] 
\draw    (352.36,82.27) -- (351.36,111.27) ;
%Straight Lines [id:da06103831913705737] 
\draw    (351.36,111.27) -- (285.89,190.2) ;
%Straight Lines [id:da9089495140040292] 
\draw  [dash pattern={on 4.5pt off 4.5pt}]  (285.89,43.2) -- (285.89,190.2) ;
%Shape: Polygon [id:ds6778178157732934] 
\draw  [fill={rgb, 255:red, 38; green, 7; blue, 241 }  ,fill opacity=0.3 ] (341.54,99.13) -- (341.04,113.63) -- (341.04,113.63) -- (315.5,131) -- (315.5,116) -- cycle ;
%Straight Lines [id:da40782261845796597] 
\draw    (479.5,130) -- (494.83,115) ;
%Straight Lines [id:da4822494790624666] 
\draw    (494.72,115) -- (505.04,112.63) ;
%Straight Lines [id:da018735650972447115] 
\draw    (479.5,115) -- (494.72,115) ;
%Straight Lines [id:da38080467695343656] 
\draw    (505.54,98.13) -- (494.72,115) ;
\draw [shift={(494.72,115)}, rotate = 122.69] [color={rgb, 255:red, 0; green, 0; blue, 0 }  ][fill={rgb, 255:red, 0; green, 0; blue, 0 }  ][line width=0.75]      (0, 0) circle [x radius= 1.34, y radius= 1.34]   ;
%Shape: Polygon [id:ds5723571511190083] 
\draw  [fill={rgb, 255:red, 38; green, 7; blue, 241 }  ,fill opacity=0.3 ] (505.54,98.13) -- (505.04,112.63) -- (505.04,112.63) -- (479.5,130) -- (479.5,115) -- cycle ;
%Shape: Polygon [id:ds8404820265905972] 
\draw  [fill={rgb, 255:red, 38; green, 7; blue, 241 }  ,fill opacity=0.3 ] (429.54,99.13) -- (429.04,113.63) -- (429.04,113.63) -- (403.5,131) -- (403.5,116) -- cycle ;
%Straight Lines [id:da30842066092349496] 
\draw    (164.06,131.84) -- (208.06,131.84) ;
\draw [shift={(210.06,131.84)}, rotate = 180] [color={rgb, 255:red, 0; green, 0; blue, 0 }  ][line width=0.75]    (10.93,-3.29) .. controls (6.95,-1.4) and (3.31,-0.3) .. (0,0) .. controls (3.31,0.3) and (6.95,1.4) .. (10.93,3.29)   ;

% Text Node
\draw (262,209.9) node [anchor=north west][inner sep=0.75pt]    {$M_{+}$};
% Text Node
\draw (491,209.9) node [anchor=north west][inner sep=0.75pt]    {$M_{-}$};
% Text Node
\draw (359,210.4) node [anchor=north west][inner sep=0.75pt]    {$Y=Z/S^{1}$};
% Text Node
\draw (118,89.4) node [anchor=north west][inner sep=0.75pt]    {$Z$};
\end{tikzpicture}
}
    \caption{The pieces $M_+, Y, M_-$. }
    \label{fig:broken manifold}
\end{figure}

Given the following data: 
\begin{itemize}
    \item a broken manifold $\mathfrak{X}$; 
    \item a Lagrangian torus fiber $L$ in $M_+$;
    \item a generic cone direction $\eta_{gen}\in \mathfrak{t}\cong \R^3$;  
    \item a system of
perturbations $\underline{\mathfrak{p}} = \left\{ \left( J_{\tilde{\Gamma}}, F_{\tilde{\Gamma}}\right)\right\}_{\tilde{\Gamma}}$, 
where, to each type $\tilde{\Gamma}$ of a deformed treed holomorphic disc, we associate
a domain-dependent almost complex structure $J_{\tilde{\Gamma}}$ and a Morse function $F_{\tilde{\Gamma}}$ on the Lagrangian intersections, compatible with  breaking of gradient lines, bubbling off holomorphic spheres or disks, weights of edges going to $0$ or $1$, and forgetting boundary input edges; and
\item a local system $\rho \in \Hom(\pi_1(L),U_{\Lambda})$;
\end{itemize} 
one can associate a
\textit{split $A_{\infty}$ algebra} structure 
$\{ \m_k\}$ on 
  \[ CF_{split}(L, \underline{\mathfrak{p}}) := CF^{geom}(L,\underline{\mathfrak{p}}) \,\oplus\, \Lambda_0\,
    x^{\greyt}[1] \,\oplus\, \Lambda_0\, x^{\whitet}. \] 

Here $CF^{geom}(L,\underline{\mathfrak{p}})$ is the vector space generated by the set of critical points of the Morse function on $L$ over $\Lambda_0$. 
The $A_{\infty}$ algebra structure is given
by counting rigid split treed disks in $\mathfrak{X}$ with boundary on $L$. 
Here a split treed disk is a degeneration of a broken treed disk along the generic cone condition, as follows.  
A \textit{broken} treed disk consists of surface components mapping to $M_+$ and $M_-$ satisfying pseudoholomorphic equations, along with treed components mapping to $L$ satisfying Morse gradient-flow equations, such that the evaluation maps satisfy the matching condition at the nodes connecting the surface parts in $M_+$ and $M_-$. 
A split treed disk is obtained by shifting the matching condition by $\lambda \eta_{gen}$ and taking the limit $\lambda\to \infty$. 

By \cite{VWSplit} Theorem 1.1, the split $A_{\infty}$ algebra $ CF_{split}(L, \underline{\mathfrak{p}})$ is $A_{\infty}$-homotopy equivalent to the broken $A_{\infty}$ algebra $ CF_{brok}(M,L)$, which is in turn $A_{\infty}$-homotopy equivalent to the ordinary $A_{\infty}$ algebra $CF(L)$ by \cite{VW} Theorem 1.6. 

We refer to \cite{VWSplit} Section 5.2 for a detailed description of the $A_{\infty}$ algebra. 
By \cite{VW} Section 10.3, $x^{\whitet}$ is a strict unit of $CF(L,\underline{\mathfrak{p}})$. 

\section{The leading order potential}
\label{section potential}
In this section, we prove Proposition \ref{prop leading order potential}, which gives the leading order potential for each manifold $M$ constructed in Section \ref{section Construction of Woodward's examples}, using the methods of \cite{CharestWoodward}, \cite{VW}, \cite{VWSplit}, and \cite{FOOOI}.

\begin{prop}[Leading order potential of $M$]
\label{prop leading order potential}
Let $M$ be a manifold of the form \autoref{defn of M} with parameters $r,\lambda_1,\lambda_2,\lambda_3, n$ satisfying
\autoref{n>2},
\autoref{relations between lambda and r}, and \autoref{cannot cut the S3 vertex}. 
Let $L=\Phi^{-1}(u)$ for some $u=(u_1,u_2,u_3)$ in the interior of $\Delta$ equipped with a relatively spin structure and a local system $\tau\in \Hom \lrp{\pi_1(L),  U_{\Lambda}}$. 
Then the corresponding split $A_{\infty}$ algebra associated to $L$ is weakly unobstructed. 
Moreover, after identifying $\bigcup\limits_{u\in \interior \Delta} \Hom \lrp{ \pi_1(\Phi^{-1}(u)), U_{\Lambda}}$   
as a subset of $\Lambda^3$ via 
\begin{equation}
\label{coordinates on union of local systems}
 (u, \tau ) 
\mapsto (x,y,z) =\lrp{(\exp \tau )(e_1^*)T^{u_1}, (\exp \tau )(e_2^*)T^{u_2}, (\exp \tau )(e_3^*)T^{u_3}},    
\end{equation} 
the leading order potential function takes the form 
\begin{align}
\label{leading potential}
 P 
 & = yT^{-\lambda_3}+y^{-1}T^{\lambda_2} + xz^{-1}  + y^{-1}z + xT^{-\lambda_2} + x^{-1}y^{-n}T^{r+n\lambda_3}. 
\end{align}
\end{prop}
The proof of Proposition \ref{prop leading order potential} is based on the study \cite{CharestWoodward, VW, VWSplit} of the behavior of pseudoholomorphic curves under symplectic cuts \cite{Cuts}, which is a generalization of relevant study (e.g. \cite{BEHWZ}, \cite{IonelParkerRelativeGW}) in Gromov-Witten theory and symplectic field theory.  

We follow this approach to prove Proposition \ref{prop leading order potential}. 

A rigid split disk in $\mathfrak{X}$ with boundary on $L$ has to take one of the following forms.
\begin{enumerate}[1)]
    \item \label{type 1} The rigid split disk is a nodal treed holomorphic disk in $M_+\setminus Y$. 
    \item \label{type 2} The rigid split disk is obtained from a broken disk $(u_+,u_-)$ by deforming an edge matching condition along the cutting hypersurface.  
    Here $u_+, u_-$ in the broken disk data are pseudoholomorphic curves in $M_+,M_-$, respectively, such that $u_+,u_-$ satisfy at least one matching condition  at a node on $Y$, and the irreducible components of $u_-$ are spheres.
\end{enumerate} 

\begin{lem}\label{lem no rigid spheres in M-}
 A rigid split disk of type \ref{type 2} does not contribute to the potential function of $M$.    
\end{lem}

\begin{proof} 
Let $A=[u_-]\in H_2(M_-)$. Since the hypersurface $Y$ of $M_-$ is a hyperplane section, $ c_1(TM_-) =3PD_{M_-}[Y]$ according to Lemma \ref{M- is a quadric} below.

Then $u_-$ represents an element of the moduli space $\mathcal{M}^Y(M_-, A, s,l)$ of nodal pseudoholomorphic spheres in $M_-$ of class $A$, with no marked points and $l$ points of tangency with $Y$ at fixed points on $Y$ such that the total order of tangency is $s$. 
Since $\mathcal{M}^Y(M_-, A, s,l)$ has virtual dimension 
\begin{align*}
  &   2c_1(A) + 2\cdot 3 - 6 + 2l-(6l + 2s - 2l)\\ 
  = & 2 \cdot 3s +4l-2s-6l \geq 2l >0,  
\end{align*}
such a disk cannot be rigid and does not contribute to $P$.    
\end{proof}

\begin{lem}[$H^1(L;\Lambda_0)$ are weak Maurer-Cartan]
\label{lem weak Maurer-Cartan}
Every local system $\rho \in \Hom \lrp{\pi_1(L),  U_{\Lambda}}$ defines a weakly unobstructed split $A_{\infty}$ structure on $L$. 
\end{lem}

\begin{proof} The argument is similar to the proof of \cite{VWSplit} Proposition 7.3. 

Since $L$ is a Lagrangian torus fiber of the toric manifold $M_+$ equipped with a Hamiltonian $T$-action, we have that $L\subset M_+$ is a tropical torus by \cite{VWSplit} Definition 7.1. 
Let $\rho \in \Hom \lrp{\pi_1(L),  U_{\Lambda}}$ be a local system on $L$.

By the proof of \cite{VWSplit} Proposition 7.3, the perturbation datum can be taken so that the almost complex structure on $M_+$ is standard.  
Consider a rigid split disk $[u] \in \Mq_{\widetilde{\Gamma}, \mathrm{red}}^{\mathrm{split}}(L;x_0)$ contributing to $\m_{split}^0(1)$ whose boundary asymptotes to $x_0$. Then $u$ is regular. If we forget the boundary leaf of $u$, we get a split disk $[u']\in \Mq_{\widetilde{\Gamma}', \mathrm{red}}^{\mathrm{split}}(L)$ without output (or input). 

Then $\dim \Mq_{\widetilde{\Gamma}', \mathrm{red}}^{\mathrm{split}}(L)\geq \dim T$, since we can construct a $(\dim T)$-family of reduced split maps $[u_t'] $ by translating the components of $u'$ lying in $M_+$  by $t\in T$ without changing the components in $M_-$.  

Since adding the boundary constraint $x_0$ reduces the expected dimension by the Morse index of $x_0$, 
\[ 0 = \dim \Mq_{\widetilde{\Gamma}, \mathrm{red}}^{\mathrm{split}}(L;x_0) = \dim \Mq_{\widetilde{\Gamma}', \mathrm{red}}^{\mathrm{split}}(L)- i_{\mathrm{Morse}}(x_0). \]
This implies that the Morse index $i_{\mathrm{Morse}}(x_0)$ of $x_0$  satisfies $i_{\mathrm{Morse}}(x_0) =  \dim \Mq_{\widetilde{\Gamma}', \mathrm{red}}^{\mathrm{split}}(L)\geq n= 3$. 
Therefore, $x_0$ has to be the unique maximal critical point $x^{\blackt}$ of the Morse function on $L$.  

Hence,   $\m_{split}^0(1)=w\cdot x^{\blackt}$, where $w$ is in $\Lambda_+$. 
Moreover, if $[u]$ is a positive area rigid disk that contributes to $\m_{split}^1( x^{\greyt})$, then the disk $[u']$ obtained from $u$ by forgetting the input would have index $2$ lower than that of $[u]$. This implies that $[u']$ has negative index, which cannot exist. 
Hence, $\m_{split}^1( x^{\greyt}) = x^{\whitet}-x^{\blackt}$. 

By dimension counting, we also have $\m_{split}^d( x^{\greyt},\ldots, x^{\greyt})=0$ for all $d\geq 2$. 
It follows that $\sum\limits_{d=0}^{\infty} \m_{split}^d
(wx^{\greyt}) = w x^{\whitet}$. In particular, $wx^{\greyt}$ is a weak Maurer-Cartan solution for the split $A_{\infty}$ algebra. Therefore, we obtain the weak unobstructedness of $L$.  
\end{proof}

\begin{proof}[Proof of Proposition \ref{prop leading order potential}]
Let $L = \mu^{-1}(u_1,u_2,u_3)$ and $\rho\in \Hom(\pi_1(L), U_{\Lambda})\cong H^1(L, U_{\Lambda})$ be a local system.
The value of the potential function $W$ at a Lagrangian brane $(L,\rho)$ is defined to be the $w$ associated to $(L,\rho)$ in Lemma \ref{lem weak Maurer-Cartan}. 
 
Let 
\[ x = e^{x_1}T^{u_1}, \quad y = e^{x_2}T^{u_2}, \quad z = e^{x_3}T^{u_3}. \]
Then $x,y,z\in \Lambda_0$. 
With the change of coordinates, we regard the potential function as a map
$W: \bigcup\limits_{u\in \interior \Delta} H^1\lrp{\Phi^{-1}(u), \Lambda_0/(2\pi i \Z)} \to \Lambda$.   
By Lemma \ref{cannot cut the S3 vertex}, the split $A_{\infty}$ algebra counts rigid disks which by Lemma \ref{lem no rigid spheres in M-} has to be of type \ref{type 1}.  By \cite{VWSplit} Proposition 7.5, $M_+\setminus Y$, each prime toric boundary divisor $D_i$ of $M_+\setminus Y$ corresponds to a type \ref{type 1} rigid holomorphic disk $u_i$ of Maslov index two and contributes 
\[ P= yT^{-\lambda_3}+y^{-1}T^{\lambda_2} + xz^{-1}  + y^{-1}z + xT^{-\lambda_2} + x^{-1}y^{-n}T^{r+n\lambda_3} \]
to the potential function $W$. 
Moreover, since the areas of other split disks are larger than at least one of these basic disks and the inward normal vectors generate $H^1(L(u);\Z)$ for all $u\in \interior \Delta$, we conclude that $P$ is the leading order potential of $M$.  
\end{proof}
\begin{lem}
    \label{M- is a quadric}
    $M_-$ is symplectomorphic to a quadric threefold. 
\end{lem}

\begin{proof} 
Note that $M_-$ is obtained from the multiplicity-free, torsion-free, and transversal $U(2)$-manifold\footnote{We refer the reader to \cite{WoodwardClassification} for the definitions of ``torsion-free" and ''transversal". } $M$ by the symplectic cut as in \cite{WoodwardClassification} Theorem 8.3 such that the $U(2)$-Kirwan polytope of $M_-$ is the clean subpolytope (as defined on page 30 of \cite{WoodwardClassification})
\[ \Delta_-=\operatorname{conv}\{(\lambda_2, \lambda_2), (\lambda_2+\epsilon, \lambda_2), (\lambda_2, \lambda_2-\epsilon)\}
\] 
of the Kirwan polytope of $M$. 
Hence, $M_-$ is a multiplicity-free, torsion-free, and transversal $U(2)$-manifold by \cite{WoodwardClassification} Theorem 8.3. 
\begin{figure}
 \begin{center}
  
\tikzset{every picture/.style={line width=0.75pt}} %set default line width to 0.75pt  
\resizebox{0.45\textwidth}{!}{     
\begin{tikzpicture}[x=0.75pt,y=0.75pt,yscale=-1,xscale=1]
%uncomment if require: \path (0,275); %set diagram left start at 0, and has height of 275

%Straight Lines [id:da8815415933271029] 
\draw    (294.39,139.4) -- (405.61,30.6) ;
\draw [shift={(407.75,28.5)}, rotate = 135.63] [fill={rgb, 255:red, 0; green, 0; blue, 0 }  ][line width=0.08]  [draw opacity=0] (5.36,-2.57) -- (0,0) -- (5.36,2.57) -- cycle    ;
\draw [shift={(292.25,141.5)}, rotate = 315.63] [fill={rgb, 255:red, 0; green, 0; blue, 0 }  ][line width=0.08]  [draw opacity=0] (5.36,-2.57) -- (0,0) -- (5.36,2.57) -- cycle    ;
%Shape: Axis 2D [id:dp587501501765215] 
\draw  (156,224) -- (192,224)(159.6,191.6) -- (159.6,227.6) (185,219) -- (192,224) -- (185,229) (154.6,198.6) -- (159.6,191.6) -- (164.6,198.6)  ;
%Straight Lines [id:da6839358209768817] 
\draw    (159.6,224) -- (179.25,202.48) ;
\draw [shift={(180.6,201)}, rotate = 132.4] [color={rgb, 255:red, 0; green, 0; blue, 0 }  ][line width=0.75]    (10.93,-3.29) .. controls (6.95,-1.4) and (3.31,-0.3) .. (0,0) .. controls (3.31,0.3) and (6.95,1.4) .. (10.93,3.29)   ;
%Shape: Axis 2D [id:dp89756392840526] 
\draw  (333,227) -- (369,227)(336.6,194.6) -- (336.6,230.6) (362,222) -- (369,227) -- (362,232) (331.6,201.6) -- (336.6,194.6) -- (341.6,201.6)  ;
%Straight Lines [id:da5742211121899803] 
\draw    (350,85) -- (350,130) ;
%Straight Lines [id:da49977514027832914] 
\draw    (350,85) -- (396.72,85.32) ;
%Straight Lines [id:da15936048033078376] 
\draw    (396.72,85.32) -- (350,130) ;
%Straight Lines [id:da757860589608266] 
\draw    (172.5,128) -- (217.84,97.94) ;
%Straight Lines [id:da8628190599982458] 
\draw    (217.5,97.94) -- (248.02,93.19) ;
%Straight Lines [id:da40799672411632426] 
\draw    (172.5,97.94) -- (217.5,97.94) ;
%Straight Lines [id:da22720240047142504] 
\draw    (249.5,64.13) -- (217.5,97.94) ;
%Shape: Polygon [id:ds8782068790824421] 
\draw   (249.5,64.13) -- (248.02,93.19) -- (248.02,93.19) -- (172.5,128) -- (172.5,97.94) -- cycle ;

% Text Node
\draw (183.6,189.4) node [anchor=north west][inner sep=0.75pt]  [xscale=0.8,yscale=0.8]  {$u_{1}$};
% Text Node
\draw (197.6,214.4) node [anchor=north west][inner sep=0.75pt]  [xscale=0.8,yscale=0.8]  {$u_{2}$};
% Text Node
\draw (157.6,169.4) node [anchor=north west][inner sep=0.75pt]  [xscale=0.8,yscale=0.8]  {$u_{3} \ $};
% Text Node
\draw (373.6,220.4) node [anchor=north west][inner sep=0.75pt]  [xscale=0.8,yscale=0.8]  {$u_{1}$};
% Text Node
\draw (340.6,185.4) node [anchor=north west][inner sep=0.75pt]  [xscale=0.8,yscale=0.8]  {$u_{2} \ $};

\end{tikzpicture}
}  
\end{center}
 
    \caption{$M_-$ and its $U(2)$-Kirwan polytope}
    \label{fig: Kirwan polytope of M+}
\end{figure}
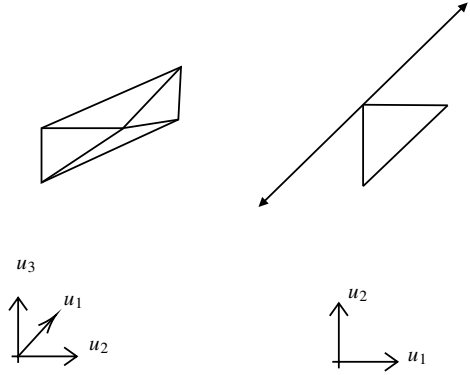

We now identify another multiplicity-free, torsion-free, and transversal $U(2)$-manifold $Q$ with the same Kirwan polytope as $M_-$, by extending the arguments in \cite{WoodwardClassification} Theorem 8.5. 
Let us identify $U(1)$ with 
\[ \set{\begin{pmatrix}
    \lambda & 0 \\
    0 & \lambda^{-1}
\end{pmatrix}}{\lambda\in U(1)}\subset SU(2)  \]
and $T^*SU(2)$ with $SU(2)\times \mathfrak{su}(2)^*$ via $(B,\xi)\mapsto (B, dL_B|_{e}(\xi))$, where $\xi\in T_B^*SU(2)$, and $L_B$ is the left multiplication by $B$. 

There is a Lie group isomorphism $U(2)\cong (SU(2)\times U(1))/\Z_2$ induced by 
\[
SU(2)\times U(1) \to U(2), 
\quad  
\left(A_0, \begin{pmatrix}
    \lambda & 0 \\
    0 & \lambda^{-1}
\end{pmatrix}\right) \mapsto \lambda^2 A_0.    
\] 
We identify the positive Weyl chamber of the Cartan subalgebra $\mathfrak t^*$ of $U(2)$ with
\[
\mathfrak t_+^* \;\cong\; (\mathbb{R}_{\ge 0})^2,
\]
the positive Weyl chamber of the Cartan subalgebra $\mathfrak t_{\mathfrak{su}(2)}^*$ of $\mathfrak{su}(2)$ with
\[
\mathfrak t_{\mathfrak{su}(2),+}^* \;\cong\; \mathbb{R}_{\ge 0}\,(1,-1),
\]
and the positive Weyl chamber of the Cartan subalgebra $\mathfrak t_{\mathfrak{u}(1)}^*$ of $\mathfrak{u}(1)$ with
\[
\mathfrak t_{\mathfrak{u}(1),+}^* \;\cong\; \mathbb{R}_{\ge 0}\,(1,1).
\]
Here we use the standard identification $\mathfrak t^* \cong \mathbb{R}^2$ coming from the diagonal Cartan subalgebra of $U(2)$.

Following \cite{WoodwardClassification} proof of Theorem 8.5, we define the \textit{left} action of $SU(2)$ on $SU(2)\times \mathfrak{su}(2)^*$ by 
\[ A\cdot (B,v) = (A B, \Ad_A^* v) \] 
and the \textit{right} action of $U(1)\subset SU(2)$ on $SU(2)\times \mathfrak{su}(2)^*$ by 
\[ C\cdot (B,v) = (BC, \Ad_{C^{-1}}^*(v)).  
\] 
We may take the moment maps for the left and right actions to be 
 
\[ 
\Phi_L: SU(2)\times \mathfrak{su}(2)^* \to \mathfrak{su}(2)^*, \qquad   (B,v)\mapsto \Ad_{B}^*(v)  
\]
and 
\[ 
\Phi_R: U(1)\times \mathfrak{su}(2)^* \to \mathfrak{u}(1)^*, \qquad   (B,v)\mapsto \pi(v)+(\lambda_2,\lambda_2),  \]
respectively, 
where $\pi: \mathfrak{su}(2)^*\to \mathfrak{u}(1)^*$ is the transpose of the inclusion map $U(1)\to SU(2)$.

The product action induces an action of $U(2)$ on $T^*SU(2)$ with a moment map 
\[
\Phi_{prod}: SU(2)\times \mathfrak{su}(2)^* \to \mathfrak{su}(2)^* \oplus \mathfrak{u}(1)^* \cong \mathfrak{u}(2)^*, \qquad   (B,v)\mapsto  (\Ad_{B}^*(v) , \pi(v)+\lambda_2(1,1) ) . 
\]
Fix $a\geq 0$. 
Then we have 
\begin{align*}
 & \Phi_{prod}(SU(2)\times \mathfrak{su}(2)^*)\cap \mathfrak{t}_+^*  \\
 = & \left \{ (w+(\lambda_2,\lambda_2), a) \in \left(\mathfrak{t}_{\mathfrak{u}(1), +}^*+(\lambda_2,\lambda_2)\right)  \oplus\mathfrak{t}_{\mathfrak{su}(2), +}^*  \;\middle|\; w \in \pi(\Ad_{SU(2)}^*(a) 
\right \}.       
\end{align*}

Then the $U(2)$-Kirwan polytope of $T^*SU(2)$ becomes the lower right quadrant translated by $(\lambda_2, \lambda_2)$: 
\[ \{ w\cdot (1,1)+a(1,-1) \mid a \geq 0, |w|\leq a\} + (\lambda_2, \lambda_2). \]

The symplectic cut $Q$ of $T^*SU(2)$ along $\{a=\epsilon/\sqrt{2}\}$ has $\Delta_-$ as its Kirwan polytope. 
Since the $U(2)$-action on $T^*SU(2)$ is a multiplicity-free, torsion-free, transversal proper Hamiltonian action, it induces a $U(2)$-action on $Q$ with the same properties by  \cite{WoodwardClassification} proof of Theorem 8.3. 

By Woodward's classification theorem (\cite{WoodwardClassification} Theorem 1.3), $M_-$ is symplectomorphic to $Q$, and hence to a quadric hypersurface $Q^3$ in $\proj^4$. (See for example \cite{Tehrani} Section 2.1.)
\end{proof}

\section{Critical points of the potential function}
\label{section Critical points of the potential}
In the study of Lagrangian Floer theory, a central question is whether one can provide a collection of generators for the Fukaya category of a given symplectic manifold. 
In many important cases, such as the compact toric Fano manifolds \cite{FOOOI}, \cite{EvansLekili}, such a collection of generators is often determined by the critical points of the potential function.  
In this section, we apply the methods in \cite{FOOOI} to identify a Lagrangian with non-trivial Floer cohomology by considering the critical points of the potential function.

\begin{prop}[Critical points of $W$]
\label{Proposition critical points of po}
Let $M$ be a manifold of the form \autoref{defn of M} with parameters $r,\lambda_1,\lambda_2,\lambda_3, n$ satisfying 
\autoref{n>2},
\autoref{relations between lambda and r}, and $\dfrac{\lam_2-\lam_3}{r-\lam_2}<\dfrac{1}{n+1}$. Then there are $6$ critical points of $W$ with valuation 
\begin{equation}
      u_0= \lrp{\frac{4r-(2n-1)\lam_2 + (2n+1)\lam_3}{6}, \frac{\lam_2+\lam_3}{2} , \frac{2r+(2-n)\lam_2+(n+2)\lam_3}{6}}.   
\label{valuation of crit points}
\end{equation} 
\end{prop} 
\begin{proof}
To simplify the notation, let 
\[ a = T^{\lam_1}, \quad b =T^{\lam_2}, \quad c= T^{\lam_3}, 
\quad 
d = T^{ r + n\lam_3}.  
\] 
By Proposition \ref{prop leading order potential},
\begin{align*} 
 P 
 & = yT^{-\lambda_3}+y^{-1}T^{\lambda_2} + xz^{-1}  + y^{-1}z + xT^{-\lambda_2} + x^{-1}y^{-n}T^{r+n\lambda_3} \\ 
 & = \frac{y}{c}+\frac{b}{y} + \frac{x}{z} + \frac{z}{y} + \frac{x}{b} + dx^{-1}y^{-n}.   \nonumber
\end{align*}
A critical point of $P$ is a solution of the following system of equations. 
\begin{equation}
\begin{dcases}
  0 = x\frac{\partial P}{\partial x}  = \frac{x}{z} +  \frac{x}{b} - dx^{-1}y^{-n} = f_1 \\
  0 =  y\frac{\partial P}{\partial y}  = \frac{y}{c} - \frac{b}{y} - \frac{z}{y} - ndx^{-1}y^{-n} = f_2 \\
  0 =   z\frac{\partial P}{\partial z}  = - \frac{x}{z} + \frac{z}{y}  = f_3
\end{dcases}    
\label{leading order crit equation}
\end{equation}
From $f_3=0$, we have 
\begin{equation}
 \label{third equation}
   xy=z^2. 
\end{equation} 
From 
$bzxy^n\cdot f_1=0$,
we have 
\[ 
0 = bx^2y^n + x^2y^nz-bdz = z\left( y^{n-2} z^3(z+b)-bd\right)
\]
If $0=z+b=z+T^{\lam_2}$, then $bd=0$, which never holds. 
Since $z\ne 0$ and $z+b\ne 0$, 
we must have 
\begin{equation}
\label{y in terms of z}
y^{n-2} = \frac{bd}{z^3(z+b)} 
\end{equation}
This implies that 
\[ 
    u_2 = \val(y) = \frac{1}{n-2}\left(\lam_2+r+n\lam_3 - 3\val(z)-\val(z+T^{\lam_2})\right) 
    \]
From $cxy^{n}\cdot f_2 =0$, we get 
\begin{align*}
0 
& = xy^{n+1} - bcxy^{n-1} - cxy^{n-1}z - ncd \\
% & = xy\cdot y^n -xy\cdot cy^{n-2}\left(z+b\right)-ncd \\
& = z^2 \left(\frac{bd}{z^3(z+b)} \right)^{\frac{n}{n-2}} - cz^2\frac{bd}{z^3(z+b)}(z+b)-ncd \\
& = \left(bd\right)^{\frac{n}{n-2}} \cdot z^{-\frac{n+4}{n-2}} (z+b)^{-\frac{n}{n-2}}-bcdz^{-1}-ncd, \quad 
% \text{since }2-\frac{3n}{n-2} =  -\frac{n+4}{n-2}
\end{align*} 
 
Hence,
% \[ (cd)^{n-2}z^6(z+b)^{n} (nz+b)^{n-2}-(bd)^n=0. \]
\begin{equation} 
0 =z^6(z+b)^n(nz+b)^{n-2} - b^nc^{-(n-2)}d^2=:g.  \label{equation in z}
\end{equation}

The degree of the polynomial is $6+n+n-2 = 2n+4$. And
\begin{align}
  \label{reduced poly in z}
    g 
   & =\sum_{l=0}^{n-2}  \left(\sum_{j=0}^{l} \binom{n}{l-j}\binom{n-2}{j}n^j\right) b^{2n-2-l}z^{l+6} 
   \\ \nonumber 
   & + \sum_{l=n-1}^{2n-2}  \left(\sum_{j=0}^{n-2} \binom{n}{l-j}\binom{n-2}{j}n^j\right) b^{2n-2-l}z^{l+6}  - b^nc^{-(n-2)}d^2 
   \nonumber 
\end{align}
Note that $\val(b^{2n-2-l}) = (2n-2-l)\lam_2$ and 
\[ \val (b^nc^{-(n-2)}d^2) = n\lam_2-(n-2)\lam_3 + 2(r+n\lam_3) = 2r+n\lam_2+(n+2)\lam_3\]
 
Recall that the Newton polygon $\operatorname{Newt}(g)$ of $g=\sum c_u z^u$ is the projection of the lower facets of
\[ P_{\val} = \operatorname{conv}\{ (u,\val(c_u)): c_u\ne 0\} \]
to the $u$-coordinates. 

We use the following theorem to tropicalize $V(g)$. 
\begin{thm}[\cite{koblitz1984p} IV.3 Lemma 4]
\label{thm counting roots}
    Let $K$ be a complete non-archimedean field with valuation $\val: K\to \R\cup \{\infty\}$. 
    Let $f  = \sum c_u x^u \in K[x^{\pm 1}]$. 
    If a segment connecting $(a,\val(c_a))$ and $(b,\val(c_b))$ of the Newton polygon of $f(x)$ has slope $-m \in \R$, then $f$ has precisely $b-a$ roots of valuation $m$. 
\end{thm}

In our case,
\[  P_{\val}  
=    \operatorname{conv} ( \{(0,2r+n\lam_2+(n+2)\lam_3)\} \cup \{(l+6, (2n-2-l)\lam_2)\mid 0\leq l\leq 2n-2\}). \]
Let 
\[ A =(0,2r+n\lam_2+(n+2)\lam_3),\quad B=(6,(2n-2)\lam_2),\quad C=(2n+4,0),   \]
We denote by $\overline{AB}, \overline{BC}, \overline{AC}$ the slopes of $AB$, $BC, AC$, respectively. Then  
\begin{align*}
\overline{AB}& = -\frac{2r+(2-n)\lam_2+(n+2)\lam_3}{6} , \\
\overline{BC}& = -  \lam_2, \\
\overline{AC} & = - \frac{2r+n\lam_2+(n+2)\lam_3}{2n+4}.
\end{align*}
There are three possible configurations of the Newton polygon, corresponding to 
the cases where $ r >, = \text{ or }<  \dfrac{(n+4)\lam_2-(n+2)\lam_3}{2}$, respectively. 
But by \autoref{cannot cut the S3 vertex}, $r$ must satisfy
\begin{equation} 
    \label{r>r0}
    r >  \frac{(n+4)\lam_2-(n+2)\lam_3}{2} \quad \Longleftrightarrow \quad \frac{\lambda_2-\lambda_3}{r-\lambda_2}<\frac{2}{n+2}.  
\end{equation}
Therefore, $\overline{AB}< \overline{BC}$, in which case the Newton polygon has two line segments, one connecting $AB$ and one connecting $BC$. 
    
By Theorem \ref{thm counting roots}, there are two possible values for $u_3$. 
It turns out that the critical points arising in case \ref{invalid critical points} \textit{do not} lie inside the Gelfand--Cetlin polytope.  \begin{enumerate}[(a)]
        \item  \label{valid critical points}  Equation \eqref{equation in z} has $6$ roots of valuation 
        \[\val(z)=u_3= \dfrac{2r+(2-n)\lam_2+(n+2)\lam_3}{6}, \] which is $> \lam_2$, making $\val(z+b) = \lam_2$. 
         By \eqref{y in terms of z}, 
    \[
     u_2 
     = \frac{1}{n-2}\left( \lam_2+r+n\lam_3- \frac{2r+(2-n)\lam_2+(n+2)\lam_3}{2}- \lam_2 \right) =  \frac{\lam_2+\lam_3}{2},
    \]
  
    \[ u_1= \frac{2r+(2-n)\lam_2+(n+2)\lam_3}{3}- \frac{\lam_2+\lam_3}{2} = \frac{4r-(2n-1)\lam_2 + (2n+1)\lam_3}{6}.\]
    These correspond to solutions to \eqref{leading order crit equation}  with valuation 
    { \fontsize{10}{11}
    \begin{align}
      & (u_1,u_2,u_3) \nonumber \\
     = & \lrp{\frac{4r-(2n-1)\lam_2 + (2n+1)\lam_3}{6}, \frac{\lam_2+\lam_3}{2} , \frac{2r+(2-n)\lam_2+(n+2)\lam_3}{6}}.   
\label{valuations/Lagrangian 1a}    
    \end{align} 
        }
    Consider \[ \text{trop} P = \min \{ u_2-\lambda_3, \lambda_2-u_2, u_1-u_3, u_3-u_2, u_1-\lambda_2, r+n\lambda_3-u_1-nu_2\}. \]
    For \autoref{valuations/Lagrangian 1a},
    \begin{align*}
      u_2 -\lambda_3 
      &  = \lambda_2-u_2 = \frac{\lambda_2-\lambda_3}{2} = :S_1,  \\
      u_1-u_3 
      & =u_3-u_2=r+n\lambda_3-u_1-nu_2  = \frac{2r-(n+1)\lambda_2 + (n-1)\lambda_3}{6}=: S_2\\
      u_1-\lambda_2 & = \frac{4r-(2n+5)\lambda_2 + (2n+1)\lambda_3}{6}=: S_3. 
    \end{align*}
    Then $0<S_1<S_2<S_3$. In particular, \autoref{valuations/Lagrangian 1a} lies in the Gelfand--Cetlin polytope of $M$. 

    To check the non-degeneracy of the corresponding critical points, we only need to check the stage-wise non-degeneracy (as in \cite{GonzalezWoodward}), or equivalently the solutions to the leading term equation are weakly non-degenerate (as in \cite{FOOOI}). 
    \begin{defn}[Stage-wise non-degenerate critical point]
         Let $\tilde{T}^{\vee} = \Hom(M, \Lambda^*)$. 
     Suppose the following holds. 
     \begin{enumerate}
         \item   \[ \tilde{T}^{\vee} = \tilde{T}_1^{\vee} \times \cdots \times \tilde{T}_K^{\vee},\]
         where each $\tilde{T}_i^{\vee}$ is a product of a torus with a finite group $\Gamma_i$. 
          \item $P: \tilde{T}^{\vee} \to \Lambda$ is a function such that $P=W_1+W_2+\cdots +W_K$ , where $W_i: \tilde{T}^{\vee}\to \prod\limits_{j=1}^i\tilde{T}_j^{\vee}$ factors through the projection onto the first $i$ factors.  
         \item $\val(W_i) = S_i$ for each $i$ and $S_1<\cdots <S_K$. 
         \item $y=(y_1,\ldots, y_K)\in \prod\limits_{j=1}^K\tilde{T}_j^{\vee}$ such that, for $i\geq 2$, $y_i$ is a non-degenerate critical point of  
         \[ \tilde{W_i} := W_{i}|_{\tilde{T}_i^{\vee}}: \tilde{T}_i^{\vee}\to \Lambda. \]
     \end{enumerate}
     Then we say $y$ is a \textit{stage-wise non-degenerate critical point} of $P$. 
    \end{defn}
    \begin{prop}[Stage-wise non-degenerate critical points, \cite{GonzalezWoodward} Proposition 4.13, \cite{FOOOI} Theorem 10.4]  
    Suppose $P: \tilde{T}^{\vee} \to \Lambda$ is a function such that every critical point $y\in \operatorname{Crit}_+(P)$ is  stage-wise non-degenerate as in the definition above. 
    For any $F$, let $\operatorname{Crit}_+(F)$ denote the set of critical points with positive valuation.  
    Then there exists a bijection 
    \[ \prod_{i=1}^{K}\operatorname{Crit}_+ (\tilde{W_i}) \to \operatorname{Crit}_+(W)  
    \]
    for sufficiently small $S_1,\ldots, S_K$. 
    \end{prop}
    In our case, we only need to consider the first two stages. 
    Stage 1: 
    \[ \widehat{W}_1= (y+y^{-1})T^{\frac
    {\lambda_2-\lambda_3}{2}} = W_1T^{S_1}. \]
    Stage 2: 
    \[ \widehat{W}_2= (xz^{-1} + zy^{-1}+ x^{-1}y^{-n} )T^{ \frac{2r-(n+1)\lambda_2 + (n-1)\lambda_3}{6}} = W_2T^{ S_2}. \]
    
The corresponding leading term equations are 
\[
\begin{dcases}
 0=  \frac{\partial W_1}{\partial y} = 1-\frac{1}{y^2}\\
 0 = \frac{\partial W_2}{\partial x} =  \frac{1}{z}-\frac{1}{x^2y^n}\\
 0 = \frac{\partial W_2}{\partial z} = -\frac{x}{z^2}+\frac{1}{y}
\end{dcases} 
\qar \begin{dcases}
 0=  y^2-1\\
 0 = x^2y^n-z \\
 0 = z^2-xy
\end{dcases}
\] 
The determinant of the stage 1 Hessian is $\Hess_1 = 2y^{-3}$. 
And the stage 2 Hessian is given by the following
\[ 
\Hess_2 
= \det 
\begin{pmatrix}
\dfrac{2}{x^3} 
& 
-\dfrac{1}{z^2}
\\
-\dfrac{1}{z^2} 
& 
\dfrac{2x}{z^3}
\end{pmatrix}
 = 
 4x^{-2}z^{-3}-z^{-4}
\]

Let $\zeta = e^{\frac{2\pi i}{3}
}$. 
\begin{enumerate}[(i)]
    \item 
For an odd $n$,  
we list the 6 solutions to the leading term equations and the corresponding determinants of the Hessians at stage $1$ and $2$.  
 
\[
\begin{array}{cccccc}
x_0 & y_0 & z_0 & \Hess_1 & \Hess_2 &P(x_0,y_0,z_0) \\
\hline
1 & 1 & 1 & 2 & 3 & w_1= 2T^{S_1}+3T^{S_2}+T^{S_3} \\
\zeta & 1  & \zeta^2 & 2 & 3\zeta & w_2= 2T^{S_1}+3\zeta^2T^{S_2}+\zeta T^{S_3} \\
\zeta^2 & 1 & \zeta & 2 & 3\zeta^{2} & w_3 = 2T^{S_1}+3\zeta T^{S_2}+\zeta^2T^{S_3} \\
-1 & -1 & -1 & -2 & -5 & -w_1 \\
-\zeta^2 & -1 & -\zeta & -2 & -5\zeta^{-1} & -w_3 \\
-\zeta & -1 & -\zeta^2 & -2 & -5\zeta^{-2} & -w_2
\end{array}
\] 
\item 
For an even $n$,  we list the 6 solutions to the leading term equations and the corresponding determinants of the Hessians at stage $1$ and $2$. 
 
\[
\begin{array}{cccccc}
x_0 & y_0 & z_0 & \Hess_1 & \Hess_2 & P(x_0,y_0,z_0) \\
\hline 
1 & 1 & 1 & 2 & 3 & w_1 \\
\zeta & 1  & \zeta^2 & 2 & 3\zeta & w_2 \\
\zeta^2 & 1 & \zeta & 2 & 3\zeta^{2} & w_3\\
 -1 & -1 & 1 & -2 & 3 & -w_1\\
-\zeta^{2} & -1 & \zeta & -2 & 4\zeta-\zeta^{-1} & -w_3 \\
 -\zeta & -1 & \zeta^2 & -2 & 4\zeta^2-\zeta & -w_2
\end{array}
\]
\end{enumerate}
 
\item 
\label{invalid critical points}
$g$ has $2n-2$ roots of valuation $u_3=\lam_2$. 
There are two cases. 
\begin{enumerate}[{Case }1.] 
    \item 
   $\val(z+b)>\lam_2$, which happens if $z=-T^{\lambda_2}+$higher valuation terms. By \autoref{equation in z},
    \begin{align*}
      \val(z+b) 
      & = \frac{1}{n}\lrp{\val (b^nc^{2-n}d^2) -6\val(z)-(n-2)\val(nz+b)} \\
      & = \frac{1}{n}\lrp{2r+n\lambda_2+(n+2)\lambda_3-6\lambda_2-(n-2)\lambda_2} \\
      & = \frac{2r-4\lambda_2+(n+2)\lambda_3}{n},  
    \end{align*}
    making 
   \begin{align*}
     u_2 
     = \val(y) & = \frac{1}{n-2}\left(\lam_2+r+n\lam_3- 3\val(z)-\val(z +b)\right) \\
   & = \frac{1}{n-2}\left( \lam_2+r+n\lam_3-  3\lam_2 - \frac{2r-4\lambda_2+(n+2)\lambda_3}{n} \right)\\
   &= \frac{r-2\lam_2+(n+1)\lam_3}{n}
    \end{align*} 
     and 
    \[ u_1= 2\lam_2-  \frac{r-2\lam_2+(n+1)\lam_3}{n} =  \frac{-r+(2n+2)\lam_2-(n+1)\lam_3}{n}.\]
  The corresponding solutions to the system \eqref{equation in z} have valuation  \begin{equation}
    (u_1,u_2,u_3)= \lrp{\frac{-r+(2n+2)\lam_2-(n+1)\lam_3}{n}, \frac{r-2\lam_2+(n+1)\lam_3}{n}, \lam_2  }. \label{valuations/Lagrangian 1bi}   
    \end{equation}
     Consider \[ \text{trop} P = \min \{ u_2-\lambda_3, \lambda_2-u_2, u_1-u_3, u_3-u_2, u_1-\lambda_2, r+n\lambda_3-u_1-nu_2\}. \]     
\begin{align*}
 u_2-\lambda_3= & r+n\lambda_3-u_1-nu_2  = \frac{r-2\lam_2+\lam_3}{n}   =:A \\
\lambda_2-u_2 = &  u_3-u_2= u_1-u_3 = u_1-\lambda_2 = \frac{-r+(n+2)\lam_2-(n+1)\lam_3}{n} =:B
\end{align*}
Then $A  > B$, since
\[ A-B = \frac{1}{n}\lrp{2r-(n+4)\lam_2+(n+2)\lam_3}>  0.  \]
Moreover,   
\[ B>0 \quad \Longleftrightarrow \quad  -r+(n+2)\lam_2-(n+1)\lam_3 >0 \quad \Longleftrightarrow \quad  \frac{\lam_2-\lam_3}{r-\lam_2}>\frac{1}{n+1}. \]
Thus, if $\dfrac{\lam_2-\lam_3}{r-\lam_2}\leq \dfrac{1}{n+1}$, then
\autoref{valuations/Lagrangian 1bi} lies outside the polytope. \\

If $\dfrac{\lam_2-\lam_3}{r-\lam_2}>  \dfrac{1}{n+1}$, then
\autoref{valuations/Lagrangian 1bi} lies inside the polytope.

\textit{
However, we can always pick $r,\lambda_2,\lambda_3$ such that $\dfrac{\lam_2-\lam_3}{r-\lam_2}<\dfrac{1}{n+1}$.}

Consider 
\[ \widehat{W} = \lrp{\frac{1}{y}+\frac{z}{y}+\frac{x}{z}+x}T^{\frac{-r+(n+2)\lambda_2-(n+1)\lambda_3}{n}} = WT^B.  \]
The leading term equations are
\[ 
\begin{dcases}
 0 = \frac{\partial W}{\partial x} = \frac{1}{z}+1 \\
  0 = \frac{\partial W}{\partial y} = -(z+1)y^{-2} \\ 
  0 = \frac{\partial W}{\partial z} = \frac{1}{y}-\frac{x}{z^2} 
\end{dcases}
\]
The solution set to the leading term equations is 
\[ 
S=\set{\left(x,\frac{1}{x},-1\right) }{x\in \C^{\times}} . 
\]
They are degenerate because
\[
\det \begin{pmatrix}
    0 & 0 & -z^{-2} \\
    0 & 2(z+1)y^{-3} & -y^{-2}\\
    -z^{-2} & -y^{-2} & 2xz^{-3}
\end{pmatrix} = 0  \quad \forall (x,y,z)\in S. 
\] 
\item $\val(z+b)=\lam_2$. 
  Then
    \begin{align*}
     u_2 
     = \val(y) & = \frac{1}{n-2}\left(\lam_2+r+n\lam_3- 3\val(z)-\val(z +b)\right) \\
   & = \frac{1}{n-2}\left( \lam_2+r+n\lam_3-  4\lam_2\right) =  \frac{r-3\lam_2+n\lam_3}{n-2}
    \end{align*}
    and 
    \[ u_1= 2\lam_2-  \frac{r-3\lam_2+n\lam_3}{n-2} = \frac{-r+(2n-1)\lam_2 -n\lam_3}{n-2}.\]
  The corresponding solutions to the system \eqref{equation in z} have valuation  \begin{equation}
    (u_1,u_2,u_3)= \lrp{\frac{-r+(2n-1)\lam_2 -n\lam_3}{n-2}, \frac{r-3\lam_2+n\lam_3}{n-2}, \lam_2  }. \label{valuations/Lagrangian 1bii}     
    \end{equation}
 To check the non-degeneracy of these solutions that may occur, we consider \[ \text{trop} P = \min \{ u_2-\lambda_3, \lambda_2-u_2, u_1-u_3, u_3-u_2, u_1-\lambda_2, r+n\lambda_3-u_1-nu_2\}. \]     
\begin{align*}
 u_2-\lambda_3   =&\frac{r-3\lambda_2+2\lambda_3}{n-2}=:A \\
\lambda_2-u_2 = &  u_3-u_2= u_1-u_3 = u_1-\lambda_2 =  r+n\lambda_3-u_1-nu_2 \\
= & \frac{-r+(n+1)\lambda_2-n\lambda_3}{n-2}=:B
\end{align*}
Then $A  > B$, since
\[ A-B> \frac{1}{n-2}\lrp{(n+4-n-4)\lambda_2+(-n-2+n+2)\lambda_3} = 0.  \]
\autoref{valuations/Lagrangian 1bii}  lies outside of the Gelfand--Cetlin polytope, since
\begin{equation} 
\frac{\lambda_2-\lambda_3}{r-\lambda_2} < \frac{1}{n} 
\quad  \Longleftrightarrow \quad 
 r > (n+1)\lambda_2-n\lambda_3 
\quad  \Longleftrightarrow \quad 
B<0.
\label{in polytope}
\end{equation}  
\end{enumerate}
\end{enumerate}
Therefore,  we can always obtain 6 non-degenerate critical points of the leading order potential with valuation \autoref{relations between lambda and r} from case \ref{valid critical points}. (We remark that the rank of the cohomology of $M$ is also $6$.)
By the following lemma, they correspond to the critical points of the full potential. 
\begin{lem}[Hensel's Lemma, \cite{CharestWoodward} Theorem 4.36]
Suppose that $W: (\Lambda^{\times})^n \to \Lambda_{0}$ is a function of the form $W=W_0+W_1$, where $W_0$ consists of the terms of lowest $T$-valuation. 
Suppose that $y_0\in \operatorname{Crit}(W_0)$ and $W_0$ has non-degenerate Hessian at $y_0$. 
Then there exists $y_1\in (\Lambda_+)^n$, such that 
$y=y_0\exp (y_1)\in \operatorname{Crit}(W)$ with the terms in $y_1$ having higher $T$-valuation than those in $y_0$. 
\end{lem}
\end{proof}

Although we require $n>2$, we may compare Proposition \ref{Proposition critical points of po} with the case of the monotone $U(3)$-coadjoint orbit $\orbit_{\lam}$ with $\lambda_1>\lambda_2>\lambda_3$ by letting $n=0$ and $r= \lambda_1$.  
Then \autoref{valuations/Lagrangian 1a} becomes
\[ (u_1,u_2,u_3)= \lrp{ 
\frac{ 4\lam_1 +\lam_2+\lam_3}{6}, \frac{\lam_2+\lam_3}{2} , \frac{\lambda_1+\lambda_2+\lambda_3}{3}},    \]
which coincides with the computation in Nishinou--Nohara--Ueda \cite{NNU}. 
In particular, for a monotone $U(3)$-coadjoint orbit $\orbit_{\lam}$, namely $\lam_1 -\lambda_2=\lambda_2-\lambda_3=\delta$, \autoref{valuations/Lagrangian 1a} becomes
\[ (u_1,u_2,u_3)= \lrp{ 
\frac{2\lam_2+\delta}{2}, \frac{2\lam_2-\delta}{2} , \lambda_2}.  \] 

\section{Displaceability of Gelfand--Cetlin fibers}
\label{section nondisplaceability}
\label{Non-displaceability of the Lagrangian torus}
In this section, we prove that the Lagrangian torus fiber over $u_{0}$ is non-displaceable, and that all other interior Gelfand--Cetlin fibers of $M$ are displaceable, except for a one-parameter family of fibers.

Previously, the works of Pabiniak \cite{PabiniakDisplacing} and \cite{ChoKimOhCrit} have studied the displaceability of the interior Gelfand--Cetlin fibers of generic $U(3)$-coadjoint orbits of the form $\orbit_{\lam}$. 
They derived that, when $\orbit_{\lam}$ is non-monotone, only the central fiber determined by the critical points of the potential function is non-displaceable; and when $\orbit_{\lam}$ is monotone, a line segment of fibers connecting the central fiber and the vertex Lagrangian $S^3$ are non-displaceable. 

We remark that, in our case (see Proposition \ref{Proposition Probe-nondisplaceable set}), $u_0$ is disjoint from the one-parameter set. 
\subsection{Non-displaceability of the central Lagrangian torus} 
\begin{thm}[Existence of non-displaceable Lagrangians]
\label{thm Existence of a non-displaceable Lagrangian}
Let $M$ be a manifold of the form \autoref{defn of M} with parameters $r,\lambda_1,\lambda_2,\lambda_3, n$ satisfying 
\autoref{n>2},
\autoref{relations between lambda and r}, and $\dfrac{\lam_2-\lam_3}{r-\lam_2}<\dfrac{1}{n+1}$. 

Let  $u_0$ be as in \autoref{valuation of crit points},    and let $\rho\in \Hom(\pi_1(\Phi^{-1}(u_0)),U_{\Lambda})$ be such that $(u_0,\rho)$ determines a critical point, found in Section \ref{section Critical points of the potential}, of $W$ via \autoref{coordinates on union of local systems}. 
Let $\widehat{L}$ be the brane structure on $\Phi^{-1}(u_0)$ determined by the holonomy $\rho$.  
Then 
\[ 
HF\lrp{\widehat{L}, \widehat{L}; \Lambda_0} \cong H^*(\widehat{L};\Lambda_0)\ne 0.  
\]
As a consequence, $\Phi^{-1}(u_0)$ is a non-displaceable Lagrangian torus in $M$. 
\end{thm}
The proof of the Theorem follows closely that in \cite{CharestWoodward} Section 4.7. 
\begin{proof}
Let $L =  \Phi^{-1}(u_0)$ and 
$\rho:\pi_1(L)\to U_{\Lambda}$ be a local system corresponding to a critical point $y=\left(\rho(e_1^*)T^{u_1},\ldots, \rho(e_n^*)T^{u_n}\right)$ of $W$ in Section \ref{section Critical points of the potential}.  
We identify 
\[ T_{y} \Hom(\pi_1(L),U_{\Lambda}) \cong H^1(L,\Lambda_0). 
\]
Let $\left(CF(\widehat{L}, \widehat{L}), \{\m_k\}\right)$ be the split $A_{\infty}$ algebra induced by $\rho$. 
Then $\left(CF(\widehat{L}, \widehat{L}),\m_1 \right)$ is a weakly finite d.g.c.f.z. in the sense of \cite{fooobook1} Definition 6.3.27. 
Therefore, by Lemma 6.3.24, the $E_2$-page of the spectral sequence  constructed in \cite{fooobook1} Section 6.3.3 computes the Morse cohomology $H^*(L;\Lambda_0)$ of $L$. 
Moreover, by  Theorem 6.3.28, the spectral sequence converges to $HF(\widehat{L},\widehat{L}; \Lambda_0)$.

Let $c\in H^1(L;\Lambda_0)$.  
Since $y$ is a critical point, $\partial_{c}W(y)=0$. 
Moreover, for any $\beta\in H_2(X,L)$, we have $\partial_{c} y(\partial \beta)=\pair{\partial \beta, c} y(\partial \beta)$. 

Let  $\widetilde{\Gamma}$ be a split tropical graph with one incoming boundary leaf and one output root. 
Let $f(\widetilde{\Gamma})$ be the split tropical graph obtained from forgetting the incoming leaf. 
Let $x_1 \in CF^1(\widehat{L}, \widehat{L})$ be a Morse cycle.  
There exists a forgetful map 
\[ f:  \Mq^{\mathrm{split}}_{\widetilde{\Gamma},\text{red}}\left(L,\mathfrak{p}_{\widetilde{\Gamma}}, \eta_{gen};x_0, x_1\right)_0 \to \Mq^{\mathrm{split}}_{f(\widetilde{\Gamma}),\text{red}}( L, \mathfrak{p}_{f\left(\widetilde{\Gamma}\right)},\eta_{gen};x_0)_0.  \] 
Let $[u:C\to M]\in \Mq^{\mathrm{split}}_{\widetilde{\Gamma},\text{red}}\left(L,\mathfrak{p}_{\widetilde{\Gamma}}, \eta_{gen};x_0, x_1\right)_0 $. 
Let $W^u(x_1) $ denote the unstable manifold of $x_1$. 
Then  
\[ f^{-1}([u]) \cong \left(u|_{\partial C}\right)^{-1} W^u(x_1)  .\]  
Thus, the cardinality of $f^{-1}([u])$ is $\pair{[\partial u], [x_1]}$. 
Since the other factors (the symplectic area, the number of interior leaves, the holonomy, the orientation sign, the number of split edges, and the multiplicity of $\widetilde{\Gamma}$) contributed by $[u]$ to the split composition maps $\m_0$ and $\m_1$ stay the same, we obtain a weak divisor equation for split $A_{\infty}$ algebras similar to that for broken $A_{\infty}$ algebras in \cite{CharestWoodward} Proposition 4.33: 
\[ 
\pair{\partial \beta, [x_1]} \m_{0,\beta}(x_1) = \m_{1,\beta}(x_1) \qquad \forall \beta\in H_2(M,L
). 
\]
This implies that 
\begin{align*}
    0=\partial_{c}W(y)x^{\whitet} 
    & = \partial_{c}\m_0(1) \\
    & = \sum_{\beta} \pair{\partial \beta, c}\m_{0,\beta}(1)T^{\omega(\beta)}\\ 
    & = \sum_{\beta} \m_{1,\beta}(c)T^{\omega(\beta)}\\
    & = \m_1(c). 
\end{align*}

We claim that $\m_{1,\beta}(c)=0$ for any $\beta\in H_2(M,L)$ and $c$ of any classical degree $k$. 

We prove it by induction on the pair $(E, k)$, where $E=\omega(\beta), k = \deg c$. 
Since $\m_{1,0}(c)=0$, the base case holds. 
Suppose, for some $(E,k)>(0,0)$, we have  $\m_{1,\beta}(c)=0$ whenever  $(\omega(\beta), \deg c) \leq (E, k) $. 
Let $c_1,c_2\in H^*(L;\Lambda_0)$ such that $1\leq \deg c_1, \deg (c_2)\leq k$. 
By $A_{\infty}$ relations and the induction hypothesis, 
\begin{align*}
& \m_{1,\beta}(\m_{2,0}(c_1, c_2)) \\
 = &\sum_{\beta_1+\beta_2 = \beta} \pm \m_{2,\beta_1}\left(\m_{1,\beta_2}(c_1), c_2\right) + \sum_{\beta_1+\beta_2 = \beta} \pm \m_{2,\beta_1}\left(c_1, \m_{1,\beta_2}(c_2)\right) \\
 & + \sum_{\beta_1+\beta_2 = \beta, \beta_2\ne 0} \pm \m_{1,\beta_1}\left(\m_{2,\beta_2}(c_1,c_2)\right) =  0.
\end{align*} 
 This proves the claim by induction. 
 Since any $c$ is a linear combination of cup products, we conclude that $\m_1(c)=0$ for all $c\in H^*(L;\Lambda_0)$. 
 Therefore, the spectral sequence collapses at $E_2$ page and  $HF(\widehat{L},\widehat{L};\Lambda_0)\cong H^*(L;\Lambda_0)$. 
 
\end{proof}

\subsection{Displaceability of some other fibers}
Except at the vertex $(\lambda_2,\lambda_2,\lambda_2)$, 
the Gelfand--Cetlin polytope $\tilde{\Delta}$ of the coadjoint orbit $\orbit_{\lam}$ is a smooth moment polytope for the Gelfand--Cetlin $T^3$-action. 
Since the $S^1$-action we used for the symplectic cut is a subaction of the Gelfand--Cetlin $T^3$-action, $\Delta\setminus \{(\lambda_2,\lambda_2,\lambda_2)\} $ is also a toric moment polytope. 
Therefore, we can apply the method of probes developed in \cite{Probes} to detect displaceable Gelfand--Cetlin fibers. 

The main result of this subsection is the following proposition.  
\begin{prop}
\label{Proposition Probe-nondisplaceable set}
For $u=(u_1,u_2,u_3)\in \interior \Delta$, the Gelfand--Cetlin fiber $\Phi^{-1}(u)$ is displaceable if  
\[ u \not \in \set{\left(\frac{3\lambda_2-\lambda_3}{2}-t,\frac{\lambda_2+\lambda_3}{2}+t,\lambda_2 \right) }{0\leq t\leq \frac{\lambda_2-\lambda_3}{2}} \cup\{ u_0\}, \]
where $u_0$ is as in \eqref{valuation of crit points}. 
\end{prop}

\begin{figure}[!h]
\centering
 \tikzset{every picture/.style={line width=0.75pt}} %set default line width to 0.75pt        
\resizebox{0.15\textwidth}{!}{     
\begin{tikzpicture}[x=0.75pt,y=0.75pt,yscale=-1,xscale=1]
%uncomment if require: \path (0,210); %set diagram left start at 0, and has height of 210

%Straight Lines [id:da09409162136470928] 
\draw    (286.7,92) -- (287.32,182) ;
%Straight Lines [id:da3584520501758405] 
\draw    (287.32,182) -- (376.53,92) ;
%Straight Lines [id:da7119344424365017] 
\draw    (331.09,19.2) -- (286.7,92) ;
%Straight Lines [id:da5508993121422686] 
\draw    (376.42,92) -- (398,83) ;
%Straight Lines [id:da7027760188499721] 
\draw    (286.81,92) -- (376.53,92) ;
%Straight Lines [id:da5833404863866453] 
\draw    (331.09,19.2) -- (397.83,58.27) ;
%Straight Lines [id:da9955875689086934] 
\draw    (398,58.27) -- (398,83) ;
%Straight Lines [id:da08076956055591866] 
\draw  [dash pattern={on 4.5pt off 4.5pt}]  (331.09,19.2) -- (332,160) ;
%Straight Lines [id:da3408084546341863] 
\draw [color={rgb, 255:red, 38; green, 7; blue, 241 }  ,draw opacity=1 ][line width=1.5]    (350,81.75) -- (376.53,92) ;
%Shape: Free Drawing [id:dp656338920626889] 
\draw  [color={rgb, 255:red, 208; green, 2; blue, 27 }  ,draw opacity=1 ][line width=3] [line join = round][line cap = round] (357.04,74.36) .. controls (357.37,74.36) and (357.7,74.36) .. (358.04,74.36) ;
%Straight Lines [id:da1182286597298704] 
\draw    (397.85,61) -- (376.42,92) ;
%Straight Lines [id:da8235281210897879] 
\draw    (287.32,182) -- (332,160) ;
%Straight Lines [id:da40183348366229354] 
\draw    (332,160) -- (398,83) ;

\end{tikzpicture}

}  

\caption{Possibly non-displaceable interior GC fibers}
\label{fig: possibly non-displaceable set}
\end{figure}
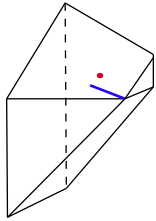

\begin{proof}
    This follows from Lemma \ref{Displaceable by top bottom probes}, Lemma \ref{Displaceable by left right probes}, Lemma \ref{Displaceable by top back probes}, and Lemma \ref{Displaceable by front back probes}, which we prove in the rest of the section. 
\end{proof}
Its worth noting that in the case when $r=\lambda_1$, $n=0$, and $\lambda_1=2\lambda_2-\lambda_3$,  we have $u_0 =\left(\frac{3\lambda_2-\lambda_3}{2} ,\frac{\lambda_2+\lambda_3}{2} ,\lambda_2 \right)$ 
and Proposition \ref{Proposition Probe-nondisplaceable set} coincide with the monotone case in Pabiniak \cite{Pabiniak} and Cho-Kim-Oh \cite{ChoKimOh}. 

To prove the lemmas, we first recall the definition of probes.  
\begin{defn}[Probes]
Let $w$ be a point of some facet $F$ of a rational polytope $\Delta$ and $\alpha\in \Z^n$ be such that $|\pair{\alpha, n_F}|=1$, where $n_F$ is an integral inner normal vector of $F$. The \textit{probe} $p_{F,\alpha}(w)$ with direction $\alpha\in \Z^n$ and initial point $w \in F$ is the half open line segment consisting of $w $ together with the points in $\interior \Delta$ that lie on the ray from $w$ in direction $\alpha$.    
\end{defn}

Using the method of probes, McDuff proved that most of the Lagrangian torus fibers in symplectic toric manifolds are displaceable by Hamiltonian diffeomorphisms.

\begin{lem}[Displacing toric fibers by Probes \cite{Probes} Lemma 2.4]
\label{probes lemma}
 Let $\Delta$ be a smooth moment polytope of a toric moment map $\Phi$. Suppose that a point $u\in \interior \Delta$ lies on the probe $p_{F,\alpha}(w)$. Then, if $w$ lies in the interior of $F$ and $u$ is less than halfway along $p_{F,\alpha}(w)$, the fiber $\Phi^{-1}(u)$ is displaceable.
\end{lem}

In our case, the facets of $\Delta$ and their inner normal vectors are given by the following. 
\begin{align*}
   F_1 \text{ (Left):} 
   &  \quad 
  l_1(u) = u_2-\lam_3\geq 0  
   &   
   (0,1,0)
    \\
  F_2  \text{ (Right):} 
   &  \quad  
   l_2(u) = -u_2+\lam_2\geq 0  
   &   
    (0,-1,0)
    \\
   F_3  \text{ (Top):} 
   &  \quad  
    l_3(u) = u_1-u_3\geq 0  
   &    
   (1,0,-1)
    \\
    F_4 \text{ (Bottom):} 
   &  \quad  
   l_4(u) = u_3-u_2\geq 0  
   &    
    (0,-1,1)
    \\
  F_5  \text{ (Front):} 
   &  \quad  
   l_5(u) = u_1-\lam_{2}\geq 0  
   &    
    (1,0,0)
    \\
   F_6 \text{ (Back):} 
    & \quad
    l_6(u) = -u_1-nu_2+r+n\lam_3\geq 0
    & (-1,-n,0)
\end{align*}

The rest of the section is dedicated to the proof of Lemmas \ref{Displaceable by top bottom probes}, \ref{Displaceable by left right probes}, \ref{Displaceable by top back probes}, and \ref{Displaceable by front back probes}. 
\begin{lem} 
\label{Displaceable by top bottom probes}
For $u=(u_1,u_2,u_3)\in \interior \Delta$, the Gelfand--Cetlin fiber $\Phi^{-1}(u)$ is displaceable if $u_1+u_2-2u_3 \ne 0$. 
\end{lem}
\begin{proof}
We will use the probes in the directions $\pm (0,0,1)$ from the top and bottom facets to displace the fibers. 

The question meant to say that the probability that any player answers any question correctly is p.  

 First, consider the probes of the form $p_{F_3,\alpha_3}(w)$ with initial points $w$ from $F_3$ with direction $\alpha_3 = (0,0,-1)$. The corresponding rays all intersect facet $F_4$. 
A point $u\in \interior \Delta$ lies halfway along some $ p_{F_3,\alpha_3}(w)$ if and only if 
\[ l_3(u)=l_4(u) \quad \Longleftrightarrow \quad u_1-u_3 = u_3-u_2  \quad \Longleftrightarrow \quad  u_1+u_2-2u_3 =0.  \]
By Lemma \ref{probes lemma}, $p_{F_3,\alpha_3}$ can displace the interior Gelfand fiber $\Phi^{-1}(u)$ if $u_1+u_2-2u_3 < 0$. 

Similarly, the probes $p_{F_4,\alpha_4}$ with initial points from $F_4$ with direction $\alpha_4 = (0,0,1)$ can displace the interior Gelfand fiber $\Phi^{-1}(u)$ if $u_1+u_2-2u_3 >0$. 
\end{proof}
Note that $C = \{ u_1+u_2-2u_3=0 \} \cap \interior \Delta $ is the interior of the convex hull of 
\[ \left\{ \left(\lambda_2, \lambda_2, \lambda_2\right), \left(a,\lambda_2, \frac{a+\lambda_2}{2}\right), \left(r,\lambda_3,\frac{r+\lambda_3}{2}\right), \left(\lambda_2,\lambda_3,\frac{\lambda_2+\lambda_3}{2}\right) \right\},  \]
where $a = r-n(\lambda_2-\lambda_3)$. 
\begin{lem}
\label{Displaceable by left right probes}
For $u=(u_1,u_2,u_3)\in \interior \Delta$ with $\lambda_2<u_1<a$, the Gelfand--Cetlin fiber $\Phi^{-1}(u)$ is displaceable if 
either 
\[ \lambda_2<  u_3< a,\quad u_2 \ne \frac{\lambda_2+\lambda_3}{2} \]
or 
\[
u_3 = \lambda_2,\quad  u_2<\frac{\lambda_2+\lambda_3}{2}.  
\]
\end{lem}
\begin{proof}
We will use some probes in the directions $\pm (0,1,0)$ from the left and right facets to displace the fibers.

Consider the probes $p_{F_1, \alpha_1}$ with $\alpha_1=(0,1,0)$. If $\lambda_2\leq  w_3< a$ and $\lambda_2< w_1< a$, then the ray corresponding to $p_{F_1, \alpha_1}(w)$ intersects $F_2$. 
$u$ lies halfway along such a probe if and only if $u_2-\lambda_3=\lambda_2-u_2$.  
Thus, such probes $p_{F_1, \alpha_1}$ can displace the fibers $\Phi^{-1}(u)$ if
$u_2< \frac{\lambda_2+\lambda_3}{2}$, $\lambda_2\leq  u_3< a, \lambda_2<u_1<a$.  

Similarly, the probes $p_{F_2, \alpha_2}$ with $\alpha_2=(0,-1,0)$ can displace the fibers $\Phi^{-1}(u)$ if
$u_2> \frac{\lambda_2+\lambda_3}{2}$, $\lambda_2<  u_3< a, \lambda_2<u_1<a$.    
\end{proof}
 
\begin{lem}
\label{Displaceable by top back probes}
For $u=(u_1,u_2,u_3)\in \interior \Delta$ with $\lambda_2<u_1<a$, the Gelfand--Cetlin fiber $\Phi^{-1}(u)$ is displaceable if 
\[ u_3>\lambda_2, \quad 2u_1+nu_2-u_3\ne r+n\lambda_3. \]
\end{lem} 
\begin{proof}
We will use the probes in the directions $\pm (1,0,0)$ from the top and back facets to displace the fibers.
As before, the probes $p_{F_6,\alpha_6}(w)$ with $\alpha_6 = (-1,0,0)$ and $w_3>\lambda_2$ and the probes $p_{F_3,\alpha_3'}$ with $\alpha_3' = (1,0,0)$ 
displace the fibers $\Phi^{-1}(u)$ with 
\[ u_3>\lambda_2, \quad -u_1-nu_2+r+n\lam_3\ne u_1-u_3. \]
\end{proof}

\begin{lem}
\label{Displaceable by front back probes}
For $u=(u_1,u_2,u_3)\in \interior \Delta$ with $\lambda_2<u_1<a$, the Gelfand--Cetlin fiber $\Phi^{-1}(u)$ is displaceable if 
\[ u_3<\lambda_2, \quad 2u_1+nu_2\ne r+\lambda_2+n\lambda_3. \]
\end{lem} 
\begin{proof}
We will use the probes in the directions $\pm (1,0,0)$ from the front and back facets to displace the fibers.   

If $u_3<\lambda_2$, the ray corresponding to $p_{F_6, \alpha_6}$ intersects facet $F_5$. 
A point $u\in \interior \Delta$, $u_3<\lambda_2$ lies halfway along some $ p_{F_6,\alpha_6}(w)$ if and only if 
\[ -u_1-nu_2+r+n\lambda_3 = u_1-\lambda_2 \quad \Longleftrightarrow \quad 2u_1+nu_2 =r+\lambda_2+n\lambda_3.  \]
By Lemma \ref{probes lemma}, $p_{F_6,\alpha_6}$ can displace the interior Gelfand--Cetlin fiber $\Phi^{-1}(u)$ if $u_3< \lambda_2$ and $2u_1+nu_2 > r+\lambda_2+n\lambda_3$. 
Similarly, we can displace the fibers  $\Phi^{-1}(u)$ with $u_3< \lambda_2$ and $2u_1+nu_2 < r+\lambda_2+n\lambda_3$ using $p_{F_5,(1,0,0)}$. 
\end{proof}
In summary, we managed to displace the fibers over the interior points complementary to the colored points in Figure \ref{fig: possibly non-displaceable set}. 

\section{Generation of the Fukaya category}
\label{section Generation of the Fukaya category}

In light of Abouzaid's and Abouzaid-Fukaya-Oh-Ohta-Ono's split-generation criteria (\cite{AbouzaidGeneration},  \cite{AFOOO}), we study the properties of the open-closed map defined on the Hochschild homology of the subcategory generated by the distinguished critical points found in Section \ref{section Critical points of the potential}. 
The criterion \cite{AFOOO}, in its general form, asserts that, for a compact symplectic manifold $(M,\omega)$, if the open-closed map on the Hochschild homology of a full subcategory $\mathcal{B}$ of the Fukaya category $\fuk(M)$ of $M$ hits the unit of the quantum cohomology of $M$, then the objects of $\mathcal{B}$ split-generate $\fuk(M)$. 

We remark that, in this section, we will consider the Fukaya category defined using the unbroken (i.e. no matching conditions for broken manifolds) holomorphic treed disks as in \cite{VWXBlowups}. 
In particular, we consider the full-subcategory $\mathcal{B}$  on the collection $\mathscr{L}=\{L_i \mid 1\leq i \leq 6\}$ of Lagrangians, where $L_i$ is the object associated with the local system  $\rho_i\in \Hom(\pi_1(L), U_{\Lambda})$  on $L=\Phi^{-1}(u_0)$ corresponding to the critical points found in Section \ref{section Critical points of the potential}. Then 
the surjectivity of the open-closed map  
\begin{equation}
 OC:   HH_*\left(\mathcal{B}\right)\to QH^*(M,  \Lambda),   
\end{equation}
implies that the collection split-generates the Fukaya category of $M$. 

The inclusion of $CF(L_i,L_i)$ in $CC_*(\mathcal{B})$ induces a map 
\[ HH_*(CF(L_i,L_i))\to HH_*(\mathcal{B}) \qquad \forall 1\leq i \leq 6. 
\] 
By composing these maps with the open-closed map, we obtain 
\begin{equation}
\label{length 0 OC}  \bigoplus_{i=1}^6 HH_*(CF(L_i,L_i))\to HH_*(\mathcal{B})\xrightarrow{OC} QH^*(M, \Lambda),   
\end{equation} 
which we again denote by $OC$.   
We will prove the following. 
\begin{thm}
\label{OC is an isomorphism}
    The map $OC: \bigoplus_{i=1}^6 HH_*(CF(L_i,L_i))\to  QH^*(M, \Lambda) $ in \autoref{length 0 OC} is an isomorphism.  
\end{thm}
By \cite{GoertschesGKM} Remark 4.7,
\[ H^*(M;\Z)\cong 
\Z[X_1, X_2]/(X_1^2 + 3X_1X_2 + X_2^2 , X_1^2X_2 + X_1X_2^2 ), \]
where $X_1,X_2$ are of degree $2$.  
We consider 
\[ H^*(M;\Lambda) \cong H^*(M;\Z)\otimes_{\Z}\Lambda. \]
Thus, $b_0=1, b_2 = 2, b_4=2, b_6=1$, 
% \[ b_0=\dim_{\Lambda} H^0(M;\Lambda) =\dim \Span_{\Lambda}\{1\} = 1,  \]
% \[ b_2=\dim_{\Lambda} H^2(M;\Lambda) =\dim \Span_{\Lambda}\{X_1,X_2\} = 2, 
% \]
% \[ b_4=\dim_{\Lambda} H^4(M;\Lambda) =\dim \Span_{\Lambda}\{X_1^2, X_1X_2\} = 2,  \] 
% \[
% b_6=\dim_{\Lambda} H^6(M;\Lambda) =\dim \Span_{\Lambda}\{X_1^2X_2\} =1, 
% \]
and $H^k(M;\Lambda)=0$ for all other $k$. 
Therefore,  $\dim QH^*(M,\Lambda) = 6$. 

\begin{proof}
Thus, it suffices to prove that the domain is also of dimension $6$ and that \autoref{length 0 OC}
is a surjective homomorphism. The statements are shown in Proposition \ref{Hochschild of Clifford is rank 1} and Corollary \ref{Surjectivity of OC}. 
\end{proof}

\subsection{\texorpdfstring{Dimension of $HH_*(CF(L_i,L_i))$}{Dimension of HH_*(CF(L_i,L_i))}}

To show that the domain $\bigoplus\limits_{i=1}^{6}HH_*(CF(L_i,L_i))$ of \autoref{length 0 OC} is $6$-dimensional, 
we appeal to the property that the Floer cohomology $HF(L_i,L_i)$ are Clifford algebras associated to non-degenerate quadratic forms, which are intrinsically formal with one-dimensional Hochschild homology.

The proof applies a similar argument as that in \cite{VWXBlowups} Corollary 6.9.  
\begin{prop}
\label{Hochschild of Clifford is rank 1}
$ HH_*(CF(L_i,L_i)) $ is one-dimensional for $1\leq i \leq 6$. 
\end{prop}
\begin{proof}   
Let $\left(CF^{\mathrm{split}}(L_i, L_i), \{\m_k\}\right)$ be the split $A_{\infty}$ algebra induced by $\rho_i$.
    Let $x_1,x_2,x_3\in CF_{\mathrm{split}}^1(L)$ be Morse cocycles such that their cohomology classes $e_1,e_2,e_3$  generate $H^1(L;\Lambda_0)$.
    
Let  $\widetilde{\Gamma}$ be a split tropical graph with two incoming boundary leaves and one output root. 
Let $f(\widetilde{\Gamma})$ be the split tropical graph obtained from forgetting the incoming leaves.  
There is a map forgetting the incoming boundary leaves: 
\begin{align*}
 f:  \Mq^{\mathrm{split}}_{\widetilde{\Gamma},\text{red}}\left(L,\mathfrak{p}_{\widetilde{\Gamma}}, \eta_{gen};x_0, x_r,x_s\right)_0 
 \cup 
 \Mq^{\mathrm{split}}_{ \widetilde{\Gamma},\text{red}}\left(L,\mathfrak{p}_{ \widetilde{\Gamma}}, \eta_{gen};x_0, x_s,x_r\right)_0  \\ 
 \longrightarrow \Mq^{\mathrm{split}}_{f(\widetilde{\Gamma}),\text{red}}( L, \mathfrak{p}_{f\left(\widetilde{\Gamma}\right)},\eta_{gen};x_0)_0.    
\end{align*}  
Suppose that $[u:C\to M]\in \Mq^{\mathrm{split}}_{\widetilde{\Gamma},\text{red}}\left(L,\mathfrak{p}_{\widetilde{\Gamma}}, \eta_{gen};x_0\right)_0$ represents a basic disk class $\beta\in H_2(X,L)$; i.e. $\beta=PD[\Phi^{-1}(F)]$ for some facet $F$ of the Gelfand--Cetlin polytope of $M$. 

Let $W^u(x_r), W^u(x_s) $ denote the unstable manifolds of $x_r$ and $x_s$. 
Then  
\[ f^{-1}([u])  
 = \left(\left(u|_{\partial C}\right)^{-1} W^u(x_r) \right)\times  \left(\left(u|_{\partial C}\right)^{-1}  W^u(x_s)\right) .\]  
Thus, the cardinality of $f^{-1}([u])$ is $\pair{[\partial u], e_r}\cdot \pair{[\partial u], e_s}$. 

Since the other factors (the symplectic area, the number of interior leaves, the holonomy, the orientation sign, the number of split edges, and the multiplicity of $\widetilde{\Gamma}$) contributed by $[u]$ to the split composition maps $\m_0$ and $\m_2$ stay the same, we obtain a weak divisor equation for split $A_{\infty}$ algebras similar to that for unbroken $A_{\infty}$ algebras in \cite{VWXBlowups} Lemma 5.12: For a basic disk class $\beta\in H_2(X,L)$ and any $1\leq r,s \leq 3$, 
\[ 
 \m_{2,\beta}(e_r,e_s) + \m_{2,\beta} (e_s,e_r) = \sum_{\beta} 
 \pair{\partial \beta, e_r}\cdot \pair{\partial \beta, e_s} \m_{0,\beta}(1). 
\]
Let $P$ be the leading order potential of $M$. 
For any $1\leq r,s \leq 3$, we have
    \begin{align*}
        &  \m_{2}(e_r,e_s) + \m_{2} (e_s,e_r) \\
      = & \sum_{\beta\in H_2(X,L)}\left( \m_{2,\beta}(e_r,e_s) + \m_{2,\beta} (e_s,e_r)\right)T^{\omega(\beta)}  \\ 
     =& \sum_{\beta \text{ basic }} \pair{\partial \beta, e_r}\cdot \pair{\partial \beta, e_s} \m_{0,\beta}(1)  + \sum_{\beta\in H_2(X,L)}\left( \m_{2,\beta}(e_r,e_s) + \m_{2,\beta} (e_s,e_r)\right)T^{\omega(\beta)} \\ 
        =& \left( \frac{\partial^2 P}{\partial x_r x_s} \eval_{y_i}  + \text{ higher order terms}\right)\cdot e_{L_i}\\
        = & Q(e_r,e_s)\cdot e_{L_i}
    \end{align*}
    where $e_{L_i}$ is the  strict unit of $CF^{\mathrm{split}}(L_i,L_i)$, and $B$ is a non-degenerate symmetric bilinear form that induces a non-degenerate quadratic form $Q$ on $H^1(L,\Lambda_0)$.
    
      Thus, for any $v\in H^1(L;\Lambda_0)$, we have \[ \m_2 \left(v,v\right) =  B(v,v)\cdot e_L. \]
Therefore, the split Floer cohomology $HF^{\mathrm{split}}(L_i,L_i)$ is isomorphic to the $\mathbb{Z}/2\mathbb{Z}$-graded Clifford algebra $A = \operatorname{Cl}\left(HF(L_i,L_i;\Lambda), Q \right)$ of a non-degenerate quadratic form $Q$. 
By \cite{VWSplit} Theorem 1.1, the split Floer cohomology 
$HF^{\mathrm{split}}(L_i,L_i)$ 
is identified with the Floer cohomology of the associated broken 
$A_{\infty}$-algebra. By \cite{CharestWoodward}, Theorem~8.13, 
the broken and unbroken $A_{\infty}$-algebras are homotopy equivalent, 
and hence their Floer cohomologies agree. It follows that the unbroken Floer cohomology $HF(L_i,L_i)$ is isomorphic to $A$. 

By \cite{SheridanFanoHypersurface} Corollary 6.4,  the Clifford algebra $A$ is intrinsically formal. Thus, there is an $A_{\infty}$ quasi-isomorphism from
$CF (L_i,L_i)$  to $HF(L_i,L_i)$. 
Moreover, there is an $A_{\infty}$ homotopy equivalence $F: CF(L_i,L_i) \to HF(L_i,L_i)$. By a homotopy perturbation argument (\cite{Seidelbook} Corollary 1.14), $F$ has a homotopy inverse. Therefore, they have the same Hochschild homology, 
\[ HH_*(CF(L_i,L_i)) \cong HH_*(HF(L_i,L_i)) \cong HH_*(A). 
\]  
By Lemma \ref{lem Hochschild homology of Clifford algebras} (see Proposition 1 in \cite{Kassel1986} Section 6), 
\[ HH_0(A) \cong \Lambda, \qquad HH_k(A)\cong 0 \quad \text{ for }k>0. \]  
\end{proof}

This implies that $\dim\bigoplus
\limits_{i=1}^6 HH_*(CF(L_i,L_i))=6$.

\subsection{\texorpdfstring{Surjectivity of $OC$}{Surjectivity of OC}}

In this section, we consider the unbroken $A_{\infty}$ category and unbroken $A_{\infty}$ algebras as in \cite{VWXBlowups}.

For simplicity, for $1\leq i \leq 6$, we denote $HH_*(CF(L_i,L_i))$ by $V_i$. Moreover, we denote  the maximum point of the chosen Morse function on $L_i$ by $v_i$ and its Hochschild homology class by $[v_i]\in HH_*(CF(L_i,L_i))$.    
It suffices to show that the set $\left\{ OC[v_1],\ldots, OC[v_6]\right\}$ is linearly independent.

\begin{lem} 
If $i,j\in \{1,\ldots, 6\}$ and $i\ne j$, then $HF(L_i,L_j)=0$. 
\end{lem}

\begin{proof}   
First, assume the values of the potential function $W(L_i)= w_i,  W(L_j)=w_j$ are distinct. 
Let $\alpha \in HF(L_i,L_i), \beta \in HF(L_j,L_j)$. By \cite{VWXBlowups} Theorem 3.23, $OC( \alpha )$, $OC( \beta )$ lie in the generalized eigenspaces of the quantum multiplication  $c_1\star: QH^*(X)\to QH^*(X)$ with eigenvalues $w_i, w_j$, respectively. 
Since $c_1\star$ is symmetric with respect to $\pair{-,-}_{PD_M}$, we have
\[ \pair{OC( \alpha ) , OC( \beta )}_{PD_M} = 0. \]

Now assume $W(L_i) = W(L_j)$. Then $\m_1^{\rho_i,\rho_j}\circ \m_1^{\rho_i,\rho_j}=0$,
and $\left(C=CF(L_i,L_j;\Lambda_{0,nov}),\m_1^{\rho_i,\rho_j}\right)$ is a graded finitely generated free module over $\Lambda_{0,nov}$ such that $C$ and each $C^p=$submodule generated by index $k$  is complete with respected to the energy filtration\footnote{This is called a differential graded completed free filtered $\Lambda_{0,nov}$-module generated by energy zero elements, abbreviated d.g.c.f.z., in \cite{fooobook1} Definition 6.3.8.}. 
Let $\rho_{i,0}, \rho_{j,0}$ be the leading term of $\rho_i,\rho_j$. 
By  \cite{fooobook1} Lemma 6.3.24 and \cite{fooobook1} Theorem 6.3.28, there is a spectral sequence with $E_2 \cong H^*(L; \Theta )$ converging to $E_{\infty}\cong HF(L_i,L_j;\Lambda_{0,nov})$, where $\Theta$ is the line bundle corresponding to monodromy representation 
\[ \tau=\rho_{i,0}^{-1}\rho_{j,0}:\Z^3\cong \pi_1(L)\to U_{\Lambda}.
\]  
Let $e_1,e_2,e_3$ be the generators of $\Z^3$. 
Let $\tau_k=\tau(e_k)$ for all $k$. 
Let $\Theta_k$ be the line bundle over $S^1$ with monodromy $\tau_k$.

Since $L\cong T^3$ is a Lagrangian $3$-torus, 
by the K\"{u}nneth formula, 
\[ 
H^*(L;\Theta) \cong \bigotimes_{k=1}^3H^*\left(S^1;   \Theta_k \right).  
\]
By the computation in Section \ref{section Critical points of the potential}, $\rho_{i,0}\ne \rho_{j,0}$ for distinct $i,j$. 
Thus, at least one $\Theta_k$ is nontrivial. 

Consider 
the Morse cohomology $H^*\left(S^1; \Theta_k \right) $ with coefficients in a nontrivial local system $\Theta_k$. Since the cohomology is independent of the morse function, we may use the standard height Morse function on $S^1$ with maximum $p_1$ and minimum $p_0$ and two negative gradient flowlines $\gamma, \gamma'$ from $p_1$ to $p_0$. 
Then \[ 
CM^i\left(S^1; \Theta_k \right) =  \Theta_k\eval_{p_i} \quad \forall i\in \{0,1\}. 
\] 
Let $P_{\gamma}, P_{\gamma'}: \Theta_k\eval_{p_1}\to \Theta_k\eval_{p_0}$ be the parallel transport maps along $\gamma, \gamma'$, respectively. 
Since $\Theta_k$ is flat, $P_{\gamma'} = P_{\gamma}\circ \tau_k(1)$. 
Thus, the differential $d$ is given by  \[ d = \pm (P_{\gamma}-P_{\gamma'}) = \pm P_{\gamma}\circ (\id - \tau_k(1) ) .\] 

Since $P_{\gamma}$ is an isomorphism and $\id - \tau_k(1) \ne 0$, we have \[ H^0(S^1;\Theta_k) \cong  \Lambda/(\id - \tau_k(1))\Lambda \cong 0, \qquad H^1(S^1;\Theta_k) \cong \ker (\id - \tau_k(1)) \cong 0,\]  
which implies $H^*(S^1;\Theta_k) \cong 0$. 

It follows that  $E_2 \cong 0$, and therefore $HF(L_i,L_j)=0$.
\end{proof}

\begin{prop}[Coincidence of pairings]
Let $i,j\in \{1,\ldots, 6\}$ and $\mathcal{J}$ be a basis of $HF(L_i,L_j)$. 
Then 
\[ \pair{OC([v_i]), OC([v_j])}_{PD_M} = \sum\limits_{e_k\in \mathcal{J} } \pair{ \m_2\left([v_j],\m_2\left([v_i],[e_k]\right)\right), [e_k]}. \]  
\end{prop}
\begin{proof}
      Consider the compactification of the moduli space $\overline{\mathcal{M}}^{\text{ann}}_{1,1,0}$ of treed annuli, with 
one inner boundary leaf and one outer boundary leaf with an angle offset of $\pi$. 
There is a homeomorphism
\[ \rho: \overline{\mathcal{M}}^{\text{ann}}_{1,1,0}\to [-\infty,+\infty] \]
defined by the following. For a treed annulus $[C]\in \overline{\mathcal{M}}^{\text{ann}}_{1,1,0}$ with inner radius $r_1$ and outer radius $r_2$, 
\[ \rho([C]) = \frac{ r_1r_2^{-1}}{1+ r_1r_2^{-1}}. \]

Let $\overline{\mathcal{M}}_{1,1}(L_i,L_j;v_i,v_j)$ be the compactification of the moduli space of stable maps with domain in $\overline{\mathcal{M}}^{\text{ann}}_{1,1,0}$ such that the following holds. 
\begin{enumerate}[label=(\roman*)]
    \item The surface maps satisfy the pseudoholomorphic equation with respect to the almost complex structure in the perturbation data. 
    \item The treed maps satisfy the negative gradient flow equation with respect to the Morse-Smale pair in the perturbation data. 
    \item The outer and inner boundary circles are mapped to $L_i, L_j$, respectively. 
    \item The inner and outer boundary leaves are constrained to $v_i, v_j$, respectively. 
\end{enumerate}

The boundary of 
$\overline{\mathcal{M}}_{1,1}(L_i,L_j;v_i,v_j)_1$---the subset of $\overline{\mathcal{M}}_{1,1}(L_i,L_j;v_i,v_j)$ 
consisting of elements of expected dimension 
1---decomposes into the following types of strata. (Also see Figure \ref{fig:moduli of annuli}.)
\begin{enumerate}[label=(\alph*)]
    \item $\rho =-\infty$ corresponds to rigid treed pseudoholomorphic maps given by two disks connected by two broken boundary edges, which separate each disk boundary into two components. The count of such strata corresponds to $\sum \limits_{x\in \mathcal{I}(L_i,L_j)} (-1)^{|x|}\pair{\m_2(v_j,\m_2(v_i,x)),x} $, where 
    \begin{align*}
     & \mathcal{I}(L_i,L_j)\\
    = 
    & \begin{cases}
      \{\text{critical points of the Morse function } F_{L_i,L_j}\}
   \quad  &  \text{ if } i \ne j\\
     \{\text{critical points of the Morse function } F_{L_i,L_i}\}\cup \{1^{\greyt}_{L_i,c}, 1^{\whitet}_{L_i,c}\} 
   \quad  & \text{ if }i = j
    \end{cases}   
    \end{align*}
 
    \item $\rho \in (-\infty,\infty)$  corresponds to rigid treed pseudoholomorphic annuli with one boundary breaking. The contributions of these configurations are 
    \[ 
    \mathcal{T}(\delta_{CC} (v_i), v_j) \pm \mathcal{T}( (v_i), \delta_{CC}(v_j) ), 
    \]
    where $ \mathcal{T}(\underline{a}_-, \underline{a}_+)$ counts the number of rigid treed annuli of essential type with inner and outer boundary mappings constrained to $\underline{a}_-$ and $\underline{a}_+$, respectively, and $\delta_{CC}$ is the Hochschild differential.  
    Since $v_i, v_j$ are Hochschild cycles, the contribution of such configurations is zero.   
    \item $\rho = \infty$ corresponds to rigid treed pseudoholomorphic maps given by two disks with boundaries on $L_i$ and $L_j$ and boundary markings constrained to  $v_i$ and $v_j$, respectively, and which are joined by an interior edge broken at a critical point of the Morse function on 
$M$. The count of such strata corresponds to the intersection pairing $\pair{OC_-(v_i),OC_+(v_j)}_{CM^{\bullet}}$. 
\end{enumerate}

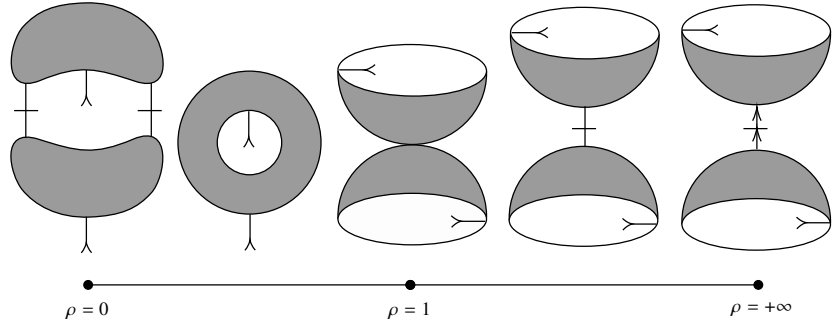
\begin{figure}
    \centering

\tikzset{every picture/.style={line width=0.75pt}} %set default line width to 0.75pt        
\scalebox{0.65}{ 

\begin{tikzpicture}[x=0.75pt,y=0.75pt,yscale=-1,xscale=1]

%Shape: Polygon Curved [id:ds6872311197481195] 
\draw  [fill={rgb, 255:red, 155; green, 155; blue, 155 }  ,fill opacity=1 ] (29.59,29.44) .. controls (40.69,17.41) and (53.12,11.74) .. (77.98,11.03) .. controls (102.84,10.33) and (120.6,20.95) .. (129.93,32.98) .. controls (139.25,45.02) and (146.8,81.83) .. (116.61,71.92) .. controls (86.42,62.01) and (73.98,58.47) .. (42.91,71.21) .. controls (11.83,83.96) and (18.49,41.48) .. (29.59,29.44) -- cycle ;
%Shape: Regular Polygon [id:dp17439618086842923] 
\draw  [fill={rgb, 255:red, 155; green, 155; blue, 155 }  ,fill opacity=1 ] (129.12,158.41) .. controls (118.22,170.56) and (105.88,176.36) .. (81.03,177.33) .. controls (56.18,178.31) and (38.25,167.88) .. (28.73,155.94) .. controls (19.2,144) and (11.05,107.26) .. (41.4,116.85) .. controls (71.75,126.45) and (84.24,129.85) .. (115.11,116.78) .. controls (145.97,103.7) and (140.02,146.26) .. (129.12,158.41) -- cycle ;
%Straight Lines [id:da9493089831300486] 
\draw    (31.62,73.3) -- (31.62,116.02) ;
%Straight Lines [id:da46442595651510643] 
\draw    (129,73.3) -- (129,115) ;
%Straight Lines [id:da462766839226589] 
\draw    (22.55,94.66) -- (41.2,94.66) ;
%Straight Lines [id:da7764633186535034] 
\draw    (118.45,94.66) -- (137.1,94.66) ;
%Shape: Circle [id:dp7662698477040144] 
\draw  [fill={rgb, 255:red, 155; green, 155; blue, 155 }  ,fill opacity=1 ] (149.68,119.5) .. controls (149.68,88.85) and (174.53,64) .. (205.18,64) .. controls (235.83,64) and (260.68,88.85) .. (260.68,119.5) .. controls (260.68,150.15) and (235.83,175) .. (205.18,175) .. controls (174.53,175) and (149.68,150.15) .. (149.68,119.5) -- cycle ;
%Shape: Circle [id:dp6375820710291517] 
\draw  [fill={rgb, 255:red, 255; green, 255; blue, 255 }  ,fill opacity=1 ] (180.18,119.5) .. controls (180.18,105.69) and (191.37,94.5) .. (205.18,94.5) .. controls (218.99,94.5) and (230.18,105.69) .. (230.18,119.5) .. controls (230.18,133.31) and (218.99,144.5) .. (205.18,144.5) .. controls (191.37,144.5) and (180.18,133.31) .. (180.18,119.5) -- cycle ;
%Shape: Arc [id:dp24111385938451668] 
\draw  [draw opacity=0][fill={rgb, 255:red, 155; green, 155; blue, 155 }  ,fill opacity=1 ] (273.86,178.9) .. controls (273.84,178.46) and (273.83,178.01) .. (273.82,177.56) .. controls (273.33,147.01) and (298.74,121.83) .. (330.58,121.32) .. controls (362.41,120.81) and (388.62,145.16) .. (389.11,175.71) .. controls (389.12,176.41) and (389.12,177.1) .. (389.1,177.8) -- (331.47,176.63) -- cycle ; \draw   (273.86,178.9) .. controls (273.84,178.46) and (273.83,178.01) .. (273.82,177.56) .. controls (273.33,147.01) and (298.74,121.83) .. (330.58,121.32) .. controls (362.41,120.81) and (388.62,145.16) .. (389.11,175.71) .. controls (389.12,176.41) and (389.12,177.1) .. (389.1,177.8) ;  
%Shape: Ellipse [id:dp5425166586002697] 
\draw  [fill={rgb, 255:red, 253; green, 253; blue, 253 }  ,fill opacity=1 ] (273.84,178.9) .. controls (273.84,167.86) and (299.64,158.9) .. (331.47,158.9) .. controls (363.29,158.9) and (389.09,167.86) .. (389.09,178.9) .. controls (389.09,189.95) and (363.29,198.9) .. (331.47,198.9) .. controls (299.64,198.9) and (273.84,189.95) .. (273.84,178.9) -- cycle ;
%Shape: Arc [id:dp9384287406295507] 
\draw  [draw opacity=0][fill={rgb, 255:red, 155; green, 155; blue, 155 }  ,fill opacity=1 ] (389.09,63.97) .. controls (389.11,64.42) and (389.12,64.87) .. (389.12,65.32) .. controls (389.28,95.87) and (363.61,120.78) .. (331.77,120.95) .. controls (299.93,121.12) and (273.98,96.5) .. (273.81,65.94) .. controls (273.81,65.24) and (273.82,64.55) .. (273.84,63.85) -- (331.47,65.63) -- cycle ; \draw   (389.09,63.97) .. controls (389.11,64.42) and (389.12,64.87) .. (389.12,65.32) .. controls (389.28,95.87) and (363.61,120.78) .. (331.77,120.95) .. controls (299.93,121.12) and (273.98,96.5) .. (273.81,65.94) .. controls (273.81,65.24) and (273.82,64.55) .. (273.84,63.85) ;  
%Shape: Ellipse [id:dp679556279489937] 
\draw  [fill={rgb, 255:red, 255; green, 255; blue, 255 }  ,fill opacity=1 ] (389.09,63.97) .. controls (388.97,75.02) and (363.08,83.7) .. (331.25,83.36) .. controls (299.43,83.02) and (273.73,73.79) .. (273.85,62.74) .. controls (273.96,51.7) and (299.86,43.02) .. (331.68,43.36) .. controls (363.5,43.7) and (389.2,52.93) .. (389.09,63.97) -- cycle ;
%Shape: Arc [id:dp3914947500628225] 
\draw  [draw opacity=0][fill={rgb, 255:red, 155; green, 155; blue, 155 }  ,fill opacity=1 ] (406.75,180.03) .. controls (406.73,179.59) and (406.72,179.14) .. (406.71,178.69) .. controls (406.22,148.14) and (431.63,122.96) .. (463.47,122.45) .. controls (495.3,121.94) and (521.51,146.29) .. (522,176.84) .. controls (522.01,177.54) and (522.01,178.23) .. (521.99,178.93) -- (464.35,177.76) -- cycle ; \draw   (406.75,180.03) .. controls (406.73,179.59) and (406.72,179.14) .. (406.71,178.69) .. controls (406.22,148.14) and (431.63,122.96) .. (463.47,122.45) .. controls (495.3,121.94) and (521.51,146.29) .. (522,176.84) .. controls (522.01,177.54) and (522.01,178.23) .. (521.99,178.93) ;  
%Shape: Ellipse [id:dp2404861806568036] 
\draw  [fill={rgb, 255:red, 255; green, 255; blue, 255 }  ,fill opacity=1 ] (406.73,180.03) .. controls (406.73,168.99) and (432.53,160.03) .. (464.35,160.03) .. controls (496.18,160.03) and (521.98,168.99) .. (521.98,180.03) .. controls (521.98,191.08) and (496.18,200.03) .. (464.35,200.03) .. controls (432.53,200.03) and (406.73,191.08) .. (406.73,180.03) -- cycle ;
%Shape: Arc [id:dp872765082687372] 
\draw  [draw opacity=0][fill={rgb, 255:red, 155; green, 155; blue, 155 }  ,fill opacity=1 ] (522.98,35.1) .. controls (523,35.55) and (523,36) .. (523.01,36.45) .. controls (523.17,67) and (497.49,91.91) .. (465.65,92.08) .. controls (433.81,92.25) and (407.87,67.63) .. (407.7,37.07) .. controls (407.7,36.37) and (407.71,35.67) .. (407.73,34.98) -- (465.35,36.76) -- cycle ; \draw   (522.98,35.1) .. controls (523,35.55) and (523,36) .. (523.01,36.45) .. controls (523.17,67) and (497.49,91.91) .. (465.65,92.08) .. controls (433.81,92.25) and (407.87,67.63) .. (407.7,37.07) .. controls (407.7,36.37) and (407.71,35.67) .. (407.73,34.98) ;  
%Shape: Ellipse [id:dp5811490797599709] 
\draw  [fill={rgb, 255:red, 255; green, 255; blue, 255 }  ,fill opacity=1 ] (522.97,35.1) .. controls (522.86,46.15) and (496.96,54.83) .. (465.14,54.49) .. controls (433.32,54.15) and (407.62,44.92) .. (407.74,33.87) .. controls (407.85,22.83) and (433.75,14.15) .. (465.57,14.49) .. controls (497.39,14.83) and (523.09,24.06) .. (522.97,35.1) -- cycle ;
%Shape: Arc [id:dp02810337071790414] 
\draw  [draw opacity=0][fill={rgb, 255:red, 155; green, 155; blue, 155 }  ,fill opacity=1 ] (540.75,183.03) .. controls (540.73,182.59) and (540.72,182.14) .. (540.71,181.69) .. controls (540.22,151.14) and (565.63,125.96) .. (597.47,125.45) .. controls (629.3,124.94) and (655.51,149.29) .. (656,179.84) .. controls (656.01,180.54) and (656.01,181.23) .. (655.99,181.93) -- (598.35,180.76) -- cycle ; \draw   (540.75,183.03) .. controls (540.73,182.59) and (540.72,182.14) .. (540.71,181.69) .. controls (540.22,151.14) and (565.63,125.96) .. (597.47,125.45) .. controls (629.3,124.94) and (655.51,149.29) .. (656,179.84) .. controls (656.01,180.54) and (656.01,181.23) .. (655.99,181.93) ;  
%Shape: Ellipse [id:dp39273602304257593] 
\draw  [fill={rgb, 255:red, 255; green, 255; blue, 255 }  ,fill opacity=1 ] (540.73,183.03) .. controls (540.73,171.99) and (566.53,163.03) .. (598.35,163.03) .. controls (630.18,163.03) and (655.98,171.99) .. (655.98,183.03) .. controls (655.98,194.08) and (630.18,203.03) .. (598.35,203.03) .. controls (566.53,203.03) and (540.73,194.08) .. (540.73,183.03) -- cycle ;
%Shape: Arc [id:dp5236077196434306] 
\draw  [draw opacity=0][fill={rgb, 255:red, 155; green, 155; blue, 155 }  ,fill opacity=1 ] (655.98,33.1) .. controls (656,33.55) and (656,34) .. (656.01,34.45) .. controls (656.17,65) and (630.49,89.91) .. (598.65,90.08) .. controls (566.81,90.25) and (540.87,65.63) .. (540.7,35.07) .. controls (540.7,34.37) and (540.71,33.67) .. (540.73,32.98) -- (598.35,34.76) -- cycle ; \draw   (655.98,33.1) .. controls (656,33.55) and (656,34) .. (656.01,34.45) .. controls (656.17,65) and (630.49,89.91) .. (598.65,90.08) .. controls (566.81,90.25) and (540.87,65.63) .. (540.7,35.07) .. controls (540.7,34.37) and (540.71,33.67) .. (540.73,32.98) ;  
%Shape: Ellipse [id:dp6811516994333258] 
\draw  [fill={rgb, 255:red, 255; green, 255; blue, 255 }  ,fill opacity=1 ] (655.97,33.1) .. controls (655.86,44.15) and (629.96,52.83) .. (598.14,52.49) .. controls (566.32,52.15) and (540.62,42.92) .. (540.74,31.87) .. controls (540.85,20.83) and (566.75,12.15) .. (598.57,12.49) .. controls (630.39,12.83) and (656.09,22.06) .. (655.97,33.1) -- cycle ;
%Straight Lines [id:da8308809857281325] 
\draw    (465.52,91.3) -- (465.52,123) ;
%Straight Lines [id:da028334239728510346] 
\draw    (455.45,108.66) -- (474.1,108.66) ;
%Straight Lines [id:da3149762387170597] 
\draw    (598.88,93.61) -- (598.88,107.61) ;
\draw [shift={(598.88,91.61)}, rotate = 90] [color={rgb, 255:red, 0; green, 0; blue, 0 }  ][line width=0.75]    (10.93,-3.29) .. controls (6.95,-1.4) and (3.31,-0.3) .. (0,0) .. controls (3.31,0.3) and (6.95,1.4) .. (10.93,3.29)   ;
%Straight Lines [id:da12364135230985929] 
\draw    (588.45,108.66) -- (607.1,108.66) ;
%Straight Lines [id:da8843858473681214] 
\draw    (78.63,83.16) -- (78.63,63.16) ;
\draw   (74.67,91.14) .. controls (76.85,88.42) and (78.16,85.71) .. (78.57,83) .. controls (79.04,85.69) and (80.41,88.36) .. (82.67,91.01) ;

%Straight Lines [id:da880259475810161] 
\draw    (78.63,197.16) -- (78.63,177.16) ;
\draw   (74.67,205.14) .. controls (76.85,202.42) and (78.16,199.71) .. (78.57,197) .. controls (79.04,199.69) and (80.41,202.36) .. (82.67,205.01) ;

%Straight Lines [id:da517965908935423] 
\draw    (204.63,115.16) -- (204.63,95.16) ;
\draw   (200.67,123.14) .. controls (202.85,120.42) and (204.16,117.71) .. (204.57,115) .. controls (205.04,117.69) and (206.41,120.36) .. (208.67,123.01) ;

%Straight Lines [id:da8518809490497552] 
\draw    (205.63,195.16) -- (205.63,175.16) ;
\draw   (201.67,203.14) .. controls (203.85,200.42) and (205.16,197.71) .. (205.57,195) .. controls (206.04,197.69) and (207.41,200.36) .. (209.67,203.01) ;

%Straight Lines [id:da4097451367266711] 
\draw    (369.28,180.63) -- (388.03,180.63) ;
\draw   (361.72,176.27) .. controls (364.32,178.7) and (366.88,180.14) .. (369.42,180.56) .. controls (366.91,181.14) and (364.44,182.73) .. (362,185.33) ;

%Straight Lines [id:da8342874005633825] 
\draw    (294.79,63.02) -- (274.79,63.02) ;
\draw   (302.79,66.95) .. controls (300.06,64.78) and (297.33,63.49) .. (294.63,63.09) .. controls (297.31,62.6) and (299.98,61.22) .. (302.63,58.96) ;

%Straight Lines [id:da30822519779385127] 
\draw    (428.68,34.15) -- (408.68,34.15) ;
\draw   (436.68,38.08) .. controls (433.94,35.91) and (431.23,34.62) .. (428.52,34.22) .. controls (431.21,33.73) and (433.87,32.35) .. (436.52,30.09) ;

%Straight Lines [id:da8953571472842385] 
\draw    (560.68,32.15) -- (540.68,32.15) ;
\draw   (568.68,36.08) .. controls (565.94,33.91) and (563.23,32.62) .. (560.52,32.22) .. controls (563.21,31.73) and (565.87,30.35) .. (568.52,28.09) ;

%Straight Lines [id:da5827578679501694] 
\draw    (502.17,182.76) -- (520.92,182.76) ;
\draw   (494.61,178.4) .. controls (497.21,180.83) and (499.77,182.26) .. (502.31,182.69) .. controls (499.8,183.27) and (497.33,184.86) .. (494.89,187.46) ;

%Straight Lines [id:da5040330987776671] 
\draw    (637.17,181.76) -- (655.92,181.76) ;
\draw   (629.61,177.4) .. controls (632.21,179.83) and (634.77,181.26) .. (637.31,181.69) .. controls (634.8,182.27) and (632.33,183.86) .. (629.89,186.46) ;

%Straight Lines [id:da38328710402150135] 
\draw    (80,230) -- (330,230) ;
\draw [shift={(330,230)}, rotate = 0] [color={rgb, 255:red, 0; green, 0; blue, 0 }  ][fill={rgb, 255:red, 0; green, 0; blue, 0 }  ][line width=0.75]      (0, 0) circle [x radius= 3.35, y radius= 3.35]   ;
\draw [shift={(80,230)}, rotate = 0] [color={rgb, 255:red, 0; green, 0; blue, 0 }  ][fill={rgb, 255:red, 0; green, 0; blue, 0 }  ][line width=0.75]      (0, 0) circle [x radius= 3.35, y radius= 3.35]   ;
%Straight Lines [id:da8739776041957618] 
\draw    (330,230) -- (600,230) ;
\draw [shift={(600,230)}, rotate = 0] [color={rgb, 255:red, 0; green, 0; blue, 0 }  ][fill={rgb, 255:red, 0; green, 0; blue, 0 }  ][line width=0.75]      (0, 0) circle [x radius= 3.35, y radius= 3.35]   ;
\draw [shift={(330,230)}, rotate = 0] [color={rgb, 255:red, 0; green, 0; blue, 0 }  ][fill={rgb, 255:red, 0; green, 0; blue, 0 }  ][line width=0.75]      (0, 0) circle [x radius= 3.35, y radius= 3.35]   ;
%Straight Lines [id:da5740651494382661] 
\draw    (598.88,109.61) -- (598.88,124.61) ;
\draw [shift={(598.88,107.61)}, rotate = 90] [color={rgb, 255:red, 0; green, 0; blue, 0 }  ][line width=0.75]    (10.93,-3.29) .. controls (6.95,-1.4) and (3.31,-0.3) .. (0,0) .. controls (3.31,0.3) and (6.95,1.4) .. (10.93,3.29)   ;

% Text Node
\draw (61,242.4) node [anchor=north west][inner sep=0.75pt]    {$\rho =0$};
% Text Node
\draw (311,242.4) node [anchor=north west][inner sep=0.75pt]    {$\rho =1$};
% Text Node
\draw (577,242.4) node [anchor=north west][inner sep=0.75pt]    {$\rho =+\infty $};

\end{tikzpicture}
}
\caption{Moduli space of treed annuli with two inputs}
    \label{fig:moduli of annuli}
\end{figure}
Therefore, we have 
\[ \pair{OC_-(v_i),OC_+(v_j)}_{CM^{\bullet}} = \sum\limits_{x\in \mathcal{I}(L_i,L_j) } \pm \pair{\m_2(v_j,\m_2(v_i,x)),x}. \]
The chain map \[ \Phi: CF(L_i,L_j) \to CF(L_i,L_j), \quad \Phi(x) := \m_2(v_j,\m_2(v_i,x))
\]  
induces a map 
\[ [\Phi]: HF(L_i,L_j) \to HF(L_i,L_j), \quad [\Phi]([x]) :=  \m_2\left([v_j],\m_2\left([v_i],[x] \right)\right), 
\]  
where $[v_i]\in HF(L_i,L_i), [v_j] \in HF(L_j,L_j)$. 

Since $HF(L_i,L_j)=0$, we have $[\Phi]=0$. 
 
Note that on cohomology level $OC_-$ and $OC_+$ are identical, both of which we denote by $OC$. 
Therefore, 
\begin{align*}
\pair{OC([v_i]),OC([v_j])}_{PD_M} 
& = \sum\limits_{e_k\in \mathcal{J} } \pm \pair{ \m_2\left([v_j],\m_2\left([v_i],[e_k]\right)\right), [e_k]}  \\
&=  \sum\limits_{e_k\in \mathcal{J} } \pm   \pair{[\Phi]([e_k])  , [e_k]} = 0.    
\end{align*} 
   
\end{proof} 
\begin{lem} \label{point class as an intersection} Let $1\leq i\leq 6$. There is a basis consisting of $\beta_1,\beta_2,\beta_3$ of $H^1(L_i;\Lambda)$ such that 
\[[v_i] = \beta_1\cup \beta_2\cup \beta_3. \]
\begin{enumerate} [label=(\arabic*)]
\item $[v_i] = \beta_1\cup \beta_2\cup \beta_3$.
\item $\beta_k\ast \beta_k = \lambda_k\cdot \bold{1}$ for some $\lambda_k\ne 0$ for all $k\in \{1,2,3\}$; \item $\beta_k\ast \beta_l + \beta_l\ast \beta_k = 0$ for all $k,l \in \{1,2,3\}$ with $k \ne l$.  
\end{enumerate} 
\end{lem} 
\begin{proof} Let $\psi: L_i\to T^3$ be a diffeomorphism. Choose an orthogonal basis of $H^1(L_i;\Z)$ consisting of primitive classes $\alpha_1,\alpha_2, \alpha_3\in H^1(L_i;\Z)\cong \Hom(\pi_1(L_i), \Z))\cong \Hom(\Z^3, \Z) $. Fix $k\in \{1,2,3\}$. Then $\alpha_k$ is given by a map of the form \[ \Z^3 \to \Z, \qquad (x_1,x_2,x_3)\mapsto c_{k,1}x_1+c_{k,2}x_2+c_{k,3}x_3, \] where $(c_{k,1},c_{k,2},c_{k,3})\in \Z^3\setminus \{(0,0,0)\}$ with $\gcd(c_{k,1},c_{k,2},c_{k,3})=1$ and $\det (c_{k,j}) = 1$.

Define $\rho_k: \R^n/\Z^n\to \R/\Z$ by \[ \rho_k[(x_1,x_2,x_3)] := [c_{k,1}x_1+c_{k,2}x_2+c_{k,3}x_3].\] Then $d\rho_k = c_{k,1}dx_1+c_{k,2}dx_2+c_{k,3}dx_3$ is nowhere zero. Thus, $\tau_k:=\rho_k\circ \psi: L\to S^1$ is a submersion, and $Z_k=\tau_k^{-1}(\tau_k(v_i))$ is a smooth codimension-1 submanifold of $L$. (Since $\tau_k$ is a proper submersion, $Z_k\hookrightarrow L_i\to S^1$ is a fibration. By the long exact sequence of the homotopy groups for the fibration, $\pi_1(Z_k)\cong \Z^{2}$, and $\pi_m(Z_k)=0$ for $m\geq 2$. This shows that $Z_k$ is a closed connected manifold which is $K(\Z^2,1)$. Hence, $Z_k$ is diffeomorphic to $T^2$. ) Thus, $\alpha_k=PD[Z_k]$. Moreover, since the $\alpha_k$'s form a $\Z$ basis, $(\tau_1,\tau_2,\tau_3)$ is a diffeomorphism, and $Z_1 \cap Z_2\cap Z_3= \{ v_i\}$. Therefore, $[v_i] = \alpha_1\cup \alpha_2\cup \alpha_3$.  

Choose an orthogonal basis $\beta_1',\beta_2',\beta_3'$ for $H^1(L;\Lambda)$ such that 
\[ \beta_{k}'\ast \beta_l' + \beta_{l}'\ast \beta_k'=0 \qquad \forall k\ne l,  \]
\[ \beta_k'\ast \beta_k' = Q(\beta_k')\cdot \bold{1}\ne 0 \qquad \forall k. \]

Then $\beta_l = \sum\limits_{j=1}^3 a_{lj}\alpha_j$ for each $l$. Let $A= (a_{lj})_{l,j}$. 
Then 
\[ 
\beta_1' \cup \beta_2' \cup \beta_3' = \det A \cdot \alpha_1\cup \alpha_2\cup \alpha_3.
\] 
Then  
\[ \beta_1 = (\det A)^{-1} \beta_1', \qquad \beta_2 = \beta_2', \qquad \beta_3 = \beta_3'
\] 
is a desired basis. 
\end{proof}
 
\begin{prop}     $\pair{OC([v_i]), OC([v_j])}_{PD_M}\begin{cases}
     = 0 & \text{ if }  i\ne j\\
     \ne 0 & \text{ if }  i= j
    \end{cases}$. 
\end{prop}
\begin{proof}
    If $i\ne j$, then $HF(L_i,L_j)=0$, which implies that 
\[ \pair{OC_-([v_i]),OC_+([v_j])}_{PD_M} =0 \quad \text{ if } i \ne j. \] 
If $i = j$, then 
\[ \pair{OC([v_i]),OC([v_i])}_{PD_M} = \sum\limits_{e_k\in \mathcal{J} } \pair{ \m_2\left([v_i],\m_2\left([v_i],[e_k]\right)\right), [e_k]^{\vee}}.\] 
By Lemma \ref{point class as an intersection},  
\[ [v_i] = \beta_1\cup \beta_2\cup \beta_3 = \beta_1\ast \beta_2 \ast \beta_3 + \eta ,\]
where $\val(\eta)>0$, and  $x\ast y:=\m_2(x,y)$.  

Fix $1\leq i \leq 6$.  We have
\begin{align*}
 \m_2([v_i],[v_i]) 
 & = \left(  \beta_1\ast \beta_2 \ast \beta_3 + \eta\right) \ast \left( \beta_1\ast \beta_2 \ast \beta_3 + \eta\right) = (\beta_1\ast \beta_2 \ast \beta_3)\ast (\beta_1\ast \beta_2 \ast \beta_3) + \eta' \\
 & = (-1)^{\frac{3\cdot (3-1)}{2}} (\beta_1\ast \beta_1)\ast(\beta_2\ast \beta_2)\ast(\beta_3\ast \beta_3)+ \eta'  \\
 & =  \pm  \gamma \cdot \bold{1}  + \eta', 
\end{align*}
where $ \val(\gamma)\geq 0, \val(\eta')>0$,
since $\beta_1,\beta_2,\beta_3$ form an orthogonal basis of $H^1(L_i;\Z)$. In particular, $\beta_r\ast \beta_s = -\beta_s\ast \beta_r$ for $r\ne s$ and, for all $r$, $\beta_r\ast \beta_r = \alpha_r \cdot \bold{1}\ne 0$ for some $\alpha_r$ with positive valuation. 
Therefore, for all $i$, we have 
\[ \pair{OC([v_i]),OC([v_i])}_{PD_M} \ne 0.  \]
\end{proof}
 \begin{cor}[Surjectivity of $OC$]
 \label{Surjectivity of OC}
The map $OC: \bigoplus\limits_{i=1}^6 HH_*(CF(L_i,L_i))\to  QH^*(M, \Lambda) $ in \autoref{length 0 OC} is surjective.  
 \end{cor}
 \begin{proof}
     We want to show that the $OC([v_1]), \ldots, OC([v_6])$ form a basis of $QH^*(M,\Lambda)$. 
Suppose  
\[ \sum_{i=1}^6 a_i OC([v_i])=0, \]
where $a_k\ne 0$ for some $k$. 
Then 
\begin{align*}
0 
=  \pair{0,0 }_{PD_M} & =\pair{\sum_{i=1}^6 a_i OC([v_i]),\sum_{i=1}^6 a_i OC([v_i]) }_{PD_M}\\
& = \sum_{i=1}^6  a_i^2 \pair{ OC([v_i]),OC([v_i]) }_{PD_M} \ne 0.     
\end{align*} 
This is a contradiction. 
Moreover, since $\dim QH^*(M,\Lambda)=6$, $OC_-([v_1]), \ldots, OC_-([v_6])$ form a basis of $QH^*(M,\Lambda)$. 
Therefore, map \autoref{length 0 OC} is surjective. 
 \end{proof}

\appendix

\newpage
\section{Multiplicity-free manifolds}
\label{section Multiplicity-free manifolds}

Multiplicity-free Hamiltonian $G$-manifolds is a natural generalization of toric manifolds to non-abelian group actions. Toric manifolds themselves are multiplicity-free manifolds for Hamiltonian torus actions.  
Moreover, there is an analog of Delzant's classification theorem due to Knop \cite{KnopClassification}, Theorem 11.2, which implies that compact multiplicity-free manifolds are classified by their Kirwan polytope and the weight lattice.

Here, we review some equivalent definitions of multiplicity-free manifolds. 

\begin{defprop}[\cite{WoodwardClassification} Proposition A.1, \cite{lanethesis} Section 4.3]
 Let $G$ be a compact connected Lie group, which acts on a compact symplectic manifold $(M,\omega)$ in a Hamiltonian fashion with moment map $\Phi$. 
Then $(M,\omega)$  is a \textit{multiplicity-free $G$-manifold} if any of the following holds. 
\begin{enumerate}[label=(\arabic*)]
    \item The Poisson subalgebra $C^{\infty}(M)^G$ of invariant smooth functions is abelian. 
    \item $G$ acts transitively on $\Phi^{-1}(\mathcal{O}_{\xi})$ for every $\xi \in \Phi(M)$.
    \item For any $\xi \in \Phi(M)$, the dimension of the reduced space $\Phi^{-1}(\xi)/G_{\xi}$ is zero. 
\end{enumerate}
Morever, the following holds. 
\begin{enumerate}[label=(\arabic*)]
\setcounter{enumi}{3}
    \item If $G$ acts locally freely on a dense subset of $M$, then $M$ is multiplicity-free if and only if 
\[ \dim M = \dim G + \rank G. \]
\item If $M$ admits a collective completely integrable system, namely a completely integrable consisting of maps of the form $f\circ \Phi$, where $f\in C^{\infty} (\g^*)$, then $M$ is multiplicity-free. 
\end{enumerate}    
\end{defprop}

\section{\texorpdfstring{$\mathbb{Z}/2\mathbb{Z}$-graded Clifford algebras}{Z/2Z-graded Clifford algebras}}
In this section, we recall some facts about Clifford algebras. In particular, we recall a computation of the Hochschild homology of a non-degenerate $\mathbb{Z}/2\mathbb{Z}$-graded Clifford algebra.  

\begin{defn}[Clifford algebra]
    Let $V$ be a free module over a ring $R$. The \textit{Clifford algebra} $\Cl(q)$ associated to a quadratic form $q: V\to R$ is defined to be $T(V) /\mathcal{I}$, 
    where $T(V)$ is the tensor algebra of $V$, and  $\mathcal{I}$ is the two-sided ideal generated by 
    $\{ v\otimes v-q(v)\cdot 1\mid v\in V\}$. 
    And we endow $\Cl(q)$ with a $\mathbb{Z}/2\mathbb{Z}$-grading by specifying that the elements of $V$ have odd degree, the unit $1$ of $\Cl(q)$ has even degree, and that the degree is multiplicative. 
\end{defn} 
For the rest of the section, 
we denote by $k$ a field of characteristic $\ne 2$ and $\Cl(q)$ a $\mathbb{Z}/2\mathbb{Z}$-graded Clifford algebra induced by a \textit{non-degenerate} quadratic form $q: V\to k$. 

\begin{defn}[Separable]
Let $A$ be a $\mathbb{Z}/2\mathbb{Z}$-graded algebra over $k$ and $A^{\opposite}$ be its opposite algebra. Let $A^e = A\otimes_{k}A^{\opposite}$ be its enveloping algebra. Then $A$ is said to be $\mathbb{Z}/2\mathbb{Z}$-graded \textit{separable} if $A$ is a $\mathbb{Z}/2\mathbb{Z}$-graded projective $A^e$-module. 
\end{defn}

\begin{lem}
\label{Clifford algebras are separable superalgebras}
    $\Cl(q)$ is a separable $\mathbb{Z}/2\mathbb{Z}$-graded algebra.  
\end{lem}

\begin{proof} 
By \cite{HelmstetterMicaliClifford} Corollary (3.7.5), $\Cl(q)$ is a $\mathbb{Z}/2\mathbb{Z}$-graded Azumaya algebra; i.e., $A=\Cl(q)$ is a finitely generated and faithful projective $k$-module, and the
canonical morphism 
\[ A \hat{\otimes} A^{\opposite} \to \End(A), \quad a\otimes b^{\opposite}\mapsto \left(x\mapsto (-1)^{|b||x|}axb\right) \] is bijective. 
This implies that $\Cl(q)$ is $\mathbb{Z}/2\mathbb{Z}$-graded separable by \cite{HelmstetterMicaliClifford} Theorem (6.7.1). 
\end{proof}
\begin{lem}\label{Lemma higher hochschild homologies vanish}
    Let $A$ be a $\mathbb{Z}/2\mathbb{Z}$-graded algebra over a field $k$. If $A$ is separable, then $HH_i(A)=0$ for $i\geq 1$. 
\end{lem}
\begin{proof}
 Since $A$ is graded separable, $A$ is a graded projective $A^e$-module and $\operatorname{Tor}_i^{A^e}(A,A)=0$ for $i>0$. Therefore,   $HH_i(A) = \operatorname{Tor}_i^{A^e}(A,A) =0$ for $i>0$ by \cite{Kassel1986} (2.3). 
\end{proof}
   
The following Lemma (Lemma \ref{lem Hochschild homology of Clifford algebras}) corresponds to Proposition 1 in \cite{Kassel1986} Section 6 and computes the Hochschild homology of such Clifford algebras for certain fields. 
\begin{lem}[Hochschild homology of non-degenerate $\mathbb{Z}/2\mathbb{Z}$-graded Clifford algebras]
\label{lem Hochschild homology of Clifford algebras} 
Let $A=\Cl(q)$. Then 
\[ HH_0(A) \cong k, \qquad HH_i(A)\cong 0 \quad \text{ for }i>0. \]    
\end{lem}
We explain the proof from \cite{Kassel1986}. 
\begin{proof}
By definition,
\[ HH_0(A) = A/[A,A]_{gr}, \quad \text{ where  } [A,A]_{gr} = \Span \left\{ab-(-1)^{|a||b|}ba \mid a,b \in A\right\} \] is the graded commutator. 
Let $e_1,\ldots, e_n$ be an orthogonal basis of $V$.  
Then $e_ie_j + e_je_i =0$ for all $i\ne j$, and $e_j^2 = q(e_j)\cdot \bold{1} = a_j \cdot \bold{1}\ne 0 $ for all $1\leq j\leq n$. 

For $I \{i_1<\cdots i_k\}\subset \{1,\ldots, n\}$, let 
\[ e_I = e_{i_1}e_{i_2}\cdots  e_{i_k}, \qquad |I|=k. \]  
Let $e_{\emptyset} = 1$. 

We want to show that 
$[A,A]_{gr}$ is the vector space $A_{n-1}$ generated by monomials of the form $e_I$ with $|I|\le n-1$. 

We first show $A_{n-1}\subset [A,A]_{gr}$. 
If $|I|\le n-1$, there exists $j\not \in I$. 
Then \begin{align*}
 [e_j, e_I e_j]_{gr} 
 & = e_j(e_Ie_j) - (-1)^{|e_j||e_Ie_j|}(e_Ie_j)e_j \\
 &= (-1)^{|I|}e_Ie_j\cdot e_j-(-1)^{|I|+1} e_Ie_j\cdot e_j = (-1)^{k}\cdot  2 a_j e_I.     
\end{align*} 
Thus, $e_I\in [A,A]_{gr}$ by the invertibility of $(-1)^k\cdot 2 a_j \in \Lambda$. 

We now show $[A,A]_{gr}\subset A_{n-1}$. 
If $I\cap J=\emptyset$, then 
\[ [e_I,e_J]_{gr} = e_Ie_J-(-1)^{|I||J|+|IJ|}e_Je_I = 0. \]
If $I\cap J=K$, then we may write 
\[ e_I = (-1)^r e_K e_{\hat{I}}, \quad e_J = (-1)^s e_K e_{\hat{J}}, \quad \text{ for some } r,s\in \Z \text{ and }\hat{I}, \hat{K}\subset \{1,\ldots, n\}\setminus K. \]
Then 
\begin{align*}
 [e_I, e_J] 
 & = (-1)^{r+s}[e_K e_{\hat{I}}, e_K e_{\hat{J}}] \\
 & = (-1)^{r+s}\left( (-1)^{|\hat{I}||K|}e_K^2e_{\hat{I}}e_{\hat{J}} -(-1)^{|I||J| + |\hat{J}||\hat{K}|}e_K^2e_{\hat{J}}e_{\hat{I} }
\right)   \\
& = (-1)^{r+s}e_K^2\left( (-1)^{|\hat{I}||K|}e_{\hat{I}}e_{\hat{J}} -(-1)^{|I||J| + |\hat{J}||\hat{K}|}e_{\hat{J}}e_{\hat{I} }. 
\right) 
\end{align*}
Since $e_K^2\in k$ and $ e_{\hat{I}}e_{\hat{J}}$, $e_{\hat{J}}e_{\hat{I} }\in A_{n-1}$, 
we conclude that $[A,A]_{gr} = A_{n-1}$ and that 
\[ HH_0(A)= k\cdot e_1\cdots e_n \cong k.
\] 
Moreover, $HH_i(A)=0$ for $i>0$ by Lemma \ref{Lemma higher hochschild homologies vanish}. 
\end{proof}

We also endow an Clifford algebra with an $A_{\infty}$ algebra structure by setting $\m_2$ to be the Clifford algebra product and $\m_k=0$ for all $k\ne 2$. 
The following definition is borrowed from Sheridan \cite{SheridanFanoHypersurface} Definitions 6.1 and 6.2. 
\begin{defn}[Formal]
A $\mathbb{Z}/2\mathbb{Z}$-graded $A_{\infty}$ algebra $A$ is \textit{formal} if $A$ is $A_{\infty}$ quasi-isomorphic to $H^*(A)$, an $A_{\infty}$ algebra with only $\m_2\ne 0$. 
A $\mathbb{Z}/2\mathbb{Z}$-graded $A_{\infty}$ algebra $A$ is \textit{intrinsically formal} if every $A_{\infty}$ algebra whose $H^*(B,\m_1)$ is isomorphic to $A$ as $\mathbb{Z}/2\mathbb{Z}$-graded algebra is formal.  
\end{defn}
By \cite{SheridanFanoHypersurface} Corollary 6.4, $\Cl(q)$ is intrinsically formal.

\newpage 

\printbibliography

@misc{koblitz1984p,
  title={P-adic Numbers, p-adic Analysis, and Zeta-Functions},
  author={Koblitz, Neal},
  year={1984},
  publisher={New Yook: Springer-Verlag}
}

@article {GSGC,
    AUTHOR = {Guillemin, V. and Sternberg, S.},
     TITLE = {The Gelfand-Cetlin system and quantization of the
              complex flag manifolds},
   JOURNAL = {J. Funct. Anal.},
  FJOURNAL = {Journal of Functional Analysis},
    VOLUME = {52},
      YEAR = {1983},
    NUMBER = {1},
     PAGES = {106--128},
      ISSN = {0022-1236},
   MRCLASS = {58F07 (22E45)},
  MRNUMBER = {705993},
MRREVIEWER = {J.\ H.\ Rawnsley},
       DOI = {10.1016/0022-1236(83)90092-7},
       URL = {https://doi.org/10.1016/0022-1236(83)90092-7},
}

@article {WoodwardExample,
    AUTHOR = {Woodward, Chris},
     TITLE = {Multiplicity-free {H}amiltonian actions need not be
              {K}\"ahler},
   JOURNAL = {Invent. Math.},
  FJOURNAL = {Inventiones Mathematicae},
    VOLUME = {131},
      YEAR = {1998},
    NUMBER = {2},
     PAGES = {311--319},
      ISSN = {0020-9910,1432-1297},
   MRCLASS = {53D20},
  MRNUMBER = {1608579},
MRREVIEWER = {Stephanie\ F.\ Singer},
       DOI = {10.1007/s002220050206},
       URL = {https://doi.org/10.1007/s002220050206},
}

@misc{VW,
      title={Tropical Fukaya Algebras}, 
      author={Sushmita Venugopalan and Chris Woodward},
      year={2022},
      eprint={2004.14314v6},
      archivePrefix={arXiv},
      primaryClass={math.SG},
      url={https://arxiv.org/abs/2004.14314v6}, 
      note={arXiv:2004.14314v6},
}

@article {WoodwardClassification,
    AUTHOR = {Woodward, Chris},
     TITLE = {The classification of transversal multiplicity-free group
              actions},
   JOURNAL = {Ann. Global Anal. Geom.},
  FJOURNAL = {Annals of Global Analysis and Geometry},
    VOLUME = {14},
      YEAR = {1996},
    NUMBER = {1},
     PAGES = {3--42},
      ISSN = {0232-704X,1572-9060},
   MRCLASS = {58F05 (57S15)},
  MRNUMBER = {1375064},
MRREVIEWER = {Michel\ Brion},
       DOI = {10.1007/BF00128193},
       URL = {https://doi.org/10.1007/BF00128193},
}

@article {GoertschesGKM,
    AUTHOR = {Goertsches, Oliver and Konstantis, Panagiotis and Zoller,
              Leopold},
     TITLE = {G{KM} theory and {H}amiltonian non-{K}\"ahler actions in
              dimension 6},
   JOURNAL = {Adv. Math.},
  FJOURNAL = {Advances in Mathematics},
    VOLUME = {368},
      YEAR = {2020},
     PAGES = {107141, 17},
      ISSN = {0001-8708,1090-2082},
   MRCLASS = {57S15 (55N91 57R91)},
  MRNUMBER = {4088417},
MRREVIEWER = {Andrea\ Spiro},
       DOI = {10.1016/j.aim.2020.107141},
       URL = {https://doi.org/10.1016/j.aim.2020.107141},
}

@article {Tolman,
    AUTHOR = {Tolman, Susan},
     TITLE = {Examples of non-{K}\"ahler {H}amiltonian torus actions},
   JOURNAL = {Invent. Math.},
  FJOURNAL = {Inventiones Mathematicae},
    VOLUME = {131},
      YEAR = {1998},
    NUMBER = {2},
     PAGES = {299--310},
      ISSN = {0020-9910,1432-1297},
   MRCLASS = {53D20},
  MRNUMBER = {1608575},
MRREVIEWER = {Stephanie\ F.\ Singer},
       DOI = {10.1007/s002220050205},
       URL = {https://doi.org/10.1007/s002220050205},
}

@article{FOOOI, 
author = {Kenji Fukaya and Yong-Geun Oh and Hiroshi Ohta and Kaoru Ono},
title = {{Lagrangian Floer theory on compact toric manifolds, I}},
volume = {151},
journal = {Duke Mathematical Journal},
number = {1},
publisher = {Duke University Press},
pages = {23 -- 175},
year = {2010},
doi = {10.1215/00127094-2009-062},
URL = {https://doi.org/10.1215/00127094-2009-062}
}

@article {NNU,
    AUTHOR = {Nishinou, Takeo and Nohara, Yuichi and Ueda, Kazushi},
     TITLE = {Toric degenerations of {G}elfand-{C}etlin systems and
              potential functions},
   JOURNAL = {Adv. Math.},
  FJOURNAL = {Advances in Mathematics},
    VOLUME = {224},
      YEAR = {2010},
    NUMBER = {2},
     PAGES = {648--706},
      ISSN = {0001-8708,1090-2082},
   MRCLASS = {53D40 (14M25 53D37)},
  MRNUMBER = {2609019},
MRREVIEWER = {Hsian-Hua\ Tseng},
       DOI = {10.1016/j.aim.2009.12.012},
       URL = {https://doi.org/10.1016/j.aim.2009.12.012},
}

@article {lmtw,
    AUTHOR = {Lerman, Eugene and Meinrenken, Eckhard and Tolman, Sue and
              Woodward, Chris},
     TITLE = {Nonabelian convexity by symplectic cuts},
   JOURNAL = {Topology},
  FJOURNAL = {Topology. An International Journal of Mathematics},
    VOLUME = {37},
      YEAR = {1998},
    NUMBER = {2},
     PAGES = {245--259},
      ISSN = {0040-9383},
   MRCLASS = {58F05 (57S15)},
  MRNUMBER = {1489203},
MRREVIEWER = {Michel\ Brion},
       DOI = {10.1016/S0040-9383(97)00030-X},
       URL = {https://doi.org/10.1016/S0040-9383(97)00030-X},
}

@article {CharestWoodward,
    AUTHOR = {Charest, Fran\c cois and Woodward, Chris T.},
     TITLE = {Floer cohomology and flips},
   JOURNAL = {Mem. Amer. Math. Soc.},
  FJOURNAL = {Memoirs of the American Mathematical Society},
    VOLUME = {279},
      YEAR = {2022},
    NUMBER = {1372},
     PAGES = {v+166},
      ISSN = {0065-9266,1947-6221},
      ISBN = {978-1-4704-5310-7; 978-1-4704-7226-9},
   MRCLASS = {53D40},
  MRNUMBER = {4464438},
MRREVIEWER = {Roman\ Golovko},
       DOI = {10.1090/memo/1372},
       URL = {https://doi.org/10.1090/memo/1372},
}

@article {Cuts,
    AUTHOR = {Lerman, Eugene},
     TITLE = {Symplectic cuts},
   JOURNAL = {Math. Res. Lett.},
  FJOURNAL = {Mathematical Research Letters},
    VOLUME = {2},
      YEAR = {1995},
    NUMBER = {3},
     PAGES = {247--258},
      ISSN = {1073-2780},
   MRCLASS = {58F05 (57S25)},
  MRNUMBER = {1338784},
MRREVIEWER = {Hansj\"org\ Geiges},
       DOI = {10.4310/MRL.1995.v2.n3.a2},
       URL = {https://doi.org/10.4310/MRL.1995.v2.n3.a2},
}

@incollection {Probes,
    AUTHOR = {McDuff, Dusa},
     TITLE = {Displacing {L}agrangian toric fibers via probes},
 BOOKTITLE = {Low-dimensional and symplectic topology},
    SERIES = {Proc. Sympos. Pure Math.},
    VOLUME = {82},
     PAGES = {131--160},
 PUBLISHER = {Amer. Math. Soc., Providence, RI},
      YEAR = {2011},
      ISBN = {978-0-8218-5235-4},
   MRCLASS = {53D12 (52B12 53D20 53D35)},
  MRNUMBER = {2768658},
MRREVIEWER = {Mark\ Alan\ Branson},
       DOI = {10.1090/pspum/082/2768658},
       URL = {https://doi.org/10.1090/pspum/082/2768658},
}

@phdthesis{lanethesis,
  author  = {Lane, Jeremy M.},
  title   = {On the Topology of Collective Integrable Systems},
  school  = {University of Toronto},
  year    = {2017},
}

@article {GSThimm,
    AUTHOR = {Guillemin, Victor and Sternberg, Shlomo},
     TITLE = {On collective complete integrability according to the method
              of {T}himm},
   JOURNAL = {Ergodic Theory Dynam. Systems},
  FJOURNAL = {Ergodic Theory and Dynamical Systems},
    VOLUME = {3},
      YEAR = {1983},
    NUMBER = {2},
     PAGES = {219--230},
      ISSN = {0143-3857,1469-4417},
   MRCLASS = {58F07 (58F05 58F17)},
  MRNUMBER = {742224},
MRREVIEWER = {Manfred\ Andri\'e},
       DOI = {10.1017/S0143385700001930},
       URL = {https://doi.org/10.1017/S0143385700001930},
}

@article {Pabiniak,
    AUTHOR = {Pabiniak, Milena},
     TITLE = {Gromov width of non-regular coadjoint orbits of {$U(n)$},
              {$SO(2n)$} and {$SO(2n+1)$}},
   JOURNAL = {Math. Res. Lett.},
  FJOURNAL = {Mathematical Research Letters},
    VOLUME = {21},
      YEAR = {2014},
    NUMBER = {1},
     PAGES = {187--205},
      ISSN = {1073-2780,1945-001X},
   MRCLASS = {17B10 (14L35 17B08 17B20 53D20)},
  MRNUMBER = {3247049},
MRREVIEWER = {Andreas\ Arvanitoyeorgos},
       DOI = {10.4310/MRL.2014.v21.n1.a15},
       URL = {https://doi.org/10.4310/MRL.2014.v21.n1.a15},
}

@article {ChoKimOh,
    AUTHOR = {Cho, Yunhyung and Kim, Yoosik and Oh, Yong-Geun},
     TITLE = {Lagrangian fibers of {G}elfand-{C}etlin systems},
   JOURNAL = {Adv. Math.},
  FJOURNAL = {Advances in Mathematics},
    VOLUME = {372},
      YEAR = {2020},
     PAGES = {107304, 57},
      ISSN = {0001-8708,1090-2082},
   MRCLASS = {53D37 (14M15 37J06)},
  MRNUMBER = {4126721},
MRREVIEWER = {W.-H.\ Steeb},
       DOI = {10.1016/j.aim.2020.107304},
       URL = {https://doi.org/10.1016/j.aim.2020.107304},
}

@article {DelzantClassification,
    AUTHOR = {Delzant, Thomas},
     TITLE = {Hamiltoniens p\'eriodiques et images convexes de l'application
              moment},
   JOURNAL = {Bull. Soc. Math. France},
  FJOURNAL = {Bulletin de la Soci\'et\'e{} Math\'ematique de France},
    VOLUME = {116},
      YEAR = {1988},
    NUMBER = {3},
     PAGES = {315--339},
      ISSN = {0037-9484},
   MRCLASS = {58F05},
  MRNUMBER = {984900},
MRREVIEWER = {J.\ J.\ Duistermaat},
       URL = {http://www.numdam.org/item?id=BSMF_1988__116_3_315_0},
}

@incollection {NU,
    AUTHOR = {Nohara, Yuichi and Ueda, Kazushi},
     TITLE = {Floer homology for the {G}elfand-{C}etlin system},
 BOOKTITLE = {Real and complex submanifolds},
    SERIES = {Springer Proc. Math. Stat.},
    VOLUME = {106},
     PAGES = {427--436},
 PUBLISHER = {Springer, Tokyo},
      YEAR = {2014},
      ISBN = {978-4-431-55215-4; 978-4-431-55214-7},
   MRCLASS = {53D40},
  MRNUMBER = {3333400},
       DOI = {10.1007/978-4-431-55215-4\_38},
       URL = {https://doi.org/10.1007/978-4-431-55215-4_38},
}

@article {ChoKimOhCrit,
    AUTHOR = {Cho, Yunhyung and Kim, Yoosik and Oh, Yong-Geun},
     TITLE = {A critical point analysis of {L}andau-{G}inzburg potentials
              with bulk in {G}elfand-{C}etlin systems},
   JOURNAL = {Kyoto J. Math.},
  FJOURNAL = {Kyoto Journal of Mathematics},
    VOLUME = {61},
      YEAR = {2021},
    NUMBER = {2},
     PAGES = {259--304},
      ISSN = {2156-2261,2154-3321},
   MRCLASS = {53D40 (14M15 37J35 53D12)},
  MRNUMBER = {4342377},
MRREVIEWER = {W.-H.\ Steeb},
       DOI = {10.1215/21562261-2021-0002},
       URL = {https://doi.org/10.1215/21562261-2021-0002},
}

@article {EvansLekili,
    AUTHOR = {Evans, Jonathan David and Lekili, Yank\i },
     TITLE = {Generating the {F}ukaya categories of {H}amiltonian
              {$G$}-manifolds},
   JOURNAL = {J. Amer. Math. Soc.},
  FJOURNAL = {Journal of the American Mathematical Society},
    VOLUME = {32},
      YEAR = {2019},
    NUMBER = {1},
     PAGES = {119--162},
      ISSN = {0894-0347,1088-6834},
   MRCLASS = {53D37 (53D12 53D40)},
  MRNUMBER = {3868001},
MRREVIEWER = {Georgios\ Dimitroglou Rizell},
       DOI = {10.1090/jams/909},
       URL = {https://doi.org/10.1090/jams/909},
}

@article{FOOOII,
  title={Lagrangian Floer theory on compact toric manifolds II: bulk deformations},
  author={Kenji Fukaya and Yong-Geun Oh and Hiroshi Ohta and Kaoru Ono},
  journal={Selecta Mathematica},
  year={2008},
  volume={17},
  pages={609-711}
}

@book{FOOOIII,
  title={Lagrangian Floer Theory and Mirror Symmetry on Compact Toric Manifolds},
  author={Kenji Fukaya and Yong-Geun Oh and Hiroshi Ohta and Kaoru Ono},
  series={Asterisque Series},
  year={2016},
  publisher={Soci{\'e}t{\'e} Math{\'e}matique de France}
}

@article {PabiniakDisplacing,
    AUTHOR = {Pabiniak, Milena},
     TITLE = {Displacing ({L}agrangian) submanifolds in the manifolds of
              full flags},
   JOURNAL = {Adv. Geom.},
  FJOURNAL = {Advances in Geometry},
    VOLUME = {15},
      YEAR = {2015},
    NUMBER = {1},
     PAGES = {101--108},
      ISSN = {1615-715X,1615-7168},
   MRCLASS = {53D12 (53D35)},
  MRNUMBER = {3300713},
MRREVIEWER = {Goo\ Ishikawa},
       DOI = {10.1515/advgeom-2014-0025},
       URL = {https://doi.org/10.1515/advgeom-2014-0025},
}

@article {GonzalezWoodward,
    AUTHOR = {Gonz\'alez, Eduardo and Woodward, Chris T.},
     TITLE = {Quantum cohomology and toric minimal model programs},
   JOURNAL = {Adv. Math.},
  FJOURNAL = {Advances in Mathematics},
    VOLUME = {353},
      YEAR = {2019},
     PAGES = {591--646},
      ISSN = {0001-8708,1090-2082},
   MRCLASS = {14L24 (14M25 81T70)},
  MRNUMBER = {3986375},
MRREVIEWER = {Andrew\ Kresch},
       DOI = {10.1016/j.aim.2019.07.004},
       URL = {https://doi.org/10.1016/j.aim.2019.07.004},
}

@book{FOOObook2,
  title={Lagrangian Intersection Floer Theory: Anomaly and Obstruction, Part II},
 author = {Kenji Fukaya and Yong-Geun Oh and Hiroshi Ohta and Kaoru Ono},
  %isbn={9780821852507},
  series={AMS/IP studies in advanced mathematics},
  year={2010},
  publisher={American Mathematical Society}
}

@book{fooobook1,
  title={Lagrangian Intersection Floer Theory: Anomaly and Obstruction},
 author = {Kenji Fukaya and Yong-Geun Oh and Hiroshi Ohta and Kaoru Ono},
  number={pt. 1},
  isbn={9780821848364},
  lccn={2009025925},
  series={AMS/IP studies in advanced mathematics},
  % url={https://books.google.com/books?id=9Zofve4n8GEC},
  year={2009},
  publisher={American Mathematical Society}
}

@misc{VWSplit,
      title={Splitting the diagonal for broken maps}, 
      author={Sushmita Venugopalan and Chris Woodward},
      year={2025},
      eprint={2504.15583},
      archivePrefix={arXiv},
      primaryClass={math.SG},
      url={https://arxiv.org/abs/2504.15583}, note={arXiv: 2504.15583}
}

@article {BEHWZ,
    AUTHOR = {Bourgeois, F. and Eliashberg, Y. and Hofer, H. and Wysocki, K.
              and Zehnder, E.},
     TITLE = {Compactness results in symplectic field theory},
   JOURNAL = {Geom. Topol.},
  FJOURNAL = {Geometry and Topology},
    VOLUME = {7},
      YEAR = {2003},
     PAGES = {799--888},
      ISSN = {1465-3060,1364-0380},
   MRCLASS = {53D45 (53D35 53D40 57R17)},
  MRNUMBER = {2026549},
MRREVIEWER = {Kai\ Cieliebak},
       DOI = {10.2140/gt.2003.7.799},
       URL = {https://doi.org/10.2140/gt.2003@misc{venugopalan2023fukayacategoriesblowups,
      title={Fukaya categories of blowups}, 
      author={Sushmita Venugopalan and Chris T. Woodward and Guangbo Xu},
      year={2023},
      eprint={2006.12264},
      archivePrefix={arXiv},
      primaryClass={math.SG},
      url={https://arxiv.org/abs/2006.12264}, 
}.7.799},
}

@misc{AFOOO,
      title={Quantum cohomology
and split generation
in Lagrangian Floer theory}, 
      author={Mohammed Abouzaid and Kenji Fukaya and Yong-Geun Oh and Hiroshi Ohta and Kaoru Ono},
      year={in preparation}
}

@article {IonelParkerRelativeGW,
    AUTHOR = {Ionel, Eleny-Nicoleta and Parker, Thomas H.},
     TITLE = {Relative {G}romov-{W}itten invariants},
   JOURNAL = {Ann. of Math. (2)},
  FJOURNAL = {Annals of Mathematics. Second Series},
    VOLUME = {157},
      YEAR = {2003},
    NUMBER = {1},
     PAGES = {45--96},
      ISSN = {0003-486X,1939-8980},
   MRCLASS = {53D45 (57R17 57R57)},
  MRNUMBER = {1954264},
MRREVIEWER = {Ignasi\ Mundet-Riera},
       DOI = {10.4007/annals.2003.157.45},
       URL = {https://doi.org/10.4007/annals.2003.157.45},
}

@article {AbouzaidGeneration,
    AUTHOR = {Abouzaid, Mohammed},
     TITLE = {A geometric criterion for generating the {F}ukaya category},
   JOURNAL = {Publ. Math. Inst. Hautes \'Etudes Sci.},
  FJOURNAL = {Publications Math\'ematiques. Institut de Hautes \'Etudes
              Scientifiques},
    NUMBER = {112},
      YEAR = {2010},
     PAGES = {191--240},
      ISSN = {0073-8301,1618-1913},
   MRCLASS = {53D37},
  MRNUMBER = {2737980},
MRREVIEWER = {Timothy\ Perutz},
       DOI = {10.1007/s10240-010-0028-5},
       URL = {https://doi.org/10.1007/s10240-010-0028-5},
}

@article {SheridanFanoHypersurface,
    AUTHOR = {Sheridan, Nick},
     TITLE = {On the {F}ukaya category of a {F}ano hypersurface in
              projective space},
   JOURNAL = {Publ. Math. Inst. Hautes \'Etudes Sci.},
  FJOURNAL = {Publications Math\'ematiques. Institut de Hautes \'Etudes
              Scientifiques},
    VOLUME = {124},
      YEAR = {2016},
     PAGES = {165--317},
      ISSN = {0073-8301,1618-1913},
   MRCLASS = {53D37 (14F05 14J33 14N35)},
  MRNUMBER = {3578916},
MRREVIEWER = {Christian\ Lehn},
       DOI = {10.1007/s10240-016-0082-8},
       URL = {https://doi.org/10.1007/s10240-016-0082-8},
}

@misc{VWXBlowups,
      title={Fukaya categories of blowups}, 
      author={Sushmita Venugopalan and Chris T. Woodward and Guangbo Xu},
      year={2023},
      eprint={2006.12264},
      archivePrefix={arXiv},
      primaryClass={math.SG},
      url={https://arxiv.org/abs/2006.12264},   note={arXiv: 2006.12264},

}

@article{Kassel1986,
author = {Kassel, Christian},
journal = {Mathematische Annalen},
pages = {683-699},
title = {A Künneth Formula for the Cyclic Cohomology of Z/2-Graded Algebras.},
url = {http://eudml.org/doc/164178},
volume = {275},
year = {1986},
}

@book {Seidelbook,
    AUTHOR = {Seidel, Paul},
     TITLE = {Fukaya categories and {P}icard-{L}efschetz theory},
    SERIES = {Zurich Lectures in Advanced Mathematics},
 PUBLISHER = {European Mathematical Society (EMS), Z\"urich},
      YEAR = {2008},
     PAGES = {viii+326},
      ISBN = {978-3-03719-063-0},
   MRCLASS = {53D40 (16E45 32Q65 53D12)},
  MRNUMBER = {2441780},
MRREVIEWER = {Timothy\ Perutz},
       DOI = {10.4171/063},
       URL = {https://doi.org/10.4171/063},
}

@book {HelmstetterMicaliClifford,
    AUTHOR = {Helmstetter, Jacques and Micali, Artibano},
     TITLE = {Quadratic mappings and {C}lifford algebras},
 PUBLISHER = {Birkh\"auser Verlag, Basel},
      YEAR = {2008},
     PAGES = {xiv+504},
      ISBN = {978-3-7643-8605-4},
   MRCLASS = {15A66 (11E88 16H05 16W50)},
  MRNUMBER = {2408410},
MRREVIEWER = {A.\ R.\ Wadsworth},
}

@article {Tehrani,
    AUTHOR = {Farajzadeh Tehrani, Mohammad},
     TITLE = {Open {G}romov-{W}itten theory on symplectic manifolds and
              symplectic cutting},
   JOURNAL = {Adv. Math.},
  FJOURNAL = {Advances in Mathematics},
    VOLUME = {232},
      YEAR = {2013},
     PAGES = {238--270},
      ISSN = {0001-8708,1090-2082},
   MRCLASS = {53D45},
  MRNUMBER = {2989982},
MRREVIEWER = {Eduardo\ A.\ Gonzalez},
       DOI = {10.1016/j.aim.2012.09.015},
       URL = {https://doi-org.proxy.libraries.rutgers.edu/10.1016/j.aim.2012.09.015},
}

@article {NoharaUedaNonTorusFibers,
    AUTHOR = {Nohara, Yuichi and Ueda, Kazushi},
     TITLE = {Floer cohomologies of non-torus fibers of the
              {G}elfand-{C}etlin system},
   JOURNAL = {J. Symplectic Geom.},
  FJOURNAL = {The Journal of Symplectic Geometry},
    VOLUME = {14},
      YEAR = {2016},
    NUMBER = {4},
     PAGES = {1251--1293},
      ISSN = {1527-5256,1540-2347},
   MRCLASS = {53D40 (14M15)},
  MRNUMBER = {3601889},
MRREVIEWER = {Ignasi\ Mundet-Riera},
       DOI = {10.4310/JSG.2016.v14.n4.a9},
       URL = {https://doi-org.proxy.libraries.rutgers.edu/10.4310/JSG.2016.v14.n4.a9},
}

@article{CO,
author = {Cheol-Hyun Cho and Yong-Geun Oh},
title = {{Floer cohomology and disc instantons of Lagrangian torus fibers in Fano toric manifolds}},
volume = {10},
journal = {Asian Journal of Mathematics},
number = {4},
publisher = {International Press of Boston},
pages = {773 -- 814},
year = {2006},
}

@misc{xiao2023equivariantlagrangianfloertheory,
      title={Equivariant Lagrangian Floer theory on compact toric manifolds}, 
      author={Yao Xiao},
      year={2023},
      eprint={2310.20202},
      archivePrefix={arXiv},
      primaryClass={math.SG},
      url={https://arxiv.org/abs/2310.20202}, 
      note={arXiv: 2310.20202}
}

@book {MR4800139,
    AUTHOR = {Xiao, Yao},
     TITLE = {Equivariant {L}agrangian {F}loer {T}heory on {C}ompact {T}oric
              {M}anifolds},
      NOTE = {Thesis (Ph.D.)--State University of New York at Stony Brook},
 PUBLISHER = {ProQuest LLC, Ann Arbor, MI},
      YEAR = {2024},
     PAGES = {176},
      ISBN = {979-8383-61359-7},
   MRCLASS = {99-05},
  MRNUMBER = {4800139},
       URL =
              {https://gateway-proquest-com.proxy.libraries.rutgers.edu/openurl?url_ver=Z39.88-2004&rft_val_fmt=info:ofi/fmt:kev:mtx:dissertation&res_dat=xri:pqm&rft_dat=xri:pqdiss:31297740},
}

@article{xiao2024moment,
  title={Moment Lagrangian correspondences are unobstructed after bulk deformation},
  author={Xiao, Yao},
    eprint={2405.11169}, 
  year={2024}, 
  note={arXiv:2405.11169}
}

@article {KnopClassification,
    AUTHOR = {Knop, Friedrich},
     TITLE = {Automorphisms of multiplicity free {H}amiltonian manifolds},
   JOURNAL = {J. Amer. Math. Soc.},
  FJOURNAL = {Journal of the American Mathematical Society},
    VOLUME = {24},
      YEAR = {2011},
    NUMBER = {2},
     PAGES = {567--601},
      ISSN = {0894-0347,1088-6834},
   MRCLASS = {53D20 (14L30 14M27)},
  MRNUMBER = {2748401},
MRREVIEWER = {Eduardo\ A.\ Gonzalez},
       DOI = {10.1090/S0894-0347-2010-00686-8},
       URL = {https://doi-org.proxy.libraries.rutgers.edu/10.1090/S0894-0347-2010-00686-8},
}
\Addresses

\end{document}